\documentclass[a4paper,12pt]{article}

\usepackage{amsmath,amssymb,amsthm,amscd}

\newcommand{\Fg}{\mathfrak{g}}
\newcommand{\Fh}{\mathfrak{h}}
\newcommand{\Fsl}{\mathfrak{sl}}

\newcommand{\CB}{\mathcal{B}}

\newcommand{\BC}{\mathbb{C}}
\newcommand{\BR}{\mathbb{R}}

\newcommand{\BZ}{\mathbb{Z}}

\newcommand{\mv}{\mathcal{MV}}
\newcommand{\bz}{\mathcal{B\hspace{-1.5pt}Z}}

\newcommand{\Gr}{\mathcal{G}r}

\newcommand{\Hom}{\mathop{\rm Hom}\nolimits}
\newcommand{\wt}{\mathop{\rm wt}\nolimits}
\newcommand{\Wt}{\mathop{\rm Wt}\nolimits}
\newcommand{\conv}{\mathop{\rm Conv}\nolimits}
\newcommand{\res}{\mathop{\rm res}\nolimits}
\newcommand{\Int}{\mathop{\rm Int}\nolimits}

\newcommand{\pair}[2]{\langle #1,\,#2 \rangle}

\newcommand{\ve}{\varepsilon}
\newcommand{\vp}{\varphi}
\newcommand{\vpi}{\varpi}

\newcommand{\ha}[1]{\widehat{#1}}
\newcommand{\ol}[1]{\overline{#1}}

\newcommand{\bzero}{{\bf 0}}

\newcommand{\bM}{{\bf M}}
\newcommand{\bO}{{\bf O}}

\newcommand{\bqed}{\quad \hbox{\rule[-0.5pt]{3pt}{8pt}}}

\makeatletter
\@addtoreset{equation}{subsection}

\renewcommand\section{\@startsection{section}{1}{0pt}
{-3.5ex plus -1ex minus -.2ex}{1.0ex plus .2ex}{\large\bf}}
\renewcommand\subsection{\@startsection{subsection}{1}{0pt}
{2.5ex plus 1ex minus .2ex}{-1em}{\bf}}
\makeatother

\newcommand{\vsp}{\vspace{3mm}}

\theoremstyle{plain}
\newtheorem{thm}{Theorem}[subsection]
\newtheorem{lem}[thm]{Lemma}
\newtheorem{prop}[thm]{Proposition}

\newtheorem{fact}[thm]{Fact}

\newtheorem{claim}{Claim}[thm]

\newtheorem*{claim*}{Claim}

\theoremstyle{definition}
\newtheorem{dfn}[thm]{Definition}

\theoremstyle{remark}
\newtheorem{rem}[thm]{Remark}
\newtheorem{ex}[thm]{Example}

\begin{document}

\setlength{\baselineskip}{18pt}

\title{\Large\bf Toward Berenstein-Zelevinsky data in affine type $A$ \\[3mm]
I: Construction of affine analogs}
\author{
 Satoshi Naito \\ 
 \small Institute of Mathematics, University of Tsukuba, \\
 \small Tsukuba, Ibaraki 305-8571, Japan \ 
 (e-mail: {\tt naito@math.tsukuba.ac.jp}) \\[5mm]
 Daisuke Sagaki \\ 
 \small Institute of Mathematics, University of Tsukuba, \\
 \small Tsukuba, Ibaraki 305-8571, Japan \ 
 (e-mail: {\tt sagaki@math.tsukuba.ac.jp})
 \\[2mm] and \\[2mm]
 Yoshihisa Saito \\ 
 \small Graduate School of Mathematical Sciences, University of Tokyo, \\
 \small Meguro-ku, Tokyo 153-8914, Japan \ 
 (e-mail: {\tt yosihisa@ms.u-tokyo.ac.jp})
}
\date{}
\maketitle

%=======================%
%     START ABSTRACT    %
%=======================%
%
\begin{abstract} \setlength{\baselineskip}{16pt}
We give (conjectural) analogs of Berenstein-Zelevinsky data 
for affine type $A$. Moreover, by using these affine analogs of 
Berenstein-Zelevinsky data, we realize the crystal basis of the 
negative part of the quantized universal enveloping algebra of the 
(Langlands dual) Lie algebra of affine type $A$.
\end{abstract}
%
%=========================%
%     START SECTION 01    %
%=========================%
%
\section{Introduction.}
\label{sec:intro}
This paper provides the first step in our attempt 
to construct and describe analogs of Mirkovi\'{c}-Vilonen 
(MV for short) polytopes for affine Lie algebras. 
In this paper, we concentrate on the case of affine type $A$, 
and construct (conjectural) affine analogs of Berenstein-Zelevinsky 
(BZ for short) data. Furthermore, using these affine analogs of BZ data, 
we give a realization of the crystal basis of the negative part 
of the quantized universal enveloping algebra associated to (the 
Langlands dual Lie algebra of) the affine Lie algebra of affine type $A$.
Here we should mention that in the course of the much more sophisticated 
discussion toward the (conjectural) geometric Satake correspondence 
for a Kac-Moody group of affine type $A$, Nakajima \cite{Nak} constructed 
affine analogs of MV cycles by using his quiver varieties; see also 
\cite{BF1}, \cite{BF2}. 

Let $G$ be a semisimple algebraic group over $\BC$ with 
(semisimple) Lie algebra $\Fg$. Anderson \cite{A} introduced 
MV polytopes for $\Fg$ as moment polytopes of MV cycles in the 
affine Grassmannian $\Gr$ associated to $G$, and, on the basis of 
the geometric Satake correspondence, used them to count weight 
multiplicities and tensor product multiplicities for finite-dimensional 
irreducible representations of the Langlands dual group $G^{\vee}$ of $G$.

Soon afterward, Kamnitzer \cite{Kam1}, \cite{Kam2} gave a combinatorial 
characterization of MV polytopes in terms of BZ data; a BZ datum is 
a collection of integers (indexed by the set of chamber weights) 
satisfying the edge inequalities and tropical Pl\"ucker relations. 
To be more precise, let $W_{I}$ be the Weyl group of $\Fg$, and 
$\vpi_{i}^{I}$, $i \in I$, the fundamental weights, where $I$ is 
the index set of simple roots; the set $\Gamma_{I}$ of chamber weights 
is by definition $\Gamma_{I}:=\bigcup_{i \in I} W_{I}\vpi_{i}^{I}$. 
Then, for a BZ datum $\bM=(M_{\gamma})_{\gamma \in \Gamma_{I}}$ with 
$M_{\gamma} \in \BZ$, the corresponding MV polytope $P(\bM)$ is given by:
\begin{equation*}
P(\bM)=\bigl\{
 h \in (\Fh_{I})_{\BR} \mid 
 \text{$\pair{h}{\gamma} \ge M_{\gamma}$ for all $\gamma \in \Gamma_{I}$}
\bigr\},
\end{equation*}
where $(\Fh_{I})_{\BR}$ is a real form of the Cartan subalgebra $\Fh_{I}$ 
of $\Fg$, and $\pair{\cdot\,}{\cdot}$ is the canonical pairing between 
$\Fh_{I}$ and $\Fh_{I}^{\ast}$. We denote by $\bz_{I}$ the set of all BZ 
data $\bM=(M_{\gamma})_{\gamma \in \Gamma_{I}}$ such that 
$M_{w_{0}^{I}\vpi_{i}^{I}}=0$ for all $i \in I$, 
where $w_{0}^{I} \in W_{I}$ is the longest element.

Now, let $\ha{\Fg}$ denote the affine Lie algebra of type $A_{\ell}^{(1)}$ 
over $\BC$ with Cartan subalgebra $\ha{\Fh}$, and 
$\ha{A}=(\ha{a}_{ij})_{i,\,j \in \ha{I}}$ its Cartan matrix
with index set $\ha{I}=\bigl\{0,\,1,\,\dots,\,\ell\bigr\}$, 
where $\ell \in \BZ_{\ge 2}$ is a fixed integer. 
Before constructing the set of (conjectural) analogs of BZ data 
for the affine Lie algebra $\ha{\Fg}$, we need to construct the set 
$\bz_{\BZ}$ of BZ data of type $A_{\infty}$.

Let $\Fsl_{\infty}(\BC)$ denote the infinite rank Lie algebra over $\BC$ 
of type $A_{\infty}$ with Cartan subalgebra $\Fh$, and 
$A_{\BZ}=(a_{ij})_{i,\,j \in \BZ}$ its Cartan matrix with index set $\BZ$. 
Let $W_{\BZ}=\langle s_{i} \mid i \in \BZ \rangle \subset GL(\Fh^{\ast})$ 
be the Weyl group of $\Fsl_{\infty}(\BC)$, and 
$\Lambda_{i} \in \Fh^{\ast}$, $i \in \BZ$, the fundamental weights; 
the set $\Gamma_{\BZ}$ of chamber weights for $\Fsl_{\infty}(\BC)$ 
is defined to be the set
\begin{equation*}
\Gamma_{\BZ}:=\bigcup_{i \in \BZ} \bigl(-W_{\BZ}\Lambda_{i}\bigr)
 = \bigl\{-w\Lambda_{i} \mid w \in W_{\BZ},\,i \in \BZ\bigr\},
\end{equation*}
not to be the set $\bigcup_{i \in \BZ} W_{\BZ}\Lambda_{i}$. 
Then, for each finite interval $I$ in $\BZ$, we can (and do) 
identify the set $\Gamma_{I}$ of chamber weights for the 
finite-dimensional simple Lie algebra $\Fg_{I}$ over $\BC$ of 
type $A_{|I|}$ with the subset 
$\bigl\{-w\Lambda_{i} \mid w \in W_{I},\,i \in I\bigr\}$, 
where $|I|$ denotes the cardinality of $I$, and 
$W_{I}=\langle s_{i} \mid i \in I \rangle \subset W_{\BZ}$ 
is the Weyl group of $\Fg_{I}$ (see \S\ref{subsec:notation-inf} for details).
Here we note that the family 
$\bigl\{\bz_{I} \mid \text{$I$ is a finite interval in $\BZ$}\bigr\}$ 
forms a projective system (cf. Lemma~\ref{lem:res1}). 

Using the projective system 
$\bigl\{\bz_{I} \mid \text{$I$ is a finite interval in $\BZ$}\bigr\}$ above, 
we define the set $\bz_{\BZ}$ of BZ data of type $A_{\infty}$ 
to be a kind of projective limit, with a certain stability constraint, of the system 
$\bigl\{\bz_{I} \mid \text{$I$ is a finite interval in $\BZ$}\bigr\}$; 
see Definition~\ref{dfn:BZdatum2} for a precise statement. 
Because of this stability constraint, we can endow 
the set $\bz_{\BZ}$ a crystal structure for the Lie algebra 
$\Fsl_{\infty}(\BC)$ of type $A_{\infty}$. 

Finally, recall the fact that the Dynkin diagram of type $A_{\ell}^{(1)}$ 
can be obtained from that of type $A_{\infty}$ by the operation of ``folding''
under the Dynkin diagram automorphism $\sigma:\BZ \rightarrow \BZ$ in type $A_{\infty}$ 
given by: $\sigma(i)=i+\ell-1$ for $i \in \BZ$, where $\ell \in \BZ_{\ge 2}$. 
In view of this fact, we consider the fixed point subset $\bz_{\BZ}^{\sigma}$ 
of $\bz_{\BZ}$ under a natural action of 
the Dynkin diagram automorphism $\sigma:\BZ \rightarrow \BZ$. 
Then, we can endow a crystal structure (canonically induced by 
that on $\bz_{\BZ}$) for the quantized universal enveloping algebra 
$U_{q}(\ha{\Fg}^{\vee})$ associated to the (Langlands) dual Lie algebra 
$\ha{\Fg}^{\vee}$ of $\ha{\Fg}$. 

However, the crystal $\bz_{\BZ}^{\sigma}$ for $U_{q}(\ha{\Fg}^{\vee})$ 
may be too big for our purpose. Therefore, we restrict our attention to 
the connected component $\bz_{\BZ}^{\sigma}(\bO)$ of the crystal 
$\bz_{\BZ}^{\sigma}$ containing the BZ datum $\bO$ of type $A_{\infty}$ 
whose $\gamma$-component is equal to $0$ for each $\gamma \in \Gamma_{\BZ}$. 
Our main result (Theorem~\ref{thm:main}) states that the crystal 
$\bz_{\BZ}^{\sigma}(\bO)$ is isomorphic, as a crystal 
for $U_{q}(\ha{\Fg}^{\vee})$, to the crystal basis $\ha{\CB}(\infty)$
of the negative part $U_{q}^{-}(\ha{\Fg}^{\vee})$ of $U_{q}(\ha{\Fg}^{\vee})$. 
Moreover, for each dominant integral weight $\ha{\lambda} \in \ha{\Fh}$ 
for $\ha{\Fg}^{\vee}$, the crystal basis $\ha{\CB}(\ha{\lambda})$ of 
the irreducible highest weight $U_{q}(\ha{\Fg}^{\vee})$-module of 
highest weight $\ha{\lambda}$ can be realized 
as a certain explicit subset of $\bz_{\BZ}^{\sigma}(\bO)$ 
(see Theorem~\ref{thm:main2}). In fact, we first prove Theorem~\ref{thm:main2}
by using Stembridge's result on a characterization of highest weight crystals
for simply-laced Kac-Moody algebras; then, Theorem~\ref{thm:main} is 
obtained as a corollary. 

Unfortunately, we have not yet found an explicit characterization of 
the connected component $\bz_{\BZ}^{\sigma}(\bO) \subset \bz_{\BZ}^{\sigma}$
in terms of the ``edge inequalities'' and ``tropical Pl\"ucker relations'' 
in type $A_{\ell}^{(1)}$ in a way analogous to the finite-dimensional case; 
we hope to mention such a description of the connected component 
$\bz_{\BZ}^{\sigma}(\bO) \subset \bz_{\BZ}^{\sigma}$ in our 
forthcoming paper \cite{NSS}.
However, from our results in this paper, it seems reasonable to 
think of an element $\bM=(M_{\gamma})_{\gamma \in \Gamma_{\BZ}}$ of 
the crystal $\bz_{\BZ}^{\sigma}(\bO)$ as a (conjectural) analog of 
a BZ datum in affine type $A$.

This paper is organized as follows. In Section~\ref{sec:BZdatum}, 
following Kamnitzer, we review some standard facts about BZ data 
for the simple Lie algebra $\Fg_{I}$ of type $A_{|I|}$, where 
$I \subset \BZ$ is the index set of simple roots with cardinality $m$,
and then show that the system of sets $\bz_{I}$ of BZ data for $\Fg_{I}$, 
where $I$ runs over all the finite intervals in $\BZ$, forms 
a projective system. In Section~\ref{sec:BZdatum-inf}, we introduce 
the notion of BZ data of type $A_{\infty}$, and define Kashiwara 
operators on the set $\bz_{\BZ}$ of BZ data of type $A_{\infty}$. 
Also, we show a technical lemma about some properties of Kashiwara 
operators on $\bz_{\BZ}$. In Section~\ref{sec:BZdatum-aff}, we first
study the action of the Dynkin diagram automorphism $\sigma$ in 
type $A_{\infty}$ on the set $\bz_{\BZ}$. Next, we define the set of 
BZ data of type $A^{(1)}_{\ell}$ to be the fixed point subset $\bz_{\BZ}^{\sigma}$ 
of $\bz_{\BZ}$ under $\sigma$, and endow a canonical crystal structure on it. 
Finally, in Subsections~\ref{subsec:main} and \ref{subsec:proofs}, we state 
and prove our main results (Theorems~\ref{thm:main} and \ref{thm:main2}), 
which give a realization of the crystal basis $\ha{\CB}(\infty)$ for the 
(Langlands dual) Lie algebra $\ha{\Fg}^{\vee}$ of type $A^{(1)}_{\ell}$. 
In the Appendix, we restate Stembridge's result on a characterization of 
simply-laced crystals in a form that will be used in the proofs of the
theorems above. 

%=========================%
%     START SECTION 02    %
%=========================%
%
\section{Berenstein-Zelevinsky data of type $A_{m}$.}
\label{sec:BZdatum}

In this section, 
following \cite{Kam1} and \cite{Kam2}, 
we briefly review some basic facts about 
Berenstein-Zelevinsky (BZ for short) data for 
the complex finite-dimensional simple Lie algebra 
of type $A_{m}$. 

%==============================%
%     START SUBSECTION 0201    %
%==============================%
%
\subsection{Basic notation in type $A_{m}$.}
\label{subsec:notation}

Let $I$ be a fixed (finite) interval in $\BZ$ 
whose cardinality is equal to $m \in \BZ_{\ge 1}$; 
that is, $I \subset \BZ$ is a finite subset of the form: 
%
%%%%%%%%%%%%%%%%%%%
%%% eq:interval %%%
%%%%%%%%%%%%%%%%%%%
%
\begin{equation} \label{eq:interval}
I=\bigl\{ n+1,\,n+2,\,\dots,\,n+m \bigr\}
\quad \text{for some $n \in \BZ$}. 
\end{equation}
Let $A_{I}=(a_{ij})_{i,j \in I}$ denote the Cartan matrix of 
type $A_{m}$ with index set $I$; 
the entries $a_{ij}$ are given by: 
%
%%%%%%%%%%%%%%%%%
%%% eq:cartan %%%
%%%%%%%%%%%%%%%%%
%
\begin{equation} \label{eq:cartan}
a_{ij}=
 \begin{cases}
 2 & \text{if $i=j$}, \\
 -1 & \text{if $|i-j|=1$}, \\
 0 & \text{otherwise},
 \end{cases}
\end{equation}
for $i,\,j \in I$. Let $\Fg_{I}$ be 
the complex finite-dimensional simple Lie algebra 
with Cartan matrix $A_{I}$, Cartan subalgebra $\Fh_{I}$, 
simple coroots $h_{i} \in \Fh_{I}$, $i \in I$, and 
simple roots $\alpha_{i} \in \Fh_{I}^{\ast}:=
\Hom_{\BC}(\Fh_{I},\,\BC)$, $i \in I$; note that 
$\Fh_{I}=\bigoplus_{i \in I} \BC h_{i}$, and that 
$\pair{h_{i}}{\alpha_{j}}=a_{ij}$ for $i,\,j \in I$, 
where $\pair{\cdot\,}{\cdot}$ is the canonical pairing 
between $\Fh_{I}$ and $\Fh_{I}^{\ast}$. 
Denote by $\vpi^{I}_{i} \in \Fh_{I}^{\ast}$, $i \in I$, 
the fundamental weights for $\Fg_{I}$, and by 
$W_{I}:=\langle s_{i} \mid i \in I \rangle \ (\subset GL(\Fh_{I}^{\ast}))$ 
the Weyl group of $\Fg_{I}$, 
where $s_{i}$ is the simple reflection 
for $i \in I$, with $e$ and $w_{0}^{I}$ 
the identity element and the longest element of 
the Weyl group $W_{I}$, respectively. 
Also, we denote by $\le$ the (strong) 
Bruhat order on $W_{I}$. 
The (Dynkin) diagram automorphism for $\Fg_{I}$ is a bijection 
$\omega_{I}:I \rightarrow I$ defined by: 
$\omega_{I}(n+i)=n+m-i+1$ for $1 \le i \le m$ 
(see \eqref{eq:interval} and \eqref{eq:cartan}). 
It is easy to see that for $i \in I$, 
%
%%%%%%%%%%%%%
%%% eq:w0 %%%
%%%%%%%%%%%%%
%
\begin{equation} \label{eq:w0}
w_{0}^{I}(\alpha_{i})=-\alpha_{\omega_{I}(i)}, \qquad
w_{0}^{I}(\vpi^{I}_{i})=-\vpi^{I}_{\omega_{I}(i)}, \qquad
w_{0}^{I}s_{\omega_{I}(i)}=s_{i}w_{0}^{I}. 
\end{equation}

Let $\Fg_{I}^{\vee}$ denote the (Langlands) dual Lie algebra of $\Fg_{I}$; 
that is, $\Fg_{I}^{\vee}$ is the complex finite-dimensional 
simple Lie algebra of type $A_{m}$ associated to 
the transpose ${}^{t}\!A_{I}\, (=A_{I})$ of $A_{I}$, with 
Cartan subalgebra $\Fh_{I}^{\ast}$, 
simple coroots $\alpha_{i} \in \Fh_{I}^{\ast}$, $i \in I$, and 
simple roots $h_{i} \in \Fh_{I}$, $i \in I$. 
Let $U_{q}(\Fg_{I}^{\vee})$ be 
the quantized universal enveloping algebra 
over the field $\BC(q)$ of rational functions in $q$ 
associated to the Lie algebra $\Fg_{I}^{\vee}$, 
$U_{q}^{-}(\Fg_{I}^{\vee})$ the negative part of 
$U_{q}(\Fg_{I}^{\vee})$, and 
$\CB_{I}(\infty)$ the crystal basis of 
$U_{q}^{-}(\Fg_{I}^{\vee})$. 
Also, for a dominant integral weight 
$\lambda \in \Fh_{I}$ for $\Fg_{I}^{\vee}$, 
$\CB_{I}(\lambda)$ denotes the crystal basis of 
the finite-dimensional irreducible highest weight 
$U_{q}(\Fg_{I}^{\vee})$-module of highest weight $\lambda$. 

%==============================%
%     START SUBSECTION 0202    %
%==============================%
%
\subsection{BZ data of type $A_{m}$.}
\label{subsec:BZdatum}

We set
%
%%%%%%%%%%%%%%%%%
%%% eq:GammaI %%%
%%%%%%%%%%%%%%%%%
%
\begin{equation} \label{eq:GammaI}
\Gamma_{I}:=
\bigl\{w\vpi_{i}^{I} \mid w \in W_{I},\,i \in I\bigr\};
\end{equation}
note that by the second equation in \eqref{eq:w0}, 
the set $\Gamma_{I}$ (of chamber weights) coincides with the set 
$-\Gamma_{I}=
\bigl\{-w\vpi_{i}^{I} \mid w \in W_{I},\,i \in I\bigr\}$. 
Let $\bM=(M_{\gamma})_{\gamma \in \Gamma_{I}}$ be a collection 
of integers indexed by $\Gamma_{I}$. 
For each $\gamma \in \Gamma_{I}$, we call $M_{\gamma}$ 
the $\gamma$-component of the collection $\bM$, and 
denote it by $(\bM)_{\gamma}$. 
%
%%%%%%%%%%%%%%%%%%%
%%% dfn:BZdatum %%%
%%%%%%%%%%%%%%%%%%%
%
\begin{dfn} \label{dfn:BZdatum}
A collection 
$\bM=(M_{\gamma})_{\gamma \in \Gamma_{I}}$ 
of integers is 
called a Berenstein-Zelevinsky (BZ for short) datum 
for $\Fg_{I}$ if it satisfies the following 
conditions (1) and (2): 

(1) (edge inequalities) 
for all $w \in W_{I}$ and $i \in I$, 
%
%%%%%%%%%%%%%%%
%%% eq:edge %%%
%%%%%%%%%%%%%%%
%
\begin{equation} \label{eq:edge}
M_{w\vpi^{I}_{i}}+M_{ws_{i}\vpi^{I}_{i}}+
\sum_{j \in I \setminus \{i\}} a_{ji} M_{w\vpi^{I}_{j}} \le 0;
\end{equation}

(2) (tropical Pl\"ucker relations) 
for all $w \in W_{I}$ and $i,\,j \in I$ with 
$a_{ij}=a_{ji}=-1$ such that $ws_{i} > w$, $ws_{j} > w$, 
%
%%%%%%%%%%%%%%
%%% eq:TPR %%%
%%%%%%%%%%%%%%
%
\begin{equation} \label{eq:TPR}
M_{ws_{i}\vpi^{I}_{i}}+M_{ws_{j}\vpi^{I}_{j}} =
\min \bigl(
 M_{w\vpi^{I}_{i}}+M_{ws_{i}s_{j}\vpi^{I}_{j}}, \ 
 M_{w\vpi^{I}_{j}}+M_{ws_{j}s_{i}\vpi^{I}_{i}}
 \bigr).
\end{equation}
\end{dfn}

%==============================%
%     START SUBSECTION 0203    %
%==============================%
%
\subsection{Crystal structure on the set of BZ data of type $A_{m}$.}
\label{subsec:crystal}

Let $\bM=(M_{\gamma})_{\gamma \in \Gamma_{I}}$ 
be a BZ datum for $\Fg_{I}$. Following \cite[\S2.3]{Kam1}, 
we define 
\begin{equation*}
P(\bM):=
 \bigl\{h \in (\Fh_{I})_{\BR} \mid 
   \pair{h}{\gamma} \ge M_{\gamma} \ 
   \text{for all } \gamma \in \Gamma_{I}
 \bigr\},
\end{equation*}
where $(\Fh_{I})_{\BR}:=\bigoplus_{i \in I} \BR h_{i}$ 
is a real form of the Cartan subalgebra $\Fh_{I}$. We know from 
\cite[Proposition~2.2]{Kam1} that 
$P(\bM)$ is a convex polytope in $(\Fh_{I})_{\BR}$ 
whose set of vertices is given by: 
%
%%%%%%%%%%%%%%%%%
%%% eq:vertex %%%
%%%%%%%%%%%%%%%%%
%
\begin{equation} \label{eq:vertex}
\left\{
 \mu_{w}(\bM):=\sum_{i \in I} M_{w\vpi^{I}_{i}} \, w h_{i}
\ \Biggm| \ w \in W
\right\} \subset (\Fh_{I})_{\BR}.
\end{equation}
The polytope $P(\bM)$ is called 
a Mirkovi\'c-Vilonen (MV) polytope 
associated to the BZ datum 
$\bM=(M_{\gamma})_{\gamma \in \Gamma_{I}}$. 

We denote by $\bz_{I}$ the set of all BZ data 
$\bM=(M_{\gamma})_{\gamma \in \Gamma_{I}}$ for $\Fg_{I}$ 
satisfying the condition that 
$M_{w_{0}^{I}\vpi^{I}_{i}}=0$ for all $i \in I$, 
or equivalently, $M_{-\vpi^{I}_{i}}=0$ for all 
$i \in I$ (by the second equation in \eqref{eq:w0}). 
By \cite[\S3.3]{Kam2}, the set $\mv_{I}:=
\bigl\{P(\bM) \mid \bM \in \bz_{I}\bigr\}$ can be endowed
with a crystal structure for $U_{q}(\Fg_{I}^{\vee})$, and 
the resulting crystal $\mv_{I}$ is isomorphic to 
the crystal basis $\CB_{I}(\infty)$ of the negative part 
$U_{q}^{-}(\Fg_{I}^{\vee})$ of $U_{q}(\Fg_{I}^{\vee})$. 
Because the map $\bz_{I} \rightarrow \mv_{I}$ 
defined by $\bM \mapsto P(\bM)$ is bijective, 
we can also endow the set $\bz_{I}$ with a crystal structure for 
$U_{q}(\Fg_{I}^{\vee})$ in such a way that 
the bijection $\bz_{I} \rightarrow \mv_{I}$ is 
an isomorphism of crystals for $U_{q}(\Fg_{I}^{\vee})$. 

Now we recall from \cite{Kam2} 
the description of the crystal structure on $\bz_{I}$. 
For $\bM=(M_{\gamma})_{\gamma \in \Gamma_{I}} \in \bz_{I}$, 
define the weight $\wt(\bM)$ of $\bM$ by: 
%
%%%%%%%%%%%%%
%%% eq:wt %%%
%%%%%%%%%%%%%
%
\begin{equation} \label{eq:wt}
\wt (\bM) = \sum_{i \in I} M_{\vpi^{I}_{i}} \, h_{i}. 
\end{equation}
The raising Kashiwara operators 
$e_{p}$, $p \in I$, on $\bz_{I}$ 
are defined as follows 
(see \cite[Theorem~3.5\,(ii)]{Kam2}). 
Fix $p \in I$. For a BZ datum 
$\bM=(M_{\gamma})_{\gamma \in \Gamma_{I}}$ for $\Fg_{I}$ 
(not necessarily an element of $\bz_{I}$), 
we set 
%
%%%%%%%%%%%%%%%%%%
%%% eq:dfn-vep %%%
%%%%%%%%%%%%%%%%%%
%
\begin{equation} \label{eq:dfn-vep}
\ve_{p}(\bM) := 
 - \left(
     M_{\vpi^{I}_{p}}+M_{s_{p}\vpi^{I}_{p}}
     +\sum_{q \in I \setminus \{p\}} 
      a_{qp} M_{\vpi^{I}_{q}}
    \right), 
\end{equation}
which is nonnegative by condition (1) of Definition~\ref{dfn:BZdatum}. 
Observe that $\mu_{s_{p}}(\bM)-\mu_{e}(\bM)=\ve_{p}(\bM)h_{p}$, 
and hence that $\mu_{s_{p}}(\bM)=\mu_{e}(\bM)$ 
if and only if $\ve_{p}(\bM)=0$. 
In view of this, 
we set $e_{p}\bM:=\bzero$ if $\ve_{p}(\bM)=0$ 
(cf. \cite[Theorem~3.5(ii)]{Kam2}), 
where $\bzero$ is an additional element, 
which is not contained in $\bz_{I}$.
We know the following fact from 
\cite[Theorem~3.5\,(ii)]{Kam2} 
(see also the comment after \cite[Theorem~3.5]{Kam2}). 
%
%%%%%%%%%%%%%%%
%%% fact:ej %%%
%%%%%%%%%%%%%%%
%
\begin{fact} \label{fact:ej}
Let $\bM=(M_{\gamma})_{\gamma \in \Gamma_{I}}$ 
be a BZ datum for $\Fg_{I}$ 
(not necessarily an element of $\bz_{I}$).
If $\ve_{p}(\bM) > 0$, then 
there exists a unique BZ datum for $\Fg_{I}$, 
denoted by $e_{p}\bM$, such that 
$(e_{p}\bM)_{\vpi^{I}_{p}}=M_{\vpi^{I}_{p}}+1$, and 
such that $(e_{p}\bM)_{\gamma}=M_{\gamma}$ for all 
$\gamma \in \Gamma_{I}$ with $\pair{h_{p}}{\gamma} \le 0$. 
\end{fact}

It is easily verified that 
if $\bM=(M_{\gamma})_{\gamma \in \Gamma_{I}} 
\in \bz_{I}$, then $e_{p}\bM \in \bz_{I} \cup \{\bzero\}$. 
Indeed, suppose that $\ve_{p}(\bM) > 0$, 
or equivalently, $e_{p}\bM \ne \bzero$. Let $i \in I$. 
Since $\pair{h_{p}}{w_{0}^{I}\vpi^{I}_{i}} \le 0$ 
by the second equation in \eqref{eq:w0}, 
it follows from the definition of $e_{p}\bM$ 
that $(e_{p}\bM)_{w_{0}^{I}\vpi^{I}_{i}}$ 
is equal to $M_{w_{0}^{I}\vpi^{I}_{i}}$, 
and hence that $(e_{p}\bM)_{w_{0}^{I}\vpi^{I}_{i}}=
M_{w_{0}^{I}\vpi^{I}_{i}}=0$. 
Thus, we obtain a map $e_{p}$ from $\bz_{I}$ to 
$\bz_{I} \cup \{\bzero\}$ sending $\bM \in \bz_{I}$ to 
$e_{p}\bM \in \bz_{I} \cup \{\bzero\}$. 
By convention, we set $e_{p}\bzero:=\bzero$. 

Similarly, the lowering Kashiwara operators 
$f_{p}$, $p \in I$, on $\bz_{I}$
are defined as follows. Fix $p \in I$. 
Let us recall the following fact from 
\cite[Theorem~3.5\,(i)]{Kam2}, 
the comment after \cite[Theorem~3.5]{Kam2}, and 
\cite[Corollary~5.6]{Kam2}. 
%
%%%%%%%%%%%%%%%
%%% fact:fj %%%
%%%%%%%%%%%%%%%
%
\begin{fact} \label{fact:fj}
Let $\bM=(M_{\gamma})_{\gamma \in \Gamma_{I}}$ 
be a BZ datum for $\Fg_{I}$ 
(not necessarily an element of $\bz_{I}$). 
Then, there exists a unique BZ datum for $\Fg_{I}$, 
denoted by $f_{p}\bM$, such that 
$(f_{p}\bM)_{\vpi^{I}_{p}}=M_{\vpi^{I}_{p}}-1$, and 
such that $(f_{p}\bM)_{\gamma}=M_{\gamma}$ 
for all $\gamma \in \Gamma_{I}$ 
with $\pair{h_{p}}{\gamma} \le 0$. 
Moreover, for each $\gamma \in \Gamma_{I}$, 
%
%%%%%%%%%%%%%
%%% eq:AM %%%
%%%%%%%%%%%%%
%
\begin{equation} \label{eq:AM}
(f_{p}\bM)_{\gamma}=
\begin{cases}
 \min \bigl(
   M_{\gamma},\ 
   M_{s_{p}\gamma}+c_{p}(\bM)
   \bigr)
   & \text{\rm if $\pair{h_{p}}{\gamma} > 0$}, \\[1.5mm]
 M_{\gamma} & \text{\rm otherwise}, 
\end{cases}
\end{equation}
where $c_{p}(\bM):=M_{\vpi^{I}_{p}}-M_{s_{p}\vpi^{I}_{p}}-1$.
\end{fact}
%
%%%%%%%%%%%%%%
%%% rem:AM %%%
%%%%%%%%%%%%%%
%
\begin{rem} \label{rem:AM}
Keep the notation and assumptions of Fact~\ref{fact:fj}. 
By \eqref{eq:AM}, we have 
$(f_{p}\bM)_{\gamma} \le M_{\gamma}$ 
for all $\gamma \in \Gamma_{I}$. 
\end{rem}

In exactly the same way as the case of $e_{p}$ above, 
we see that if $\bM \in \bz_{I}$, 
then $f_{p}\bM \in \bz_{I}$. 
Thus, we obtain a map $f_{p}$ from $\bz_{I}$ to itself  
sending $\bM \in \bz_{I}$ to $f_{p}\bM \in \bz_{I}$.
By convention, we set $f_{p}\bzero:=\bzero$. 

Finally, we set $\vp_{p}(\bM):=
 \pair{\wt(\bM)}{\alpha_{p}}+\ve_{p}(\bM)$ 
for $\bM \in \bz_{I}$ and $p \in I$. 
%
%%%%%%%%%%%%%%%%
%%% thm:binf %%%
%%%%%%%%%%%%%%%%
%
\begin{thm}[\cite{Kam2}] \label{thm:binf}
The set $\bz_{I}$, equipped with the maps 
$\wt$, $e_{p},\,f_{p} \ (p \in I)$, and 
$\ve_{p},\,\vp_{p} \ (p \in I)$ above, is 
a crystal for $U_{q}(\Fg_{I}^{\vee})$ 
isomorphic to the crystal basis $\CB_{I}(\infty)$ of 
the negative part $U_{q}^{-}(\Fg_{I}^{\vee})$ of $U_{q}(\Fg_{I}^{\vee})$. 
\end{thm}
%
%%%%%%%%%%%%%%%%%%%
%%% rem:highest %%%
%%%%%%%%%%%%%%%%%%%
%
\begin{rem} \label{rem:highest}
Let $\bO$ be the collection of integers 
indexed by $\Gamma_{I}$ whose $\gamma$-component is 
equal to $0$ for all $\gamma \in \Gamma_{I}$. 
It is obvious that $\bO$ is an element 
of $\bz_{I}$ whose weight is equal to $0$. 
Hence it follows from Theorem~\ref{thm:binf} that 
for each $\bM \in \bz_{I}$, there exists 
$p_{1},\,p_{2},\,\dots,\,p_{N} \in I$ such that 
$\bM=f_{p_{1}}f_{p_{2}} \cdots f_{p_{N}} \bO$. 
Therefore, using this fact and Remark~\ref{rem:AM}, 
we deduce that if 
$\bM=(M_{\gamma})_{\gamma \in \Gamma_{I}} \in \bz_{I}$, 
then $M_{\gamma} \in \BZ_{\le 0}$ 
for all $\gamma \in \Gamma_{I}$. 
\end{rem}

Let $\lambda \in \Fh_{I}$ be 
a dominant integral weight for $\Fg_{I}^{\vee}$. 
We define $\mv_{I}(\lambda)$ to be the set of those MV polytopes 
$P \in \mv_{I}$ such that $\lambda+P$ is contained in 
the convex hull $\conv (W_{I}\lambda)$ in $(\Fh_{I})_{\BR}$ 
of the $W_{I}$-orbit $W_{I}\lambda$ through $\lambda$. 
We see from \cite[\S3.2]{Kam2} that 
for $\bM=(M_{\gamma})_{\gamma \in \Gamma_{I}} \in \bz_{I}$, 
\begin{equation*}
\lambda+P(\bM)=
 \bigl\{h \in \Fh_{\BR} \mid 
   \pair{h}{\gamma} \ge M_{\gamma}' \ 
   \text{for all } \gamma \in \Gamma_{I}
 \bigr\},
\end{equation*}
where $M_{\gamma}':=M_{\gamma}+\pair{\lambda}{\gamma}$ 
for $\gamma \in \Gamma_{I}$. 
We know from \cite[Theorem~8.5]{Kam1} and \cite[\S6.2]{Kam2} that 
$\lambda+P(\bM) \subset \conv (W_{I}\lambda)$ 
if and only if $M_{w_{0}s_{i}\vpi^{I}_{i}}' \ge 
\pair{w_{0}\lambda}{\vpi^{I}_{i}}$ for all $i \in I$. 
A simple computation shows the following lemma.
\begin{lem}
Let $\bM=(M_{\gamma})_{\gamma \in \Gamma_{I}} \in \bz_{I}$. 
Then, the MV polytope $P(\bM)$ is contained in 
$\mv_{I}(\lambda)$ (i.e., $\lambda+P(\bM) \subset 
\conv (W_{I}\lambda)$) if and only if 
%
%%%%%%%%%%%%%%%%%%
%%% eq:blambda %%%
%%%%%%%%%%%%%%%%%%
%
\begin{equation} \label{eq:blambda}
M_{-s_{i}\vpi^{I}_{i}} \ge -\pair{\lambda}{\alpha_{i}}
\quad \text{\rm for all $i \in I$}.
\end{equation}
\end{lem}

We denote by $\bz_{I}(\lambda)$ the set of all BZ data 
$\bM=(M_{\gamma})_{\gamma \in \Gamma_{I}} \in \bz_{I}$ 
satisfying \eqref{eq:blambda}. By the lemma above,  
the restriction of the bijection 
$\bz_{I} \rightarrow \mv_{I}$, $\bM \mapsto P(\bM)$, 
to the subset $\bz_{I}(\lambda) \subset \bz_{I}$ 
gives rise to a bijection between 
$\bz_{I}(\lambda)$ and $\mv_{I}(\lambda)$. 
By \cite[Theorem~6.4]{Kam2}, the set 
$\mv_{I}(\lambda)$ can be endowed with a crystal structure 
for $U_{q}(\Fg_{I}^{\vee})$, and 
the resulting crystal $\mv_{I}(\lambda)$ is isomorphic to 
the crystal basis $\CB_{I}(\lambda)$ of 
the finite-dimensional irreducible highest weight 
$U_{q}(\Fg_{I}^{\vee})$-module of highest weight $\lambda$. 
Thus, we can also endow the set $\bz_{I}(\lambda)$ with 
a crystal structure for $U_{q}(\Fg_{I}^{\vee})$ 
in such a way that the bijection 
$\bz_{I}(\lambda) \rightarrow \mv_{I}(\lambda)$ above is 
an isomorphism of crystals for $U_{q}(\Fg_{I}^{\vee})$. 

Now we recall from \cite[\S6.4]{Kam2} 
the description of the crystal structure on $\bz_{I}(\lambda)$. 
For $\bM=(M_{\gamma})_{\gamma \in \Gamma_{I}} \in \bz_{I}(\lambda)$, 
define the weight $\Wt(\bM)$ of $\bM$ by: 
%
%%%%%%%%%%%%%%%%%
%%% eq:dfn-Wt %%%
%%%%%%%%%%%%%%%%%
%
\begin{equation} \label{eq:dfn-Wt}
\Wt(\bM)
 =\lambda+\wt(\bM)
 =\lambda+\sum_{i \in I} M_{\vpi^{I}_{i}} \, h_{i}.
\end{equation}
The raising Kashiwara operators $e_{p}$, $p \in I$, and 
the maps $\ve_{p}$, $p \in I$, on $\bz_{I}(\lambda)$ are 
defined by restricting those on $\bz_{I}$ to 
the subset $\bz_{I}(\lambda) \subset \bz_{I}$. 
The lowering Kashiwara operators $F_{p}$, $p \in I$, 
on $\bz_{I}(\lambda)$ are defined as follows: 
for $\bM \in \bz_{I}(\lambda)$ and $p \in I$, 
\begin{equation*}
F_{p}\bM=
\begin{cases}
f_{p}\bM & 
 \text{if $f_{p}\bM$ is an element of $\bz_{I}(\lambda)$}, \\[1.5mm]
\bzero & \text{otherwise}.
\end{cases}
\end{equation*}
Also, we set $\Phi_{p}(\bM):=
 \pair{\Wt(\bM)}{\alpha_{p}}+\ve_{p}(\bM)$ 
for $\bM \in \bz_{I}(\lambda)$ and $p \in I$. 
It is easily seen by \eqref{eq:dfn-vep} and 
\eqref{eq:dfn-Wt} that 
if $\bM=(M_{\gamma})_{\gamma \in \Gamma_{I}}$, then 
%
%%%%%%%%%%%%%%%%
%%% eq:vp-bM %%%
%%%%%%%%%%%%%%%%
%
\begin{equation} \label{eq:vp-bM}
\Phi_{p}(\bM)=
 M_{\vpi^{I}_{p}}-M_{s_{p}\vpi^{I}_{p}}+\pair{\lambda}{\alpha_{p}}.
\end{equation}
%
%%%%%%%%%%%%%%%%%%%
%%% thm:blambda %%%
%%%%%%%%%%%%%%%%%%%
%
\begin{thm}[{\cite[Theorem~6.4]{Kam2}}] \label{thm:blambda}
Let $\lambda \in \Fh_{I}$ be 
a dominant integral weight for $\Fg_{I}^{\vee}$. 
Then, the set $\bz_{I}(\lambda)$, equipped with the maps 
$\wt$, $e_{p},\,F_{p} \ (p \in I)$, and 
$\ve_{p},\,\Phi_{p} \ (p \in I)$ above, is 
a crystal for $U_{q}(\Fg_{I}^{\vee})$ 
isomorphic to the crystal basis $\CB_{I}(\lambda)$ of 
the finite-dimensional irreducible highest weight 
$U_{q}(\Fg_{I}^{\vee})$-module of highest weight $\lambda$.
\end{thm}

%==============================%
%     START SUBSECTION 0204    %
%==============================%
%
\subsection{Restriction to subintervals.}
\label{subsec:res-subint}

Let $K$ be a fixed (finite) interval in $\BZ$ such that $K \subset I$. 
The Cartan matrix $A_{K}$ of 
the finite-dimensional simple Lie algebra $\Fg_{K}$ equals 
the principal submatrix of the Cartan matrix $A_{I}$ 
of $\Fg_{I}$ corresponding to the subset $K \subset I$. 
Also, the Weyl group $W_{K}$ of $\Fg_{K}$ can be identified 
with the subgroup of the Weyl group $W_{I}$ of $\Fg_{I}$ 
generated by the subset $\bigl\{s_{i} \mid i \in K\bigr\}$ 
of $\bigl\{s_{i} \mid i \in I\bigr\}$. 
Moreover, we can (and do) identify the set $\Gamma_{K}$ 
(of chamber weights) for $\Fg_{K}$ 
(defined by \eqref{eq:GammaI} with $I$ replaced by $K$) 
with the subset $\bigl\{-w\vpi_{i}^{I} \mid w \in W_{K},\,i \in K\bigr\}$ 
of the set $\Gamma_{I}$ (of chamber weights) through 
the following bijection of sets: 
%
%%%%%%%%%%%%%%%%%%%%%
%%% eq:bij-index1 %%%
%%%%%%%%%%%%%%%%%%%%%
%
\begin{equation} \label{eq:bij-index1}
\begin{array}{rcl}
\Gamma_{K} & \stackrel{\sim}{\rightarrow} & 
\bigl\{-w\vpi_{i}^{I} \mid 
w \in W_{K},\,i \in K\bigr\} \subset \Gamma_{I}, \\[3mm]
-w\vpi_{i}^{K} & \mapsto & -w\vpi_{i}^{I} 
\quad \text{for $w \in W_{K}$ and $i \in K$};
\end{array}
\end{equation}
observe that the map above is well-defined. 
Indeed, suppose that $w\vpi_{i}^{K}=v\vpi_{j}^{K}$ 
for some $w,\,v \in W_{K}$ and $i,\,j \in K$. 
Since $\vpi_{i}^{K}$ and $\vpi_{j}^{K}$ are dominant, 
it follows immediately that $i=j$, and hence 
$w\vpi_{i}^{K}=v\vpi_{j}^{K}=v\vpi_{i}^{K}$. 
Since $v^{-1}w\vpi_{i}^{K}=\vpi_{i}^{K}$ (i.e., 
$v^{-1}w$ stabilizes $\vpi_{i}^{K}$), we see that 
$v^{-1}w$ is a product of $s_{k}$'s for 
$k \in K \setminus \{i\}$. 
Therefore, we obtain $v^{-1}w\vpi_{i}^{I}=\vpi_{i}^{I}$, 
and hence $w\vpi_{i}^{I}=v\vpi_{i}^{I}=v\vpi_{j}^{I}$, 
as desired. 
Also, note that for each $i \in K$, 
the fundamental weight $\vpi_{i}^{K} \in \Gamma_{K}$ for $\Fg_{K}$ 
corresponds to 
$-w_{0}^{K}(\vpi_{\omega_{K}(i)}^{I})=
 w_{0}^{K}w_{0}^{I}\vpi_{\omega_{I}\omega_{K}(i)}^{I} \in \Gamma_{I}$ 
under the bijection \eqref{eq:bij-index1}, 
where $\omega_{K}:K \rightarrow K$ denotes the (Dynkin) diagram 
automorphism for $\Fg_{K}$. 
For a collection $\bM=(M_{\gamma})_{\gamma \in \Gamma_{I}}$ 
of integers indexed by $\Gamma_{I}$, 
we set $\bM_{K}:=(M_{\gamma})_{\gamma \in \Gamma_{K}}$, 
regarding the set $\Gamma_{K}$ 
as a subset of the set $\Gamma_{I}$ through 
the bijection \eqref{eq:bij-index1}. 
%
%%%%%%%%%%%%%%%%
%%% lem:res1 %%%
%%%%%%%%%%%%%%%%
%
\begin{lem} \label{lem:res1}
Keep the notation above. 
If $\bM=(M_{\gamma})_{\gamma \in \Gamma_{I}}$ is an element of $\bz_{I}$, 
then $\bM_{K}=(M_{\gamma})_{\gamma \in \Gamma_{K}}$ is 
a BZ datum for $\Fg_{K}$ that is an element of $\bz_{K}$. 
\end{lem}

\begin{proof}
First we show that $\bM_{K}$ satisfies condition (1) of 
Definition~\ref{dfn:BZdatum} (with $I$ replaced by $K$), 
i.e., for $w \in W_{K}$ and $i \in K$, 
%
%%%%%%%%%%%%%%%%
%%% eq:edge0 %%%
%%%%%%%%%%%%%%%%
%
\begin{equation} \label{eq:edge0}
M_{w\vpi^{K}_{i}}+M_{ws_{i}\vpi^{K}_{i}}+
\sum_{j \in K \setminus \{i\}} a_{ji} M_{w\vpi^{K}_{j}} \le 0. 
\end{equation}
Observe that under the bijection \eqref{eq:bij-index1}, 
we have 
%
%%%%%%%%%%%%%%%
%%% eq:corr %%%
%%%%%%%%%%%%%%%
%
\begin{equation} \label{eq:corr}
\begin{array}{rcl}
w\vpi_{k}^{K} & \mapsto & 
wv_{0}\vpi_{\tau(k)}^{I} \quad (k \in K), \\[3mm]
ws_{i}\vpi_{i}^{K} & \mapsto &
ws_{i}v_{0}\vpi_{\tau(i)}^{I}=
wv_{0}s_{\tau(i)}\vpi_{\tau(i)}^{I},
\end{array}
\end{equation}
where we set $v_{0}:=w_{0}^{K}w_{0}^{I}$ and 
$\tau:=\omega_{I}\omega_{K}$ for simplicity of notation.
Since $\bM$ is a BZ datum for $\Fg_{I}$, it follows 
from condition (1) of Definition~\ref{dfn:BZdatum} 
for $wv_{0} \in W_{I}$ and $\tau(i) \in I$ that
%
%%%%%%%%%%%%%%%%%
%%% eq:edge01 %%%
%%%%%%%%%%%%%%%%%
%
\begin{equation} \label{eq:edge01}
M_{wv_{0}\vpi_{\tau(i)}^{I}}+M_{wv_{0}s_{\tau(i)}\vpi_{\tau(i)}^{I}}+
\sum_{j \in I \setminus \{\tau(i)\}} 
a_{j, \tau(i)} M_{wv_{0}\vpi_{j}^{I}} \le 0. 
\end{equation}
Here, using the equality $a_{\omega_{I}(j),\tau(i)} = 
a_{j,\omega_{K}(i)}$ for $j \in I$, we see that
\begin{equation*}
\sum_{j \in I \setminus \{\tau(i)\}} 
a_{j, \tau(i)} M_{wv_{0}\vpi_{j}^{I}}=
\sum_{\omega_{I}(j) \in I \setminus \{\tau(i)\}} 
a_{\omega_{I}(j), \tau(i)} M_{wv_{0}\vpi_{\omega_{I}(j)}^{I}}=
\sum_{j \in I \setminus \{\omega_{K}(i)\}} 
a_{j, \omega_{K}(i)} M_{wv_{0}\vpi_{\omega_{I}(j)}^{I}}.
\end{equation*}
Also, if $j \in I \setminus K$, then 
\begin{align*}
M_{wv_{0}\vpi_{\omega_{I}(j)}^{I}} & 
 =M_{-ww_{0}^{K}\vpi_{j}^{I}}
 =M_{-\vpi_{j}^{I}} \quad \text{since $ww_{0}^{K} \in W_{K}$} \\
& = 0 \quad \text{since $\bM \in \bz_{I}$}.
\end{align*}
Hence it follows that 
\begin{equation*}
\sum_{j \in I \setminus \{\omega_{K}(i)\}} 
a_{j, \omega_{K}(i)} M_{wv_{0}\vpi_{\omega_{I}(j)}^{I}}=
\sum_{j \in K \setminus \{\omega_{K}(i)\}} 
a_{j, \omega_{K}(i)} M_{wv_{0}\vpi_{\omega_{I}(j)}^{I}}.
\end{equation*}
Furthermore, using the equality 
$a_{\omega_{K}(j),\omega_{K}(i)}=a_{ji}$ for $j \in K$, 
we get
\begin{align*}
\sum_{j \in K \setminus \{\omega_{K}(i)\}} 
a_{j, \omega_{K}(i)} M_{wv_{0}\vpi_{\omega_{I}(j)}^{I}} & =
\sum_{\omega_{K}(j) \in K \setminus \{\omega_{K}(i)\}} 
a_{\omega_{K}(j), \omega_{K}(i)} 
M_{wv_{0}\vpi_{\omega_{I}(\omega_{K}(j))}^{I}} \\[3mm]
& 
=
\sum_{j \in K \setminus \{i\}} 
a_{ji} M_{wv_{0}\vpi_{\tau(j)}^{I}}.
\end{align*}
Substituting this into \eqref{eq:edge01}, we obtain
\begin{equation*}
M_{wv_{0}\vpi_{\tau(i)}^{I}}+M_{wv_{0}s_{\tau(i)}\vpi_{\tau(i)}^{I}}+
\sum_{j \in K \setminus \{i\}} 
a_{ji} M_{wv_{0}\vpi_{\tau(j)}^{I}} \le 0.
\end{equation*}
The inequality \eqref{eq:edge0} follows immediately 
from this inequality and the correspondence \eqref{eq:corr}. 

Next we show that $\bM_{K}$ satisfies condition (2) of 
Definition~\ref{dfn:BZdatum} (with $I$ replaced by $K$), 
i.e., for $w \in W_{K}$ and $i,\,j \in K$ with 
$a_{ij}=a_{ji}=-1$ such that $ws_{i} > w$, $ws_{j} > w$, 
%
%%%%%%%%%%%%%%%
%%% eq:TPR0 %%%
%%%%%%%%%%%%%%%
%
\begin{equation} \label{eq:TPR0}
M_{ws_{i}\vpi^{K}_{i}}+M_{ws_{j}\vpi^{K}_{j}} =
\min \bigl(
 M_{w\vpi^{K}_{i}}+M_{ws_{i}s_{j}\vpi^{K}_{j}}, \ 
 M_{w\vpi^{K}_{j}}+M_{ws_{j}s_{i}\vpi^{K}_{i}}
 \bigr).
\end{equation}
Observe that under the bijection \eqref{eq:bij-index1}, 
we have 
%
%%%%%%%%%%%%%%%%
%%% eq:corr1 %%%
%%%%%%%%%%%%%%%%
%
\begin{equation} \label{eq:corr1}
\begin{array}{rcl}
w\vpi_{k}^{K} & \mapsto & 
wv_{0}\vpi_{\tau(k)}^{I} \quad (k \in K), \\[3mm]
ws_{k}\vpi_{k}^{K} & \mapsto &
ws_{k}v_{0}\vpi_{\tau(k)}^{I}=
wv_{0}s_{\tau(k)}\vpi_{\tau(k)}^{I} \quad (k \in K), \\[3mm]
ws_{l}s_{k}\vpi_{k}^{K} & \mapsto &
ws_{l}s_{k}v_{0}\vpi_{\tau(k)}^{I}=
wv_{0}s_{\tau(l)}s_{\tau(k)}\vpi_{\tau(k)}^{I} \quad (k,\,l \in K).
\end{array}
\end{equation}
Since $a_{\tau(i),\tau(j)}=a_{\tau(j),\tau(i)}=-1$ and 
$wv_{0}s_{\tau(k)}=ws_{k}v_{0} > wv_{0}$ for $k=i,\,j$, and 
since $\bM$ is a BZ datum for $\Fg_{I}$, 
it follows from condition (2) of Definition~\ref{dfn:BZdatum} 
for $wv_{0} \in W_{I}$ and $\tau(i),\,\tau(j) \in I$ that 
\begin{align*}
& M_{wv_{0}s_{\tau(i)}\vpi^{I}_{\tau(i)}}+
  M_{wv_{0}s_{\tau(j)}\vpi^{I}_{\tau(j)}} \\
& \hspace*{10mm} =
\min \bigl(
 M_{wv_{0}\vpi^{I}_{\tau(i)}}+M_{wv_{0}s_{\tau(i)}s_{\tau(j)}\vpi^{I}_{\tau(j)}}, \ 
 M_{wv_{0}\vpi^{I}_{\tau(j)}}+M_{wv_{0}s_{\tau(j)}s_{\tau(i)}\vpi^{I}_{\tau(i)}}
 \bigr).
\end{align*}
The equation \eqref{eq:TPR0} follows immediately 
from this equation and the correspondence \eqref{eq:corr1}. 

Finally, it is obvious that $M_{w_{0}^{K}\vpi_{i}^{K}}=
M_{-\vpi_{\omega_{K}(i)}^{I}}=0$ 
for all $i \in K$, since $\bM \in \bz_{I}$. 
This proves the lemma.
\end{proof}

Now, we set $\Gamma_{I}^{K}:=
\bigl\{w\vpi_{i}^{I} \mid w \in W_{K},\,i \in K\bigr\}
\subset \Gamma_{I}$. Then there exists 
the following bijection of sets between 
$\Gamma_{K}$ and $\Gamma_{I}^{K}$: 
%
%%%%%%%%%%%%%%%%%%%%%
%%% eq:bij-index2 %%%
%%%%%%%%%%%%%%%%%%%%%
%
\begin{equation} \label{eq:bij-index2}
\begin{array}{rcl}
\Gamma_{K} & \stackrel{\sim}{\rightarrow} & 
\Gamma_{I}^{K}, \\[3mm]
w\vpi_{i}^{K} & \mapsto & w\vpi_{i}^{I} 
\quad \text{for $w \in W_{K}$ and $i \in K$};
\end{array}
\end{equation}
the argument above for the correspondence \eqref{eq:bij-index1} 
shows that this map is well-defined. 
For a collection $\bM=(M_{\gamma})_{\gamma \in \Gamma_{I}}$ 
of integers indexed by $\Gamma_{I}$, 
we define $\bM^{K}:=(M_{\gamma})_{\gamma \in \Gamma_{I}^{K}}$, 
and regard it as a collection of integers indexed by 
$\Gamma_{K}$ through the bijection \eqref{eq:bij-index2} 
between the index sets. 
%
%%%%%%%%%%%%%%%%
%%% lem:res2 %%%
%%%%%%%%%%%%%%%%
%
\begin{lem} \label{lem:res2}
Keep the notation above. 
If $\bM=(M_{\gamma})_{\gamma \in \Gamma_{I}}$ 
is an element of $\bz_{I}$, then $\bM^{K}$ is 
a BZ datum for $\Fg_{K}$. 
% (but, in general, not contained in $\bz_{K}$). 
\end{lem}

\begin{proof}
First we show that $\bM^{K}$ satisfies condition (1) of 
Definition~\ref{dfn:BZdatum} (with $I$ replaced by $K$), 
i.e., for $w \in W_{K}$ and $i \in K$, 
%
%%%%%%%%%%%%%%%%
%%% eq:edge2 %%%
%%%%%%%%%%%%%%%%
%
\begin{equation} \label{eq:edge2}
M_{w\vpi^{K}_{i}}+M_{ws_{i}\vpi^{K}_{i}}+
\sum_{j \in K \setminus \{i\}} a_{ji} M_{w\vpi^{K}_{j}} \le 0. 
\end{equation}
Since $\bM$ is a BZ datum for $\Fg_{I}$, it follows 
from condition (1) of Definition~\ref{dfn:BZdatum} 
for $w \in W_{I}$ and $i \in I$ that
\begin{equation*}
M_{w\vpi^{I}_{i}}+M_{ws_{i}\vpi^{I}_{i}}+
\sum_{j \in I \setminus \{i\}} a_{ji} M_{w\vpi^{I}_{j}} \le 0,
\end{equation*}
and hence
%
%%%%%%%%%%%%%%%%%%
%%% eq:edge2-1 %%%
%%%%%%%%%%%%%%%%%%
%
\begin{equation} \label{eq:edge2-1}
M_{w\vpi^{I}_{i}}+M_{ws_{i}\vpi^{I}_{i}}+
\sum_{j \in K \setminus \{i\}} a_{ji} M_{w\vpi^{I}_{j}} + 
\sum_{j \in I \setminus K} a_{ji} M_{w\vpi^{I}_{j}} \le 0.
\end{equation}
Because $M_{\gamma} \in \BZ_{\le 0}$ for all $\gamma \in \Gamma_{I}$ 
by Remark~\ref{rem:highest}, it follows that all terms 
$a_{ji}M_{w\vpi^{I}_{j}}$, $j \in I \setminus K$, of 
the second sum in \eqref{eq:edge2-1} are 
nonnegative integers. Hence we obtain 
\begin{equation*}
M_{w\vpi^{I}_{i}}+M_{ws_{i}\vpi^{I}_{i}}+
\sum_{j \in K \setminus \{i\}} a_{ji} M_{w\vpi^{I}_{j}} \le 0.
\end{equation*}
The inequality \eqref{eq:edge2} follows immediately 
from this equality and the correspondence 
\eqref{eq:bij-index2}. 

Next we show that $\bM^{K}$ satisfies condition (2) of 
Definition~\ref{dfn:BZdatum} (with $I$ replaced by $K$), 
i.e., for $w \in W_{K}$ and $i,\,j \in K$ with 
$a_{ij}=a_{ji}=-1$ such that $ws_{i} > w$, $ws_{j} > w$, 
%
%%%%%%%%%%%%%%%
%%% eq:TPR1 %%%
%%%%%%%%%%%%%%%
%
\begin{equation} \label{eq:TPR1}
M_{ws_{i}\vpi^{K}_{i}}+M_{ws_{j}\vpi^{K}_{j}} =
\min \bigl(
 M_{w\vpi^{K}_{i}}+M_{ws_{i}s_{j}\vpi^{K}_{j}}, \ 
 M_{w\vpi^{K}_{j}}+M_{ws_{j}s_{i}\vpi^{K}_{i}}
 \bigr).
\end{equation}
Since $\bM$ is a BZ datum for $\Fg_{I}$, it follows from 
condition (2) of Definition~\ref{dfn:BZdatum} 
for $w \in W_{I}$ and $i,\,j \in I$ that 
\begin{equation*}
M_{ws_{i}\vpi^{I}_{i}}+M_{ws_{j}\vpi^{I}_{j}} =
\min \bigl(
 M_{w\vpi^{I}_{i}}+M_{ws_{i}s_{j}\vpi^{I}_{j}}, \ 
 M_{w\vpi^{I}_{j}}+M_{ws_{j}s_{i}\vpi^{I}_{i}}
 \bigr).
\end{equation*}
The equation \eqref{eq:TPR1} follows immediately 
from this equation and the correspondence \eqref{eq:bij-index2}. 
This proves the lemma.
\end{proof}

%=========================%
%     START SECTION 03    %
%=========================%
%
\section{Berenstein-Zelevinsky data of type $A_{\infty}$.}
\label{sec:BZdatum-inf}

%==============================%
%     START SUBSECTION 0301    %
%==============================%
%
\subsection{Basic notation in type $A_{\infty}$.}
\label{subsec:notation-inf}

Let $A_{\BZ}=(a_{ij})_{i,j \in \BZ}$ denote 
the generalized Cartan matrix of 
type $A_{\infty}$ with index set $\BZ$; 
the entries $a_{ij}$ are given by: 
%
%%%%%%%%%%%%%%%%%
%%% eq:index3 %%%
%%%%%%%%%%%%%%%%%
%
\begin{equation} \label{eq:index3}
a_{ij}=
 \begin{cases}
 2 & \text{if $i=j$}, \\
 -1 & \text{if $|i-j|=1$}, \\
 0 & \text{otherwise},
 \end{cases}
\end{equation}
for $i,\,j \in \BZ$. Let
\begin{equation*}
\bigl(A_{\BZ}, \, 
 \Pi:=\bigl\{\alpha_{i}\bigr\}_{i \in \BZ}, \, 
 \Pi^{\vee}:=\bigl\{h_{i}\bigr\}_{i \in \BZ}, \, 
 \Fh^{\ast},\,\Fh
\bigr)
\end{equation*}
be the root datum of type $A_{\infty}$. Namely, 
$\Fh$ is a complex infinite-dimensional vector space, with 
$\Pi^{\vee}$ a basis of $\Fh$, and 
$\Pi$ is a linearly independent subset of the (full) dual space 
$\Fh^{\ast}:=\Hom_{\BC}(\Fh,\,\BC)$ of $\Fh$ such that 
$\pair{h_{i}}{\alpha_{j}}=a_{ij}$ for $i,\,j \in \BZ$, 
where $\pair{\cdot}{\cdot}$ is the canonical pairing 
between $\Fh$ and $\Fh^{\ast}$. 
For each $i \in \BZ$, 
define $\Lambda_{i} \in \Fh^{\ast}$ by: 
$\pair{h_{j}}{\Lambda_{i}}=\delta_{ij}$ 
for $j \in \BZ$. 
Let $W_{\BZ}:=\langle s_{i} \mid i \in \BZ \rangle \ (\subset GL(\Fh^{\ast}))$ 
be the Weyl group of type $A_{\infty}$, 
where $s_{i}$ is the simple reflection for $i \in \BZ$. 
Also, we denote by $\le$ the (strong) Bruhat order on $W_{\BZ}$ 
(cf. \cite[\S8.3]{BjB}). 

Set 
%
%%%%%%%%%%%%%%%%%
%%% eq:GammaZ %%%
%%%%%%%%%%%%%%%%%
%
\begin{equation} \label{eq:GammaZ}
\Gamma_{\BZ}:=\bigl\{-w\Lambda_{i} \mid w \in W_{\BZ},\ i \in \BZ\bigr\}, 
\quad \text{and} \quad
\Xi_{\BZ}:=-\Gamma_{\BZ}.
\end{equation}
We should note that $\Gamma_{\BZ} \cap \Xi_{\BZ}=\emptyset$. 
Indeed, suppose that $\gamma \in \Gamma_{\BZ} \cap \Xi_{\BZ}$. 
Since $\gamma \in \Gamma_{\BZ}$ (resp., $\gamma \in \Xi_{\BZ}$), 
it can be written as: $\gamma=-w\Lambda_{i}$ 
(resp., $\gamma=v\Lambda_{j}$) for some $w \in W_{\BZ}$ and 
$i \in \BZ$ (resp., $v \in W_{\BZ}$ and $j \in \BZ$). 
Then we have $\gamma=-w\Lambda_{i}=v\Lambda_{j}$, 
and hence $-\Lambda_{i}=w^{-1}v\Lambda_{j}$. 
Since $\Lambda_{j}$ is a dominant integral weight, we see that 
$w^{-1}v\Lambda_{j}$ is of the form:
\begin{equation*}
w^{-1}v\Lambda_{j}=
 \Lambda_{j}-(m_{1}\alpha_{i_{1}}+m_{2}\alpha_{i_{2}}+ \cdots + m_{p}\alpha_{i_{p}})
\end{equation*}
for some $m_{1},\,m_{2},\,\dots,\,m_{p} \in \BZ_{> 0}$ and 
$i_{1},\,i_{2},\,\dots,\,i_{p} \in \BZ$ with 
$i_{1} < i_{2} < \cdots < i_{p}$. If we set $k:=i_{p}+1$, then
we see that 
\begin{equation*}
\pair{h_{k}}{w^{-1}v\Lambda_{j}}=
\pair{h_{k}}{\Lambda_{j}}-m_{p}\pair{h_{k}}{\alpha_{i_{p}}}=
\pair{h_{k}}{\Lambda_{j}}+m_{p} > 0.
\end{equation*}
However, we have
\begin{equation*}
0 < \pair{h_{k}}{w^{-1}v\Lambda_{j}}=
\pair{h_{k}}{-\Lambda_{i}} \le 0,
\end{equation*}
which is a contradiction. Thus we have shown that 
$\Gamma_{\BZ} \cap \Xi_{\BZ}=\emptyset$. 

Let $\bM=(M_{\gamma})_{\gamma \in \Gamma_{\BZ}}$ 
(resp., $\bM=(M_{\xi})_{\xi \in \Xi_{\BZ}}$) be 
a collection of integers indexed by $\Gamma_{\BZ}$ 
(resp., $\Xi_{\BZ}$). 
For each $\gamma \in \Gamma_{\BZ}$ 
(resp., $\xi \in \Xi_{\BZ}$), 
we call $M_{\gamma}$ (resp., $M_{\xi}$) 
the $\gamma$-component (resp. the $\xi$-component) of $\bM$, and 
denote it by $(\bM)_{\gamma}$ (resp., $(\bM)_{\xi}$). 

Let $I$ be a (finite) interval in $\BZ$. 
Then the Cartan matrix $A_{I}$ of 
the finite-dimensional simple 
Lie algebra $\Fg_{I}$ (see \S\ref{subsec:notation}) 
equals the principal submatrix of $A_{\BZ}$ corresponding to $I \subset \BZ$. 
Also, the Weyl group $W_{I}$ of $\Fg_{I}$ can be identified with 
the subgroup of the Weyl group $W_{\BZ}$ generated by 
the subset $\bigl\{s_{i} \mid i \in I\bigr\}$ of 
$\bigl\{s_{i} \mid i \in \BZ\bigr\}$.
Moreover, we can (and do) identify the set $\Gamma_{I}$ 
(of chamber weights) for $\Fg_{I}$, defined by \eqref{eq:GammaI}, 
with the subset $\bigl\{-w\Lambda_{i} \mid w \in W_{I},\,i \in I\bigr\}$ 
of the set $\Gamma_{\BZ}$ (of chamber weights) through 
the following bijection of sets: 
%
%%%%%%%%%%%%%%%%%%%%%
%%% eq:bij-index3 %%%
%%%%%%%%%%%%%%%%%%%%%
%
\begin{equation} \label{eq:bij-index3}
\begin{array}{rcl}
\Gamma_{I} & \stackrel{\sim}{\rightarrow} & 
\bigl\{-w\Lambda_{i} \mid 
w \in W_{I},\,i \in I\bigr\} \subset \Gamma_{\BZ}, \\[3mm]
-w\vpi_{i}^{I} & \mapsto & -w\Lambda_{i}
\quad \text{for $w \in W_{I}$ and $i \in I$};
\end{array}
\end{equation}
the same argument as for the correspondence \eqref{eq:bij-index1} 
shows that this map is well-defined.
Note that for each $i \in I$, 
the fundamental weight $\vpi_{i}^{I} \in \Gamma_{I}$ for $\Fg_{I}$ 
corresponds to 
$-w_{0}^{I}(\Lambda_{\omega_{I}(i)}) \in \Gamma_{\BZ}$ 
under the bijection \eqref{eq:bij-index3}, 
where $\omega_{I}:I \rightarrow I$ denotes 
the (Dynkin) diagram automorphism for $\Fg_{I}$. 

\begin{rem}
Let $I$ be an interval in $\BZ$, and fix $i \in I$. 
The element 
$\vpi_{i}^{I}=-w_{0}^{I}(\Lambda_{\omega_{I}(i)}) \in \Gamma_{\BZ}$ 
satisfies the following property: for $j \in \BZ$, 
%
%%%%%%%%%%%%%%%%%%%
%%% eq:pair-vpi %%%
%%%%%%%%%%%%%%%%%%%
%
\begin{equation} \label{eq:pair-vpi}
\pair{h_{j}}{\vpi_{i}^{I}}=
\begin{cases}
\delta_{ij} & \text{if $j \in I$}, \\
-1 & \text{if $j=(\min I) -1$ or $j= (\max I) + 1$}, \\
0 & \text{otherwise}.
\end{cases}
\end{equation}
Indeed, it is easily seen that 
$\pair{h_{j}}{\vpi_{i}^{I}}=\delta_{ij}$ for $j \in I$. 
Also, if $j < (\min I) -1$ or $j > (\max I) + 1$, 
then $(w_{0}^{I})^{-1}h_{j}=h_{j}$ 
since $w_{0}^{I} \in W_{I}=\langle s_{i} \mid i \in I \rangle$. 
Hence
\begin{equation*}
\pair{h_{j}}{\vpi_{i}^{I}}=
\pair{h_{j}}{-w_{0}^{I}(\Lambda_{\omega_{I}(i)})}=
-\pair{(w_{0}^{I})^{-1}h_{j}}{\Lambda_{\omega_{I}(i)}}=
-\pair{h_{j}}{\Lambda_{\omega_{I}(i)}}=0.
\end{equation*}
It remains to show that $\pair{h_{j}}{\vpi_{i}^{I}}=-1$ if 
$j=(\min I) -1$ or $j= (\max I) + 1$. 
For simplicity of notation, suppose that 
$I=\bigl\{1,\,2\,\dots,\,m\bigr\}$ and $j=0$. 
Then, by using the reduced expression 
$w_{0}^{I}=(s_{1}s_{2} \cdots s_{m})
  (s_{1}s_{2} \cdots s_{m-1}) \cdots 
  (s_{1}s_{2})s_{1}$
of the longest element $w_{0}^{I} \in W_{I}$, 
we deduce that 
$(w_{0}^{I})^{-1}h_{0}=
h_{0}+h_{1}+\cdots+h_{m}$.
Therefore, 
\begin{align*}
\pair{h_{0}}{\vpi_{i}^{I}} & =
\pair{h_{0}}{-w_{0}^{I}(\Lambda_{\omega_{I}(i)})}=
-\pair{(w_{0}^{I})^{-1}h_{0}}{\Lambda_{\omega_{I}(i)}} \\
& =
-\pair{h_{0}+h_{1}+\cdots+h_{m}}{\Lambda_{\omega_{I}(i)}}=-1,
\end{align*}
as desired.
\end{rem}

For a collection $\bM=(M_{\gamma})_{\gamma \in \Gamma_{\BZ}}$ 
of integers indexed by $\Gamma_{\BZ}$, 
we set $\bM_{I}:=(M_{\gamma})_{\gamma \in \Gamma_{I}}$, 
regarding the set $\Gamma_{I}$ as a subset of the set $\Gamma_{\BZ}$ 
through the bijection \eqref{eq:bij-index3}. 
Note that if $K$ is an interval in $\BZ$ such that $K \subset I$, 
then $(\bM_{I})_{K}=\bM_{K}$ (for the notation, 
see \S\ref{subsec:res-subint}). 

%==============================%
%     START SUBSECTION 0302    %
%==============================%
%
\subsection{BZ data of type $A_{\infty}$.}
\label{subsec:BZdatum-inf}
%
%%%%%%%%%%%%%%%%%%%%
%%% dfn:BZdatum2 %%%
%%%%%%%%%%%%%%%%%%%%
%
\begin{dfn} \label{dfn:BZdatum2}
A collection 
$\bM=(M_{\gamma})_{\gamma \in \Gamma_{\BZ}}$ 
of integers indexed by $\Gamma_{\BZ}$ is 
called a BZ datum of type $A_{\infty}$ 
if it satisfies the following conditions: 

(a) For each interval $K$ in $\BZ$, 
$\bM_{K}=(M_{\gamma})_{\gamma \in \Gamma_{K}}$ 
is a BZ datum for $\Fg_{K}$, and is an element of $\bz_{K}$ 
(cf. Lemma~\ref{lem:res1}). 

(b) For each $w \in W_{\BZ}$ and $i \in \BZ$, 
there exists an interval $I$ in $\BZ$ such that 
$i \in I$, $w \in W_{I}$, and 
$M_{w\vpi_{i}^{J}}=M_{w\vpi_{i}^{I}}$ for all 
intervals $J$ in $\BZ$ containing $I$. 
\end{dfn}
%
%%%%%%%%%%%%
%%% ex:O %%%
%%%%%%%%%%%%
%
\begin{ex} \label{ex:O}
Let $\bO$ be a collection of integers indexed by $\Gamma_{\BZ}$ 
whose $\gamma$-component is equal to $0$ 
for each $\gamma \in \Gamma_{\BZ}$. 
Then it is obvious that $\bO$ is a BZ datum of type $A_{\infty}$ 
(cf. Remark~\ref{rem:highest}). 
\end{ex}

Let $\bz_{\BZ}$ denote the set of all BZ data of type $A_{\infty}$. 
For $\bM=(M_{\gamma})_{\gamma \in \Gamma_{\BZ}} \in \bz_{\BZ}$, and 
for each $w \in W$ and $i \in \BZ$, we denote by $\Int (\bM; w,i)$ 
the set of all intervals $I$ in $\BZ$ satisfying 
condition~(b) of Definition~\ref{dfn:BZdatum2} for the $w$ and $i$. 
%
%%%%%%%%%%%%%%%%%
%%% rem:limit %%%
%%%%%%%%%%%%%%%%%
%
\begin{rem} \label{rem:limit}
(1) Let $\bM$ be a BZ datum of type $A_{\infty}$, 
i.e., $\bM \in \bz_{\BZ}$, and 
let $w \in W_{\BZ}$ and $i \in \BZ$. 
It is obvious that if $I \in \Int (\bM;w,i)$, 
then $J \in \Int(\bM;w,i)$ for every interval 
$J$ in $\BZ$ containing $I$. 

(2) Let $\bM_{b} \ (1 \le b \le a)$ be BZ data of type $A_{\infty}$, 
and let $w_{b} \in W_{\BZ} \ (1 \le b \le a)$ and 
$i_{b} \in \BZ \ (1 \le b \le a)$. Then the intersection 
\begin{equation*}
\Int(\bM_{1};w_{1},i_{1}) \cap
\Int(\bM_{2};w_{2},i_{2}) \cap \cdots \cap
\Int(\bM_{a};w_{a},i_{a})
\end{equation*}
is nonempty. Indeed, we first take 
$I_{b} \in \Int(\bM_{b};w_{b},i_{b})$ arbitrarily for each $1 \le b \le a$, 
and then take an interval $J$ in $\BZ$ such that 
$J \supset I_{b}$ for all $1 \le b \le a$ (i.e., 
$J \supset I_{1} \cup I_{2} \cup \cdots \cup I_{a}$). 
Then, it follows immediately from part (1) that 
$J \in \Int(\bM_{b};w_{b},i_{b})$ for all $1 \le b \le a$, 
and hence that $J \in \Int(\bM_{1};w_{1},i_{1}) \cap
\Int(\bM_{2};w_{2},i_{2}) \cap \cdots \cap
\Int(\bM_{a};w_{a},i_{a})$.
\end{rem}

For each $\bM=(M_{\gamma})_{\gamma \in \Gamma_{\BZ}} \in \bz_{\BZ}$, 
we define a collection 
$\Theta(\bM)=(M_{\xi})_{\xi \in \Xi_{\BZ}}$ 
of integers indexed by 
$\Xi_{\BZ}=-\Gamma_{\BZ}$ as follows. 
Fix $\xi \in \Xi_{\BZ}$, and write it as 
$\xi=w\Lambda_{i}$ for some 
$w \in W_{\BZ}$ and $i \in \BZ$.
Here we note that 
if $I_{1},\,I_{2} \in \Int (\bM;w,i)$, then 
$M_{w\vpi_{i}^{I_{1}}}=M_{w\vpi_{i}^{I_{2}}}$. 
Indeed, take an interval $J$ in $\BZ$ such that 
$J \supset I_{1} \cup I_{2}$. Then we have 
$M_{w\vpi_{i}^{I_{1}}}=M_{w\vpi_{i}^{J}}=
 M_{w\vpi_{i}^{I_{2}}}$, and hence 
$M_{w\vpi_{i}^{I_{1}}}=M_{w\vpi_{i}^{I_{2}}}$. 
We now define $M_{\xi}=M_{w\Lambda_{i}}:=M_{w\vpi_{i}^{I}}$
for $I \in \Int (\bM;w,i)$. 
Let us check that 
this definition of $M_{\xi}$ does not depend 
on the choice of an expression $\xi=w\Lambda_{i}$. 
Suppose that $\xi=w\Lambda_{i}=v\Lambda_{j}$
for some $w,\,v \in W_{\BZ}$ and $i,\,j \in \BZ$; 
it is obvious that $i=j$ since $\Lambda_{i}$ and $\Lambda_{j}$ 
are dominant integral weights. Take an interval $I$ in 
$\BZ$ such that $I \in \Int(\bM;w,i) \cap \Int(\bM;v,j)$ 
(see Remark~\ref{rem:limit}\,(2)). 
Then, since $w,\,v \in W_{I}$ and 
$w\Lambda_{i}=v\Lambda_{j}$, 
the same argument as for the correspondence 
\eqref{eq:bij-index1} shows that 
$w\vpi_{i}^{I}=v\vpi_{j}^{I}$. 
Therefore, we obtain 
$M_{w\Lambda_{i}}=M_{w\vpi_{i}^{I}}=
 M_{v\vpi_{j}^{I}}=M_{v\Lambda_{j}}$, 
as desired. 

%==============================%
%     START SUBSECTION 0303    %
%==============================%
%
\subsection{Kashiwara operators on the set of BZ data of type $A_{\infty}$.}
\label{subsec:kashiwara}

Let $\bM=(M_{\gamma})_{\gamma \in \Gamma_{\BZ}} \in \bz_{\BZ}$, 
and fix $p \in \BZ$. We define 
$f_{p}\bM=(M_{\gamma}')_{\gamma \in \Gamma_{\BZ}}$ as follows.
For each $\gamma \in \Gamma_{\BZ}$, take an interval $I$ 
in $\BZ$ such that 
%
%%%%%%%%%%%%%%%%%
%%% eq:fj-int %%%
%%%%%%%%%%%%%%%%%
%
\begin{equation} \label{eq:fj-int}
\gamma \in \Gamma_{I} \quad \text{and} \quad 
I \in \Int(\bM;e,p) \cap \Int(\bM;s_{p},p);
\end{equation}
since $\bM_{I} \in \bz_{I}$ 
by condition~(a) of Definition~\ref{dfn:BZdatum2}, 
we can apply the lowering Kashiwara operator $f_{p}$ on 
$\bz_{I}$ to $\bM_{I}$. 
We define $(f_{p}\bM)_{\gamma}=
M_{\gamma}'$ to be $(f_{p}\bM_{I})_{\gamma}$. 
It follows from \eqref{eq:AM} that 
\begin{equation*}
M_{\gamma}' =
\begin{cases}
 \min \bigl(
   M_{\gamma},\ 
   M_{s_{p}\gamma}+c_{p}(\bM_{I})
   \bigr)
   & \text{if $\pair{h_{p}}{\gamma} > 0$}, \\[1.5mm]
 M_{\gamma} & \text{otherwise},
\end{cases}
\end{equation*}
where $c_{p}(\bM_{I})=M_{\vpi_{p}^{I}}-M_{s_{p}\vpi_{p}^{I}}-1$. 
Since $I \in \Int(\bM;e,p) \cap \Int(\bM;s_{p},p)$, 
we have 
\begin{equation*}
c_{p}(\bM_{I})=
 M_{\vpi_{p}^{I}}-M_{s_{p}\vpi_{p}^{I}}-1=
 M_{\Lambda_{p}}-M_{s_{p}\Lambda_{p}}-1=:c_{p}(\bM), 
\end{equation*}
where $M_{\Lambda_{p}}:=\Theta(\bM)_{\Lambda_{p}}$, and 
$M_{s_{p}\Lambda_{p}}:=\Theta(\bM)_{s_{p}\Lambda_{p}}$. Thus, 
%
%%%%%%%%%%%%%%
%%% eq:AM2 %%%
%%%%%%%%%%%%%%
%
\begin{equation} \label{eq:AM2}
M_{\gamma}'=
\begin{cases}
 \min \bigl(
   M_{\gamma},\ 
   M_{s_{p}\gamma}+c_{p}(\bM)
   \bigr)
   & \text{if $\pair{h_{p}}{\gamma} > 0$}, \\[1.5mm]
 M_{\gamma} & \text{otherwise}. 
\end{cases}
\end{equation}
From this description, we see that
the definition of $M_{\gamma}'$ 
does not depend on the choice of an interval $I$ 
satisfying \eqref{eq:fj-int}. 
%
%%%%%%%%%%%%%%%
%%% rem:AM2 %%%
%%%%%%%%%%%%%%%
%
\begin{rem} \label{rem:AM2}
(1) Keep the notation and assumptions above. 
It follows from \eqref{eq:AM2} that 
$M_{\gamma}'=(f_{p}\bM)_{\gamma} \le M_{\gamma}$ 
for all $\gamma \in \Gamma_{\BZ}$. 

(2) For $\bM \in \bz_{\BZ}$ and $p \in I$, there holds 
%
%%%%%%%%%%%%%%%%%
%%% eq:res-fj %%%
%%%%%%%%%%%%%%%%%
%
\begin{equation} \label{eq:res-fj}
(f_{p}\bM)_{I}=f_{p}\bM_{I}%
\quad \text{if} \quad 
I \in \Int(\bM;e,p) \cap 
 \Int(\bM;s_{p},p).
\end{equation}

\end{rem}

%%%%%%%%%%%%%%%
%%% prop:fj %%%
%%%%%%%%%%%%%%%
%
\begin{prop} \label{prop:fj}
Let $\bM \in \bz_{\BZ}$, and $p \in \BZ$. 
Then, $f_{p}\bM$ is an element of $\bz_{\BZ}$. 
\end{prop}

By this proposition, for each $p \in \BZ$, 
we obtain a map $f_{p}$ from $\bz_{\BZ}$ to itself sending 
$\bM \in \bz_{\BZ}$ to $f_{p}\bM \in \bz_{\BZ}$, 
which we call the lowering Kashiwara operator on $\bz_{\BZ}$. 

\begin{proof}[Proof of Proposition~\ref{prop:fj}]
First we show that $f_{p}\bM$ satisfies condition~(a) of 
Definition~\ref{dfn:BZdatum2}. 
Let $K$ be an interval in $\BZ$. 
Take an interval $I$ in $\BZ$ such that 
$K \subset I$ and $I \in \Int(\bM;e,p) \cap 
\Int(\bM;s_{p},p)$. Then, by \eqref{eq:res-fj}, 
we have $(f_{p}\bM)_{I}=f_{p}\bM_{I} \in \bz_{I}$. 
Also, it follows from Lemma~\ref{lem:res1} that 
$\bigl((f_{p}\bM)_{I}\bigr)_{K}=(f_{p}\bM_{I})_{K} \in \bz_{K}$. 
Since $\bigl((f_{p}\bM)_{I}\bigr)_{K}=(f_{p}\bM)_{K}$, 
we conclude that $(f_{p}\bM)_{K} \in \bz_{K}$, as desired. 

Next we show that $f_{p}\bM$ satisfies condition~(b) of 
Definition~\ref{dfn:BZdatum2}.
Write $\bM \in \bz_{\BZ}$ and $f_{p}\bM$ as: 
$\bM=(M_{\gamma})_{\gamma \in \Gamma_{\BZ}}$ and 
$f_{p}\bM=(M_{\gamma}')_{\gamma \in \Gamma_{\BZ}}$. 
Fix $w \in W_{\BZ}$ and $i \in \BZ$, and take 
an interval $I$ in $\BZ$ such that 
%
%%%%%%%%%%%%%%%%%%
%%% eq:fj-int1 %%%
%%%%%%%%%%%%%%%%%%
%
\begin{equation} \label{eq:fj-int1}
I \in 
\Int(\bM;e,p) \cap \Int(\bM;s_{p},p) \cap 
\Int(\bM;w,i) \cap \Int(\bM;s_{p}w,i). 
\end{equation}
Then, by \eqref{eq:AM2}, we have 
\begin{equation*}
M_{w\vpi_{i}^{I}}'=
\begin{cases}
 \min \bigl(
   M_{w\vpi_{i}^{I}},\ 
   M_{s_{p}w\vpi_{i}^{I}}+c_{p}(\bM)
   \bigr)
   & \text{if $\pair{h_{p}}{w\vpi_{i}^{I}} > 0$}, \\[1.5mm]
 M_{w\vpi_{i}^{I}} & \text{otherwise}. 
\end{cases}
\end{equation*}
Now, let $J$ be an interval in $\BZ$ containing $I$. 
Then, $J$ is also an element of 
the intersection in \eqref{eq:fj-int1} 
(see Remark~\ref{rem:limit}\,(1)). 
Therefore, again by \eqref{eq:AM2}, 
\begin{equation*}
M_{w\vpi_{i}^{J}}'=
\begin{cases}
 \min \bigl(
   M_{w\vpi_{i}^{J}},\ 
   M_{s_{p}w\vpi_{i}^{J}}+c_{p}(\bM)
   \bigr)
   & \text{if $\pair{h_{p}}{w\vpi_{i}^{J}} > 0$}, \\[1.5mm]
 M_{w\vpi_{i}^{J}} & \text{otherwise}. 
\end{cases}
\end{equation*}
Since $I \in \Int(\bM;w,i)$ 
(resp., $I \in \Int(\bM;s_{p}w,i)$) 
and $J \supset I$, it follows from the definition that 
$M_{w\vpi_{i}^{J}}=M_{w\vpi_{i}^{I}}$ 
(resp., $M_{s_{p}w\vpi_{i}^{J}}=M_{s_{p}w\vpi_{i}^{I}}$).
Also, since $w \in W_{I}$ and $p \in I$, 
we see that $w^{-1}h_{p} \in 
\bigoplus_{j \in I} \BZ h_{j} \subset 
\bigoplus_{j \in J} \BZ h_{j}$. 
Hence it follows from \eqref{eq:pair-vpi} that 
\begin{equation*}
\pair{h_{p}}{w\vpi_{i}^{I}}=
\pair{w^{-1}h_{p}}{\vpi_{i}^{I}}=
\pair{w^{-1}h_{p}}{\vpi_{i}^{J}}=
\pair{h_{p}}{w\vpi_{i}^{J}}.
\end{equation*}
In particular, 
$\pair{h_{p}}{w\vpi_{i}^{I}} > 0$ if and only if
$\pair{h_{p}}{w\vpi_{i}^{J}} > 0$. 
Consequently, we obtain 
$M_{w\vpi_{i}^{J}}'=M_{w\vpi_{i}^{I}}'$, 
which shows that $f_{p}\bM=(M_{\gamma}')_{\gamma \in \Gamma_{\BZ}}$ 
satisfies condition~(b) of Definition~\ref{dfn:BZdatum2}, as desired. 
Thus, we have proved that $f_{p}\bM \in \bz_{\BZ}$, 
thereby completing the proof of the proposition.
\end{proof}
%
%%%%%%%%%%%%%%
%%% rem:fj %%%
%%%%%%%%%%%%%%
%
\begin{rem} \label{rem:fj}
Let $\bM \in \bz_{\BZ}$, and fix $p \in \BZ$. Also, 
let $w \in W_{\BZ}$ and $i \in \BZ$. 
The proof of Proposition~\ref{prop:fj} shows that 
if an interval $I$ in $\BZ$ is an element of 
the intersection 
\begin{equation*}
\Int(\bM;e,p) \cap 
\Int(\bM;s_{p},p) \cap 
\Int(\bM;w,i) \cap 
\Int(\bM;s_{p}w,i), 
\end{equation*}
then $I$ is an element of $\Int(f_{p}\bM; w,i)$. 
\end{rem}

For intervals $I$, $K$ in $\BZ$ such that $I \supset K$, 
let $\bz_{\BZ}(I,\,K)$ denote the subset of $\bz_{\BZ}$ 
consisting of all elements 
$\bM \in \bz_{\BZ}$ such that $I \in \Int(\bM;v,k)$ 
for every $v \in W_{K}$ and $k \in K$; 
note that $\bz_{\BZ}(I,\,K)$ is nonempty 
since $\bO \in \bz_{\BZ}(I,\,K)$ 
(for the definition of $\bO$, see Example~\ref{ex:O}). 
%
%%%%%%%%%%%%%%%%%
%%% lem:IK-fj %%%
%%%%%%%%%%%%%%%%%
%
\begin{lem} \label{lem:IK-fj}
Keep the notation above. 

{\rm (1)} The set $\bz_{\BZ}(I,\,K)$ is 
stable under the lowering Kashiwara operators $f_{p}$ for $p \in K$. 

{\rm (2)} Let $\bM \in \bz_{\BZ}(I,\,K)$, and 
$p_{1},\,p_{2},\,\dots,\,p_{a} \in K$. Then, 
%
%%%%%%%%%%%%%%%%%%
%%% eq:res-fj2 %%%
%%%%%%%%%%%%%%%%%%
%
\begin{equation} \label{eq:res-fj2}
(f_{p_{a}}f_{p_{a-1}} \cdots f_{p_{1}}\bM)_{I}=
f_{p_{a}}f_{p_{a-1}} \cdots f_{p_{1}}\bM_{I}.
\end{equation}
\end{lem}

\begin{proof}
(1) Let $\bM \in \bz_{\BZ}(I,\,K)$, and $p \in K$. 
We show that $I \in \Int(f_{p}\bM;v,k)$ 
for all $v \in W_{K}$ and $k \in K$. 
Fix $v \in W_{K}$ and $k \in K$. 
Since the interval $I$ is an element of 
the intersection 
\begin{equation*}
\Int(\bM;e,p) \cap 
\Int(\bM;s_{p},p) \cap 
\Int(\bM;v,k) \cap 
\Int(\bM;s_{p}v,k), 
\end{equation*}
it follows from Remark~\ref{rem:fj} that 
$I \in \Int(f_{p}\bM;v,k)$. 
This proves part (1). 

(2) We show formula \eqref{eq:res-fj2} by induction on $a$. 
Assume first that $a=1$. Since $I \in \Int(\bM;e,p) \cap \Int(\bM;s_{p},p)$ 
for all $p \in K$, it follows from \eqref{eq:res-fj} that 
$(f_{p_1}\bM)_{I}=f_{p_1}\bM_{I}$. Assume next that $a > 1$. 
We set $\bM':=f_{p_{a-1}} \cdots f_{p_{1}}\bM$. 
Because $\bM' \in \bz_{\BZ}(I,\,K)$ by part (1), 
we see by the same argument as above that
$(f_{p_{a}}f_{p_{a-1}} \cdots f_{p_{1}}\bM)_{I} = 
(f_{p_a}\bM')_{I}=f_{p_a}\bM'_{I}$. 
Also, by the induction hypothesis, 
$\bM'_{I}
 =(f_{p_{a-1}} \cdots f_{p_{1}}\bM)_{I}
 =f_{p_{a-1}} \cdots f_{p_{1}}\bM_{I}$.
Combining these, we obtain 
$(f_{p_{a}}f_{p_{a-1}} \cdots f_{p_{1}}\bM)_{I}= 
 f_{p_{a}}f_{p_{a-1}} \cdots f_{p_{1}}\bM_{I}$, 
as desired. This proves part (2). 
\end{proof}

%%%%%%%%%%%

For $\bM=(M_{\gamma})_{\gamma \in \Gamma_{\BZ}} \in \bz_{\BZ}$
and $p \in \BZ$, we set 
%
%%%%%%%%%%%%%%%%%%%
%%% eq:dfn-vep2 %%%
%%%%%%%%%%%%%%%%%%%
%
\begin{equation} \label{eq:dfn-vep2}
\ve_{p}(\bM):= - \left(
     M_{\Lambda_{p}}+M_{s_{p}\Lambda_{p}}
     +\sum_{q \in \BZ \setminus \{p\}} a_{qp} M_{\Lambda_{q}}
    \right),
\end{equation}
where $M_{\Lambda_{i}}:=\Theta(\bM)_{\Lambda_{i}}$ for $i \in \BZ$, 
and $M_{s_{p}\Lambda_{p}}:=\Theta(\bM)_{s_{p}\Lambda_{p}}$. 
Note that $\ve_{p}(\bM)$ is a nonnegative integer. 
Indeed, let $I$ be an interval in $\BZ$ such that 
\begin{equation*}
I \in \Int(\bM;e,p) \cap 
\Int(\bM;s_{p},p) \cap \Int(\bM;e,p+1) \cap 
\Int(\bM;e,p-1).
\end{equation*}
Then, we have
\begin{align}
\ve_{p}(\bM) & = - \left(
     M_{\Lambda_{p}}+M_{s_{p}\Lambda_{p}}-
     M_{\Lambda_{p-1}}-M_{\Lambda_{p+1}}
    \right) \nonumber \\[1.5mm]
& = - \left(
     M_{\vpi_{p}^{I}}+M_{s_{p}\vpi_{p}^{I}}-
     M_{\vpi_{p-1}^{I}}-M_{\vpi_{p+1}^{I}}
     \right) \nonumber \\[1.5mm]
& = - \left(
     M_{\vpi_{p}^{I}}+M_{s_{p}\vpi_{p}^{I}}
     +\sum_{q \in I \setminus \{p\}} a_{qp} M_{\vpi_{q}^{I}}
    \right) = \ve_{p}(\bM_{I}). \label{eq:vej}
\end{align}
Hence it follows from 
condition~(a) of Definition~\ref{dfn:BZdatum2} and 
the comment following \eqref{eq:dfn-vep}
that $\ve_{p}(\bM) = \ve_{p}(\bM_{I})$ 
is a nonnegative integer.

Now, for $\bM=(M_{\gamma})_{\gamma \in \Gamma_{\BZ}} \in \bz_{\BZ}$ 
and $p \in \BZ$, we define $e_{p}\bM$ as follows. 
If $\ve_{p}(\bM)=0$, then 
we set $e_{p}\bM:=\bzero$, 
where $\bzero$ is an additional element, 
which is not contained in $\bz_{\BZ}$.
If $\ve_{p}(\bM) > 0$, then 
we define $e_{p}\bM=(M_{\gamma}')_{\gamma \in \Gamma_{\BZ}}$ 
as follows.
For each $\gamma \in \Gamma_{\BZ}$, 
take an interval $I$ in $\BZ$ such that 
%
%%%%%%%%%%%%%%%%%
%%% eq:ej-int %%%
%%%%%%%%%%%%%%%%%
%
\begin{equation} \label{eq:ej-int}
\begin{array}{l}
\gamma \in \Gamma_{I} \quad \text{and} \\[3mm]
I \in \Int(\bM;e,p) \cap \Int(\bM;s_{p},p) \cap 
\Int(\bM;e,p-1) \cap \Int(\bM;e,p+1); 
\end{array}
\end{equation}
note that $\min I < p < \max I$, since $p-1,\,p+1 \in I$. 
Consider $\bM_{I} \in \bz_{I}$ (see condition~(a) of 
Definition~\ref{dfn:BZdatum2}); 
since $\ve_{p}(\bM)=\ve_{p}(\bM_{I})$ by \eqref{eq:vej}, 
we have $\ve_{p}(\bM_{I}) > 0$, 
which implies that $e_{p}\bM_{I} \ne \bzero$. 
We define $(e_{p}\bM)_{\gamma}=
M_{\gamma}'$ to be $(e_{p}\bM_{I})_{\gamma}$. 
By virtue of the following lemma, 
this definition of $M_{\gamma}'$ 
does not depend on the choice of an interval $I$ 
satisfying \eqref{eq:ej-int}. 

\begin{lem}
Keep the notation and assumptions above. 
Assume that an interval $J$ in $\BZ$ satisfies 
the condition \eqref{eq:ej-int} with $I$ replaced by $J$. 
Then, we have $(e_{p}\bM_{J})_{\gamma}=(e_{p}\bM_{I})_{\gamma}$.
\end{lem}

\begin{proof}
We may assume from the beginning that $J \supset I$. 
Indeed, let $K$ be an interval in $\BZ$ containing 
both of the intervals $J$ and $I$. Then we see from 
Remark~\ref{rem:limit}\,(1) that $K$ satisfies 
the condition \eqref{eq:ej-int} with $I$ replaced by $K$.
If the assertion is true for $K$, then we have
$(e_{p}\bM_{K})_{\gamma}=(e_{p}\bM_{I})_{\gamma}$ and 
$(e_{p}\bM_{K})_{\gamma}=(e_{p}\bM_{J})_{\gamma}$, 
and hence $(e_{p}\bM_{J})_{\gamma}=(e_{p}\bM_{I})_{\gamma}$. 

We may further assume that 
$J=I \cup \bigl\{\max I + 1\bigr\}$ or 
$J=I \cup \bigl\{\min I - 1\bigr\}$; 
for simplicity of notation, suppose that 
$I=\bigl\{1,\,2,\,\dots,\,m\bigr\}$ and 
$J=\bigl\{1,\,2,\,\dots,\,m,\,m+1\bigr\}$. 
Note that $1=\min I < p < \max I=m$ 
(see the comment preceding this proposition).

We write $e_{p}\bM_{I} \in \bz_{I}$ and 
$e_{p}\bM_{J} \in \bz_{J}$ as: 
$e_{p}\bM_{I}=(M_{\gamma}')_{\gamma \in \Gamma_{I}}$ and
$e_{p}\bM_{J}=(M_{\gamma}'')_{\gamma \in \Gamma_{J}}$, 
respectively; we need to show that $M_{\gamma}''=M_{\gamma}'$ 
for all $\gamma \in \Gamma_{I}$. 
Recall that 
$e_{p}\bM_{I}=(M_{\gamma}')_{\gamma \in \Gamma_{I}}$
is defined to be the unique BZ datum for $\Fg_{I}$ 
such that $M_{\vpi_{p}^{I}}'=M_{\vpi_{p}^{I}}+1$, and 
such that $M_{\gamma}'=M_{\gamma}$ for all $\gamma \in \Gamma_{I}$ 
with $\pair{h_{p}}{\gamma} \le 0$ (see Fact~\ref{fact:ej}). 
It follows from Lemma~\ref{lem:res1} that 
$(e_{p}\bM_{J})_{I}=(M_{\gamma}'')_{\gamma \in \Gamma_{I}}$ 
is a BZ datum for $\Fg_{I}$. Also, we see from 
the definition of $e_{p}\bM_{J}$ that 
$M_{\gamma}''=M_{\gamma}$ for all 
$\gamma \in \Gamma_{I} \subset \Gamma_{J}$ 
with $\pair{h_{p}}{\gamma} \le 0$. 
Therefore, if we can show the equality 
$M_{\vpi_{p}^{I}}''=M_{\vpi_{p}^{I}}+1$, 
then it follows from the uniqueness that 
$(e_{p}\bM_{J})_{I}=(M_{\gamma}'')_{\gamma \in \Gamma_{I}}$ 
is equal to 
$e_{p}\bM_{I}=(M_{\gamma}')_{\gamma \in \Gamma_{I}}$, 
and hence $M_{\gamma}''=M_{\gamma}'$ 
for all $\gamma \in \Gamma_{I}$, as desired. 
We will show that 
$M_{\vpi_{p}^{I}}''=M_{\vpi_{p}^{I}}+1$. 

First, let us verify the following formula: 
%
%%%%%%%%%%%%%%%%%
%%% eq:vpi-IJ %%%
%%%%%%%%%%%%%%%%%
%
\begin{equation} \label{eq:vpi-IJ}
\vpi_{k}^{I}=
 s_{m+1} \cdots s_{k+2}s_{k+1}(\vpi_{k+1}^{J})
\quad \text{for $1 \le k \le m$}. 
\end{equation}
Indeed, we have 
\begin{align*}
\vpi_{k}^{I} & =-w_{0}^{I}(\Lambda_{\omega_{I}(k)})
=-w_{0}^{I}(\Lambda_{m-k+1}) \\
& 
=-w_{0}^{I}w_{0}^{J}w_{0}^{J}(\Lambda_{m-k+1})
=w_{0}^{I}w_{0}^{J}(\vpi_{\omega_{J}(m-k+1)}^{J})
=w_{0}^{I}w_{0}^{J}(\vpi_{k+1}^{J}). 
\end{align*}
Consequently, by using the reduced expressions
\begin{align*}
& w_{0}^{J}=s_{1}(s_{2}s_{1})(s_{3}s_{2}s_{1}) \cdots 
  (s_{m} \cdots s_{2}s_{1})(s_{m+1} \cdots s_{2}s_{1}), \\
& w_{0}^{I}=(s_{m} \cdots s_{2}s_{1}) \cdots 
  (s_{1}s_{2}s_{3})(s_{1}s_{2})s_{1},
\end{align*}
we see that
$\vpi_{k}^{I}=
 s_{m+1} \cdots s_{2}s_{1}(\vpi_{k+1}^{J})=
 s_{m+1} \cdots s_{k+2}s_{k+1}(\vpi_{k+1}^{J})$,
as desired. 

Now, let us show that 
$M_{\vpi_{p}^{I}}''=M_{\vpi_{p}^{I}}+1$. 
We set $w:=s_{m+1} \cdots s_{p+3}s_{p+2} \in W_{J}$. 
Then, $a_{p,p+1}=a_{p+1,p}=-1$ and 
$ws_{p+1} > w$, $ws_{p} > w$. 
Therefore, since $e_{p}\bM_{J}=(M_{\gamma}'')_{\gamma \in \Gamma_{J}} \in \bz_{J}$, 
it follows from condition~(2) of Definition~\ref{dfn:BZdatum} that 
%
%%%%%%%%%%%%%%%
%%% eq:tpr1 %%%
%%%%%%%%%%%%%%%
%
\begin{equation} \label{eq:tpr1}
M_{ws_{p+1}\vpi_{p+1}^{J}}''+M_{ws_{p}\vpi_{p}^{J}}'' =
\min \bigl(
 M_{w\vpi_{p+1}^{J}}''+M_{ws_{p+1}s_{p}\vpi_{p}^{J}}'', \ 
 M_{w\vpi_{p}^{J}}''+M_{ws_{p}s_{p+1}\vpi_{p+1}^{J}}''
 \bigr).
\end{equation}
Also, by using \eqref{eq:vpi-IJ} and the facts that 
$s_{q}\vpi_{p}^{J}=\vpi_{p}^{J}$, $s_{q}\vpi_{p+1}^{J}=\vpi_{p+1}^{J}$ 
for $p+2 \le q \le m+1$ and that $s_{q}s_{p}=s_{p}s_{q}$ 
for $p+2 \le q \le m+1$, we get 
\begin{align*}
& ws_{p+1}\vpi_{p+1}^{J}=
  s_{m+1} \cdots s_{p+2}s_{p+1}\vpi_{p+1}^{J}=
  \vpi_{p}^{I}, \\[1mm]
& ws_{p}\vpi_{p}^{J}=
  s_{m+1} \cdots s_{p+2}s_{p}\vpi_{p}^{J}=
  s_{p}s_{m+1} \cdots s_{p+2}\vpi_{p}^{J}=
  s_{p}\vpi_{p}^{J}, \\[1mm]
& w\vpi_{p+1}^{J}=s_{m+1} \cdots s_{p+2}\vpi_{p+1}^{J}
  =\vpi_{p+1}^{J}, \\[1mm]
& ws_{p+1}s_{p}\vpi_{p}^{J}=
  s_{m+1} \cdots s_{p+2}s_{p+1}s_{p}\vpi_{p}^{J}=
  \vpi_{p-1}^{I}, \\[1mm]
& w\vpi_{p}^{J}=s_{m+1} \cdots s_{p+2}\vpi_{p}^{J}
  =\vpi_{p}^{J}, \\[1mm]
& ws_{p}s_{p+1}\vpi_{p+1}^{J}=
  s_{m+1} \cdots s_{p+2}s_{p}s_{p+1}\vpi_{p+1}^{J}=
  s_{p}s_{m+1} \cdots s_{p+2}s_{p+1}\vpi_{p+1}^{J}=
  s_{p}\vpi_{p}^{I}.
\end{align*}
Hence the equation \eqref{eq:tpr1} can be rewritten as:
\begin{equation} \label{eq:tpr2}
M_{\vpi_{p}^{I}}''+M_{s_{p}\vpi_{p}^{J}}'' =
\min \bigl(
 M_{\vpi_{p+1}^{J}}''+M_{\vpi_{p-1}^{I}}'', \ 
 M_{\vpi_{p}^{J}}''+M_{s_{p}\vpi_{p}^{I}}''
 \bigr).
\end{equation}
Since $\pair{h_{p}}{s_{p}\vpi_{p}^{J}}=-1 < 0$, 
it follows from the definition of $e_{p}\bM_{J}$ that 
$M_{s_{p}\vpi_{p}^{J}}''=M_{s_{p}\vpi_{p}^{J}}$. 
Similarly, $M_{\vpi_{p+1}^{J}}''=M_{\vpi_{p+1}^{J}}$, 
$M_{\vpi_{p-1}^{J}}''=M_{\vpi_{p-1}^{J}}$, and 
$M_{s_{p}\vpi_{p}^{I}}''=M_{s_{p}\vpi_{p}^{I}}$. 
In addition, it follows from the definition of $e_{p}\bM_{J}$ that 
$M_{\vpi_{p}^{J}}''=M_{\vpi_{p}^{J}}+1$. 
Substituting these into \eqref{eq:tpr2}, we obtain 
\begin{equation} \label{eq:tpr3}
M_{\vpi_{p}^{I}}''+M_{s_{p}\vpi_{p}^{J}} =
\min \bigl(
 M_{\vpi_{p+1}^{J}}+M_{\vpi_{p-1}^{I}}, \ 
 M_{\vpi_{p}^{J}}+1+M_{s_{p}\vpi_{p}^{I}}
 \bigr).
\end{equation}
Here, observe that 
$M_{\vpi_{p-1}^{I}}=M_{\vpi_{p-1}^{J}}$ 
(resp., $M_{s_{p}\vpi_{p}^{I}}=M_{s_{p}\vpi_{p}^{J}}$) 
since $I \in \Int(\bM; e,p-1)$ 
(resp., $I \in \Int(\bM; s_{p},p)$) and $J \supset I$. 
As a result, we get
%
%%%%%%%%%%%%%%%
%%% eq:tpr4 %%%
%%%%%%%%%%%%%%%
%
\begin{equation} \label{eq:tpr4}
M_{\vpi_{p}^{I}}''+M_{s_{p}\vpi_{p}^{J}} =
\min \bigl(
 M_{\vpi_{p+1}^{J}}+M_{\vpi_{p-1}^{J}}, \ 
 M_{\vpi_{p}^{J}}+1+M_{s_{p}\vpi_{p}^{J}}
 \bigr).
\end{equation}
Moreover, since $\ve_{p}(\bM) > 0$ by assumption, we see 
from \eqref{eq:vej} with $I$ replaced by $J$ that 
$M_{\vpi_{p}^{J}}+M_{s_{p}\vpi_{p}^{J}} < 
 M_{\vpi_{p+1}^{J}}+M_{\vpi_{p-1}^{J}}$, 
which implies that
\begin{equation*}
\min \bigl(
 M_{\vpi_{p+1}^{J}}+M_{\vpi_{p-1}^{J}}, \ 
 M_{\vpi_{p}^{J}}+1+M_{s_{p}\vpi_{p}^{J}}
 \bigr)=
M_{\vpi_{p}^{J}}+1+M_{s_{p}\vpi_{p}^{J}}.
\end{equation*}
Combining this and \eqref{eq:tpr4}, we obtain 
$M_{\vpi_{p}^{I}}''=M_{\vpi_{p}^{J}}+1$.
Noting that $M_{\vpi_{p}^{J}}=M_{\vpi_{p}^{I}}$ 
since $I \in \Int(\bM;e,p)$ and $J \supset I$, 
we conclude that 
$M_{\vpi_{p}^{I}}''=M_{\vpi_{p}^{I}}+1$, 
as desired. This completes the proof of the lemma. 
\end{proof}

\begin{rem} 
(1) Let $\bM=(M_{\gamma})_{\gamma \in \Gamma_{\BZ}} \in \bz_{\BZ}$ 
and $p \in \BZ$ be such that $e_{p}\bM \ne \bzero$. Then, 
%
%%%%%%%%%%%%%%%%%%%
%%% eq:ep-gamma %%%
%%%%%%%%%%%%%%%%%%%
%
\begin{equation} \label{eq:ep-gamma}
(e_{p}\bM)_{\gamma}=M_{\gamma} 
\quad \text{for all $\gamma \in \Gamma_{\BZ}$ with 
$\pair{h_{p}}{\gamma} \le 0$}.
\end{equation}
Indeed, let $\gamma \in \Gamma_{\BZ}$ be such that 
$\pair{h_{p}}{\gamma} \le 0$. Take an interval $I$ in $\BZ$ 
satisfying the condition \eqref{eq:ej-int}. 
Then, by the definition, 
$(e_{p}\bM)_{\gamma}=(e_{p}\bM_{I})_{\gamma}$. 
Also, we see from the definition of $e_{p}$ on 
$\bz_{I}$ (see Fact~\ref{fact:ej}) that 
$(e_{p}\bM_{I})_{\gamma}=M_{\gamma}$. 
Hence we get $(e_{p}\bM)_{\gamma}=(e_{p}\bM_{I})_{\gamma}=
M_{\gamma}$, as desired. 

(2) For $\bM \in \bz_{\BZ}$ and $p \in \BZ$, there holds 
%
%%%%%%%%%%%%%%%%%
%%% eq:res-ej %%%
%%%%%%%%%%%%%%%%%
%
\begin{equation} \label{eq:res-ej}
\begin{array}{l}
(e_{p}\bM)_{I}=e_{p}\bM_{I} \\[3mm]
\text{if 
$I \in \Int(\bM;e,p) \cap 
 \Int(\bM;s_{p},p) \cap 
 \Int(\bM;e,p-1) \cap 
 \Int(\bM;e,p+1)$}.
\end{array}
\end{equation}
\end{rem}
%
%%%%%%%%%%%%%%%
%%% prop:ej %%%
%%%%%%%%%%%%%%%
%
\begin{prop} \label{prop:ej}
Let $\bM \in \bz_{\BZ}$, and $p \in \BZ$. 
Then, $e_{p}\bM$ is an element of 
$\bz_{\BZ} \cup \{\bzero\}$. 
\end{prop}

By this proposition, for each $p \in \BZ$, 
we obtain a map $e_{p}$ from $\bz_{\BZ}$ to 
$\bz_{\BZ} \cup \{\bzero\}$ sending 
$\bM \in \bz_{\BZ}$ to $e_{p}\bM \in \bz_{\BZ} \cup \{\bzero\}$, 
which we call the raising Kashiwara operator on $\bz_{\BZ}$. 
By convention, we set $e_{p}\bzero:=\bzero$ for all $p \in \BZ$, 
and $f_{p}\bzero:=\bzero$ for all $p \in \BZ$. 

\begin{proof}[Proof of Proposition~\ref{prop:ej}]
Assume that $e_{p}\bM \ne \bzero$. 
Using \eqref{eq:res-ej} instead of \eqref{eq:res-fj}, we can show 
by an argument (for $f_{p}\bM$) in the proof of Proposition~\ref{prop:fj}
that $e_{p}\bM$ satisfies condition~(a) of 
Definition~\ref{dfn:BZdatum2}. 
We will, therefore, show that 
$e_{p}\bM$ satisfies condition~(b) of 
Definition~\ref{dfn:BZdatum2}. 
We write $\bM$ and $e_{p}\bM$ as: 
$\bM=(M_{\gamma})_{\gamma \in \Gamma_{\BZ}}$ and 
$e_{p}\bM=(M_{\gamma}')_{\gamma \in \Gamma_{\BZ}}$, respectively. 
Fix $w \in W$ and $i \in \BZ$, 
and then fix an interval $K$ in $\BZ$ such that 
$w \in W_{K}$ and $i,\,p-1,\,p,\,p+1 \in K$. 
Now, take an interval $I$ in $\BZ$ such that 
$I \in \Int(\bM;v,k)$ 
for all $v \in W_{K}$ and $k \in K$ (see Remark~\ref{rem:limit}\,(2)); 
note that $I$ is an element of the intersection 
%
%%%%%%%%%%%%%%%%%%
%%% eq:ej-int2 %%%
%%%%%%%%%%%%%%%%%%
%
\begin{equation} \label{eq:ej-int2}
\Int(\bM;e,\,p) \cap 
\Int(\bM;s_{p},\,p) \cap \Int(\bM;e,\,p-1) \cap \Int(\bM;e,\,p+1),
\end{equation}
since $p-1,\,p,\,p+1 \in K$. 
We need to show that 
$M_{w\vpi_{i}^{J}}'=
 M_{w\vpi_{i}^{I}}'$ for all intervals $J$ in $\BZ$ 
containing $I$. 

Before we proceed further, we make some remarks: 
Through the bijections \eqref{eq:bij-index2} and \eqref{eq:bij-index3}, 
we can (and do) identify the set $\Gamma_{K}$ (of chamber weights) 
for $\Fg_{K}$ with the subset $\Gamma_{I}^{K}=
\bigl\{v\vpi_{k}^{I} \mid v \in W_{K},\,k \in K\bigr\} 
\subset \Gamma_{I} \subset \Gamma_{\BZ}$; note that 
$v\vpi_{k}^{K} \in \Gamma_{K}$ corresponds to 
$v\vpi_{k}^{I} \in \Gamma_{I}^{K}$ for $v \in W_{K}$ and $k \in K$. 
Let $J$ be an interval in $\BZ$ containing $I$. As above, 
we can (and do) identify the set $\Gamma_{K}$ (of chamber weights) 
for $\Fg_{K}$ with the subset $\Gamma_{J}^{K}=
\bigl\{v\vpi_{k}^{J} \mid v \in W_{K},\,k \in K\bigr\} 
\subset \Gamma_{J} \subset \Gamma_{\BZ}$; note that 
$v\vpi_{k}^{K} \in \Gamma_{K}$ corresponds to 
$v\vpi_{k}^{J} \in \Gamma_{J}^{K}$ for $v \in W_{K}$ and $k \in K$. 
Thus, the three sets $\Gamma_{J}^{K} \ (\subset \Gamma_{J} \subset \Gamma_{\BZ})$, 
$\Gamma_{I}^{K} \ (\subset \Gamma_{I} \subset \Gamma_{\BZ})$, and $\Gamma_{K}$ 
can be identified as follows: 
%
%%%%%%%%%%%%%%%%%%%%%
%%% eq:bij-index4 %%%
%%%%%%%%%%%%%%%%%%%%%
%
\begin{equation} \label{eq:bij-index4}
\begin{array}{ccccl}
\Gamma_{K} & \stackrel{\sim}{\rightarrow} & 
\Gamma_{J}^{K} & \stackrel{\sim}{\rightarrow} & \Gamma_{I}^{K}, \\[3mm]
v\vpi_{k}^{K} & \mapsto & 
v\vpi_{k}^{J} & \mapsto &  v\vpi_{k}^{I}
\quad \text{for $v \in W_{K}$ and $k \in K$}.
\end{array}
\end{equation}
Also, it follows from the definition of $\bz_{\BZ}$ that 
$\bM_{I}=(M_{\gamma})_{\gamma \in \Gamma_{I}} \in \bz_{I}$ and 
$\bM_{J}=(M_{\gamma})_{\gamma \in \Gamma_{J}} \in \bz_{J}$. 
Therefore, by Lemma~\ref{lem:res2}, 
$(\bM_{I})^{K}=(M_{\gamma})_{\gamma \in \Gamma_{I}^{K}}$ and 
$(\bM_{J})^{K}=(M_{\gamma})_{\gamma \in \Gamma_{J}^{K}}$ are 
BZ data for $\Fg_{K}$ if we identify the sets 
$\Gamma_{I}^{K}$ and $\Gamma_{J}^{K}$ with 
the set $\Gamma_{K}$
through the bijection \eqref{eq:bij-index4}. 
Since $M_{v\vpi_{k}^{J}}=M_{v\vpi_{k}^{I}}$ 
for all $v \in W_{K}$ and $k \in K$ by our assumption, 
we deduce that $(\bM_{J})^{K}=(\bM_{I})^{K}$
if we identify the three sets $\Gamma_{J}^{K}$, $\Gamma_{I}^{K}$, and 
$\Gamma_{K}$ as in \eqref{eq:bij-index4}. 

Now we are ready to show that 
$M_{w\vpi_{i}^{J}}'=
 M_{w\vpi_{i}^{I}}'$. 
By our assumption that $e_{p}\bM \ne \bzero$ and 
\eqref{eq:ej-int2}, it follows that 
$e_{p}\bM_{I} \ne \bzero$, and hence 
$e_{p}\bM_{I}$ is an element of $\bz_{I}$; 
we see 
from \eqref{eq:res-ej} that 
$e_{p}\bM_{I}=(e_{p}\bM)_{I}=
 (M_{\gamma}')_{\gamma \in \Gamma_{I}}$. 
Hence, by Lemma~\ref{lem:res2}, 
$(e_{p}\bM_{I})^{K}=(M_{\gamma}')_{\gamma \in \Gamma_{I}^{K}}$ is 
a BZ datum for $\Fg_{K}$ if we identify the set $\Gamma_{I}^{K}$ 
with the set $\Gamma_{K}$ through 
the bijection \eqref{eq:bij-index4}. 
Also, by the definition of $e_{p}\bM_{I}$, 
we see that $M_{\vpi_{p}^{I}}'=M_{\vpi_{p}^{I}}+1$, and 
$M_{v\vpi_{k}^{I}}'=M_{v\vpi_{k}^{I}}$ for all 
$v \in W_{K}$ and $k \in K$ with 
$\pair{h_{p}}{v\vpi_{k}^{I}} \le 0$. 
Here we observe that 
for $v \in W_{K}$ and $k \in K$, 
the inequality $\pair{h_{p}}{v\vpi_{k}^{I}} \le 0$ holds 
if and only if the inequality $\pair{h_{p}}{v\vpi_{k}^{K}} \le 0$ holds. 
Indeed, let $v \in W_{K}$, and $k \in K$. 
Note that $v^{-1}h_{p} \in 
\bigoplus_{j \in K} \BZ h_{j} \subset 
\bigoplus_{j \in I} \BZ h_{j}$ since $p \in K$ by our assumption. 
Hence it follows from \eqref{eq:pair-vpi} that
\begin{equation*}
\pair{h_{p}}{v\vpi_{k}^{I}}=
\pair{v^{-1}h_{p}}{\vpi_{k}^{I}}=
\pair{v^{-1}h_{p}}{\vpi_{k}^{K}}=
\pair{h_{p}}{v\vpi_{k}^{K}},
\end{equation*}
which implies that 
$\pair{h_{p}}{v\vpi_{k}^{I}} \le 0$ 
if and only if $\pair{h_{p}}{v\vpi_{k}^{K}} \le 0$. 
Therefore, we deduce from Fact~\ref{fact:ej} that 
$(e_{p}\bM_{I})^{K}=(M_{\gamma}')_{\gamma \in \Gamma_{I}^{K}}$
is equal to $e_{p}\bigl((\bM_{I})^{K}\bigr)$ 
if we identify $\Gamma_{I}^{K}$ and $\Gamma_{K}$
by \eqref{eq:bij-index4}. 
Furthermore, we see from Remark~\ref{rem:limit}\,(1) that 
the interval $J \supset I$ is also an element of 
$\Int(\bM;v,k)$ for all $v \in W_{K}$ and $k \in K$.
In exactly the same way as above 
(with $I$ replaced by $J$), we can show that 
$(e_{p}\bM_{J})^{K}=(M_{\gamma}')_{\gamma \in \Gamma_{J}^{K}}$
is a BZ datum for $\Fg_{K}$, and is equal to $e_{p}\bigl((\bM_{J})^{K}\bigr)$ 
if we identify $\Gamma_{J}^{K}$ and $\Gamma_{K}$ by \eqref{eq:bij-index4}. 
Since $(\bM_{I})^{K}=(\bM_{J})^{K}$ as seen above, 
we obtain $e_{p}\bigl((\bM_{I})^{K}\bigr)=e_{p}\bigl((\bM_{J})^{K}\bigr)$. 
Consequently, we infer that 
$(e_{p}\bM_{J})^{K}=(M_{\gamma}')_{\gamma \in \Gamma_{J}^{K}}$ 
is equal to $(e_{p}\bM_{I})^{K}=(M_{\gamma}')_{\gamma \in \Gamma_{I}^{K}}$ 
if we identify $\Gamma_{J}^{K}$ and $\Gamma_{I}^{K}$ by \eqref{eq:bij-index4}. 
Because $w\vpi_{i}^{J} \in \Gamma_{J}^{K}$ corresponds 
to $w\vpi_{i}^{I} \in \Gamma_{I}^{K}$
through the bijection \eqref{eq:bij-index4}, 
we finally obtain $M_{w\vpi_{i}^{J}}'=M_{w\vpi_{i}^{I}}'$, as desired. 
This completes the proof of the proposition.
\end{proof}
%
%%%%%%%%%%%%%%
%%% rem:ej %%%
%%%%%%%%%%%%%%
%
\begin{rem} \label{rem:ej}
Let $\bM \in \bz_{\BZ}$ and $p \in \BZ$ be 
such that $e_{p}\bM \ne \bzero$. 
Let $K$ be an interval in $\BZ$ 
such that $p-1,\,p,\,p+1 \in K$. 
The proof of Proposition~\ref{prop:ej} shows that 
if an interval $I$ in $\BZ$ is an element of 
$\Int(\bM;v,k)$ for all $v \in W_{K}$ and $k \in K$, 
then $I \in \Int(e_{p}\bM;v,k)$ for all 
$v \in W_{K}$ and $k \in K$. 
\end{rem}
%
%%%%%%%%%%%%%%%%%
%%% lem:IK-ej %%%
%%%%%%%%%%%%%%%%%
%
\begin{lem} \label{lem:IK-ej}
Let $I$ and $K$ be intervals in $\BZ$ such that $I \supset K$
and $\#K \ge 3$. 

{\rm (1)} The set $\bz_{\BZ}(I,\,K) \cup \{\bzero\}$ is 
stable under the raising Kashiwara operators $e_{p}$ for 
$p \in K$ with $\min K < p < \max K$. 

{\rm (2)} Let $\bM \in \bz_{\BZ}(I,\,K)$, and 
let $p_{1},\,p_{2},\,\dots,\,p_{a} \in K$ be such that 
$\min K < p_{1},\,p_{2},\,\dots,\,p_{a} < \max K$. 
Then, $e_{p_{a}}e_{p_{a-1}} \cdots e_{p_{1}}\bM \ne \bzero$
if and only if $e_{p_{a}}e_{p_{a-1}} \cdots e_{p_{1}}\bM_{I} \ne \bzero$.
Moreover, if $e_{p_{a}}e_{p_{a-1}} \cdots e_{p_{1}}\bM \ne \bzero$
(or equivalently, 
$e_{p_{a}}e_{p_{a-1}} \cdots e_{p_{1}}\bM_{I} \ne \bzero$), 
then 
%
%%%%%%%%%%%%%%%%%%
%%% eq:res-ej2 %%%
%%%%%%%%%%%%%%%%%%
%
\begin{equation} \label{eq:res-ej2}
(e_{p_{a}}e_{p_{a-1}} \cdots e_{p_{1}}\bM)_{I}=
e_{p_{a}}e_{p_{a-1}} \cdots e_{p_{1}}\bM_{I}.
\end{equation}
\end{lem}

\begin{proof}
Part (1) follows immediately from Remark~\ref{rem:ej}. 
We will show part (2) by induction on $a$. 
Assume first that $a=1$. Since $\bM \in \bz_{\BZ}(I,\,K)$, 
it follows immediately that 
\begin{equation*}
I \in \Int(\bM;e,p_{1}) \cap 
\Int(\bM;s_{p_{1}},p_{1}) \cap \Int(\bM;e,p_{1}+1) \cap 
\Int(\bM;e,p_{1}-1). 
\end{equation*}
Therefore, we have 
$\ve_{p_{1}}(\bM)=\ve_{p_{1}}(\bM_{I})$ by \eqref{eq:vej}, 
which implies that $e_{p_{1}}\bM \ne \bzero$ if and 
only if $e_{p_{1}}\bM_{I} \ne \bzero$. 
Also, it follows from \eqref{eq:res-ej} that 
if $e_{p_{1}}\bM \ne \bzero$, then 
$(e_{p_{1}}\bM)_{I}=e_{p_{1}}\bM_{I}$. 

Assume next that $a > 1$. For simplicity of notation, we set
\begin{equation*}
\bM':=e_{p_{a-1}} \cdots e_{p_{1}}\bM
\quad \text{and} \quad
\bM'':=e_{p_{a-1}} \cdots e_{p_{1}}\bM_{I}.
\end{equation*}
Let us show that $e_{p_{a}}\bM' \ne \bzero$ 
if and only if $e_{p_{a}}\bM'' \ne \bzero$. 
By the induction hypothesis, we may assume that 
$\bM' \ne \bzero$, $\bM'' \ne \bzero$, and $\bM_{I}'=\bM''$. 
It follows from part (1) that 
$\bM' \in \bz_{\BZ}(I,\,K)$. 
Hence, by the same argument as above (the case $a=1$), 
we deduce that $e_{p_{a}}\bM' \ne \bzero$ if and 
only if $e_{p_{a}}\bM'_{I} \ne \bzero$, which implies that 
$e_{p_{a}}\bM' \ne \bzero$ if and 
only if $e_{p_{a}}\bM'' \ne \bzero$. 
Furthermore, it follows from \eqref{eq:res-ej} that 
if $e_{p_{a}}\bM' \ne \bzero$, then 
$(e_{p_{a}}\bM')_{I}=e_{p_{a}}\bM'_{I}=e_{p_{a}}\bM''$. 
This proves the lemma.
\end{proof}

%==============================%
%     START SUBSECTION 0304    %
%==============================%
%
\subsection{Some properties of Kashiwara operators on $\bz_{\BZ}$.}
\label{subsec:properties}

%%%%%%%%%%%%%%
%%% lem:ef %%%
%%%%%%%%%%%%%%
%
\begin{lem} \label{lem:ef}
{\rm (1)} Let $\bM \in \bz_{\BZ}$, and $p \in \BZ$. 
Then, $e_{p}f_{p}\bM=\bM$. Also, if $e_{p}\bM \ne \bzero$, 
then $f_{p}e_{p}\bM=\bM$. 

{\rm (2)} Let $\bM \in \bz_{\BZ}$, 
and let $p,\,q \in \BZ$ be such that $|p-q| \ge 2$. 
Then, $\ve_{p}(f_{p}\bM)=\ve_{p}(\bM)+1$ and $\ve_{q}(f_{p}\bM)=\ve_{q}(\bM)$. 
Also, if $e_{p}\bM \ne \bzero$, then
$\ve_{p}(e_{p}\bM)=\ve_{p}(\bM)-1$ and 
$\ve_{q}(e_{p}\bM)=\ve_{q}(\bM)$. 

{\rm (3)}
Let $p,\,q \in \BZ$ be such that $|p-q| \ge 2$. 
Then, $f_{p}f_{q}=f_{q}f_{p}$, $e_{p}e_{q}=e_{q}e_{p}$, 
and $e_{p}f_{q}=f_{q}e_{p}$ on $\bz_{\BZ} \cup \{\bzero\}$. 
\end{lem}

\begin{proof}
(1) We prove that $e_{p}f_{p}\bM=\bM$; 
by a similar argument, we can prove that 
$f_{p}e_{p}\bM=\bM$ if $e_{p}\bM \ne \bzero$.
We need to show that $e_{p}f_{p}\bM \ne \bzero$, and that 
the $\gamma$-component of $e_{p}f_{p}\bM$ is equal to 
that of $\bM$ for each $\gamma \in \Gamma_{\BZ}$. 
We fix $\gamma \in \Gamma_{\BZ}$. 
Set $K:=\bigl\{p-1,\,p,\,p+1\bigr\}$, and 
take an interval $I$ in $\BZ$ such that 
$\gamma \in \Gamma_{I}$, and such that 
$I \in \Int(\bM;v,k)$ for all $v \in W_{K}$ and $k \in K$. 
Then, we have $\bM \in \bz_{\BZ}(I,\,K)$, and hence 
we see from Lemma~\ref{lem:IK-fj} that 
$f_{p}\bM \in \bz_{\BZ}(I,\,K)$ and 
$(f_{p}\bM)_{I}=f_{p}\bM_{I}$. 
Because $e_{p}(f_{p}\bM)_{I}=e_{p}(f_{p}\bM_{I})=\bM_{I} \ne \bzero$ 
by condition~(a) of Definition~\ref{dfn:BZdatum2} and 
Theorem~\ref{thm:binf}, it follows from 
Lemma~\ref{lem:IK-ej}\,(2) that $e_{p}f_{p}\bM \ne \bzero$. 
Also, we deduce from Lemmas~\ref{lem:IK-fj}\,(2) and 
\ref{lem:IK-ej}\,(2) that 
$(e_{p}f_{p}\bM)_{I}=e_{p}f_{p}\bM_{I}=\bM_{I}$. 
Since $\gamma \in \Gamma_{I}$ by our assumption on $I$, 
we infer that the $\gamma$-component of $e_{p}f_{p}\bM$ 
is equal to that of $\bM$. This proves part (1). 

(2) We give a proof only for the equalities 
$\ve_{p}(f_{p}\bM)=\ve_{p}(\bM)+1$ and 
$\ve_{q}(f_{p}\bM)=\ve_{q}(\bM)$; 
by a similar argument, we can prove that 
$\ve_{p}(e_{p}\bM)=\ve_{p}(\bM)-1$ and 
$\ve_{q}(e_{p}\bM)=\ve_{q}(\bM)$ 
if $e_{p}\bM \ne \bzero$.
Write $\bM \in \bz_{\BZ}$ and $f_{p}\bM \in \bz_{\BZ}$ as: 
$\bM=(M_{\gamma})_{\gamma \in \Gamma_{\BZ}}$ and 
$f_{p}\bM=(M_{\gamma}')_{\gamma \in \Gamma_{\BZ}}$, 
respectively. 
Also, write 
$\Theta(\bM)$ and $\Theta(f_{p}\bM)$ as: 
$\Theta(\bM)=(M_{\xi})_{\xi \in \Xi_{\BZ}}$ and 
$\Theta(f_{p}\bM)=(M_{\xi}')_{\xi \in \Xi_{\BZ}}$, 
respectively. 
First we show that for $i \in \BZ$, 
%
%%%%%%%%%%%%%%%%
%%% eq:lam-i %%%
%%%%%%%%%%%%%%%%
%
\begin{equation} \label{eq:lam-i}
M_{\Lambda_{i}}'=
\begin{cases}
M_{\Lambda_{p}}-1 & \text{if $i=p$}, \\[1.5mm]
M_{\Lambda_{i}} & \text{otherwise}.
\end{cases}
\end{equation}
Fix $i \in \BZ$, and take 
an interval $I$ in $\BZ$ such that 
\begin{equation*}
I \in 
\Int(\bM;e,p) \cap \Int(\bM;s_{p},p) \cap 
\Int(\bM;e,i) \cap \Int(\bM;s_{p},i).
\end{equation*}
We see from Remark~\ref{rem:fj} that 
$I \in \Int(f_{p}\bM;e,i)$, and hence 
that $M_{\Lambda_{i}}'=M_{\vpi_{i}^{I}}'$ by the definition.
Assume now that $i \ne p$. 
Since $\pair{h_{p}}{\vpi_{i}^{I}} \le 0$ by \eqref{eq:pair-vpi}, 
it follows from \eqref{eq:AM2} that 
$M_{\vpi_{i}^{I}}'=
 (f_{p}\bM)_{\vpi_{i}^{I}}=M_{\vpi_{i}^{I}}$.
Also, since $I \in \Int(\bM;e,i)$, we have 
$M_{\vpi_{i}^{I}}=M_{\Lambda_{i}}$ by the definition. 
Therefore, we obtain 
\begin{equation*}
M_{\Lambda_{i}}'=M_{\vpi_{i}^{I}}'=M_{\vpi_{i}^{I}}=M_{\Lambda_{i}}
\quad \text{if $i \ne p$.}
\end{equation*}
Assume then that $i = p$. 
Since $\pair{h_{p}}{\vpi_{p}^{I}}=1$, 
it follows from \eqref{eq:AM2} that 
%
%%%%%%%%%%%%%%
%%% eq:AM3 %%%
%%%%%%%%%%%%%%
%
\begin{equation} \label{eq:AM3}
M_{\vpi_{p}^{I}}' = 
 (f_{p}\bM)_{\vpi_{p}^{I}} =
 \min \bigl(
   M_{\vpi_{p}^{I}},\ 
   M_{s_{p}\vpi_{p}^{I}}+c_{p}(\bM)
   \bigr),
\end{equation}
where $c_{p}(\bM)=
M_{\Lambda_{p}}-M_{s_{p}\Lambda_{p}}-1$. 
Note that $M_{\vpi_{p}^{I}}=M_{\Lambda_{p}}$ 
(resp., $M_{s_{p}\vpi_{p}^{I}}=M_{s_{p}\Lambda_{p}}$) 
since $I \in \Int(\bM;e,p)$ (resp., $I \in \Int(\bM;s_{p},p)$).
Substituting these into \eqref{eq:AM3}, 
we conclude that $M_{\Lambda_{p}}'=M_{\vpi_{p}^{I}}' = 
M_{\Lambda_{p}}-1$, as desired. 

Next we show that 
%
%%%%%%%%%%%%%%%%%
%%% eq:silami %%%
%%%%%%%%%%%%%%%%%
%
\begin{equation} \label{eq:silami}
M_{s_{i}\Lambda_{i}}'=M_{s_{i}\Lambda_{i}}
\quad \text{for $i \in \BZ$ with $i \ne p-1,\,p+1$}.
\end{equation}
Take an interval $I$ in $\BZ$ such that 
\begin{equation*}
I \in \Int(\bM;e,p) \cap 
\Int(\bM;s_{p},p) \cap 
\Int(\bM;s_{i},i) \cap 
\Int(\bM;s_{p}s_{i},i).
\end{equation*}
We see from Remark~\ref{rem:fj} that 
$I \in \Int(f_{p}\bM;s_{i},i)$, 
and hence that $M_{s_{i}\Lambda_{i}}'=
M_{s_{i}\vpi_{i}^{I}}'$ by the definition. 
Since $i \ne p-1,\,p+1$, 
we deduce from \eqref{eq:pair-vpi} that 
$\pair{h_{p}}{s_{i}\vpi_{i}^{I}} \le 0$. 
Hence it follows from \eqref{eq:AM2} that 
$M_{s_{i}\vpi_{i}^{I}}'= 
 (f_{p}\bM)_{s_{i}\vpi_{i}^{I}}=
 M_{s_{i}\vpi_{i}^{I}}$. 
Also, since $I \in \Int(\bM;s_{i},i)$, we have 
$M_{s_{i}\vpi_{i}^{I}}=M_{s_{i}\Lambda_{i}}$. 
Thus we obtain $M_{s_{i}\Lambda_{i}}'=
M_{s_{i}\vpi_{i}^{I}}'=M_{s_{i}\vpi_{i}^{I}}=
M_{s_{i}\Lambda_{i}}$, as desired. 

Now, recall from \eqref{eq:dfn-vep2} that 
\begin{equation*}
\ve_{p}(f_{p}\bM) = 
 - \left(
     M_{\Lambda_{p}}'+M_{s_{p}\Lambda_{p}}'
     +\sum_{r \in \BZ \setminus \{p\}} 
      a_{rp} M_{\Lambda_{r}}'
    \right). 
\end{equation*}
Here, by \eqref{eq:lam-i} and \eqref{eq:silami}, 
we have $M_{\Lambda_{p}}'=M_{\Lambda_{p}}-1$, 
$M_{s_{p}\Lambda_{p}}'=M_{s_{p}\Lambda_{p}}$, and 
\begin{equation*}
\sum_{r \in \BZ \setminus \{p\}} a_{rp} M_{\Lambda_{r}}' = 
\sum_{r \in \BZ \setminus \{p\}} a_{rp} M_{\Lambda_{r}}.
\end{equation*}
Therefore, by \eqref{eq:dfn-vep2}, we conclude that
\begin{equation*}
\ve_{p}(f_{p}\bM) = 
 - \left(
     (M_{\Lambda_{p}}-1)+M_{s_{p}\Lambda_{p}}
     +\sum_{r \in \BZ \setminus \{p\}} 
      a_{rp} M_{\Lambda_{r}}
    \right)=\ve_{p}(\bM)+1.
\end{equation*}
Arguing in the same manner, we can prove 
that $\ve_{q}(f_{p}\bM) = \ve_{q}(\bM)$. 
This proves part (2). 

(3) We prove that $e_{p}f_{q}=f_{q}e_{p}$; 
the proofs of the other equalities are similar. 
Let $\bM \in \bz_{\BZ}$. Assume first that 
$e_{p}\bM = \bzero$, or equivalently, 
$\ve_{p}(\bM)=0$. Then we have 
$f_{q}e_{p}\bM=\bzero$. Also, 
it follows from part (2) that 
$\ve_{p}(f_{q}\bM)=\ve_{p}(\bM)=0$, 
which implies that $e_{p}(f_{q}\bM)=\bzero$. 
Thus we get $e_{p}f_{q}\bM=f_{q}e_{p}\bM=\bzero$. 

Assume next that $e_{p}\bM \ne \bzero$, 
or equivalently, $\ve_{p}(\bM) > 0$. 
Then we have $f_{q}e_{p}\bM \ne \bzero$. 
Also, it follows from part (2) that 
$\ve_{p}(f_{q}\bM)=\ve_{p}(\bM) > 0$, 
which implies that $e_{p}(f_{q}\bM) \ne \bzero$. 
We need to show that 
$(e_{p}f_{q}\bM)_{\gamma}=(f_{q}e_{p}\bM)_{\gamma}$ 
for all $\gamma \in \Gamma_{\BZ}$. 
Fix $\gamma \in \Gamma_{\BZ}$, and 
take an interval $I$ in $\BZ$ satisfying 
the following conditions: 

(i) $\gamma \in \Gamma_{I}$; 

(ii) $I \in \Int(f_{q}\bM;e,p) \cap \Int(f_{q}\bM;s_{p},p) \cap 
\Int(f_{q}\bM;e,p-1) \cap \Int(f_{q}\bM;e,p+1)$; 

(iii) $I \in \Int(\bM;e,q) \cap \Int(\bM;s_{q},q)$; 

(iv) $I \in \Int(e_{p}\bM;e,q) \cap \Int(e_{p}\bM;s_{q},q)$; 

(v) $I \in \Int(\bM;e,p) \cap \Int(\bM;s_{p},p) \cap 
\Int(\bM;e,p-1) \cap \Int(\bM;e,p+1)$. 

\noindent
Then, we have 
\begin{align*}
(e_{p}f_{q}\bM)_{I}
 & =e_{p}(f_{q}\bM)_{I}
   \quad \text{by \eqref{eq:res-ej} and condition (ii)} \\
 & = e_{p}(f_{q}\bM_{I})
   \quad \text{by \eqref{eq:res-fj} and condition (iii)} \\
 & =e_{p}f_{q}\bM_{I},
\end{align*}
and 
\begin{align*}
(f_{q}e_{p}\bM)_{I}
 & =f_{q}(e_{p}\bM)_{I}
   \quad \text{by \eqref{eq:res-fj} and condition (iv)} \\
 & = f_{q}(e_{p}\bM_{I})
   \quad \text{by \eqref{eq:res-ej} and condition (v)} \\
 & = f_{q}e_{p}\bM_{I}.
\end{align*}
Hence we see from condition~(a) of Definition~\ref{dfn:BZdatum2} and 
Theorem~\ref{thm:binf} that 
$e_{p}f_{q}\bM_{I}=f_{q}e_{p}\bM_{I}$, and hence 
$(e_{p}f_{q}\bM)_{I}=(f_{q}e_{p}\bM)_{I}$. 
Therefore, we obtain 
$(e_{p}f_{q}\bM)_{\gamma}=(f_{q}e_{p}\bM)_{\gamma}$ 
since $\gamma \in \Gamma_{I}$ by condition~(i). 
This proves part (3), 
thereby completing the proof of the lemma.
\end{proof}

\begin{rem}
Let $\bM \in \bz_{\BZ}$, and $p \in I$. 
From the definition, it follows that 
$\ve_{p}(\bM)=0$ if and only if 
$e_{p}\bM=\bzero$, and that 
$\ve_{p}(\bM) \in \BZ_{\ge 0}$. 
In addition, $\ve_{p}(e_{p}\bM)=\ve_{p}(\bM)-1$ by 
Lemma~\ref{lem:ef}\,(2). 
Consequently, we deduce that 
$\ve_{p}(\bM)=
 \max \bigl\{N \ge 0 \mid 
 e_{p}^{N}\bM \ne \bzero\bigr\}$. 
\end{rem}

%=========================%
%     START SECTION 04    %
%=========================%
%
\section{Berenstein-Zelevinsky data of type $A_{\ell}^{(1)}$.}
\label{sec:BZdatum-aff}

Throughout this section, 
we take and fix $\ell \in \BZ_{\ge 2}$ arbitrarily. 

%==============================%
%     START SUBSECTION 0401    %
%==============================%
%
\subsection{Basic notation in type $A_{\ell}^{(1)}$.}
\label{subsec:notation-aff}
Let $\ha{\Fg}$ be the affine Lie algebra of 
type $A_{\ell}^{(1)}$ over $\BC$. Let 
$\ha{A}=(\ha{a}_{ij})_{i,j \in \ha{I}}$ 
denote the Cartan matrix of $\ha{\Fg}$ 
with index set $\ha{I}:=\bigl\{0,\,1,\,\dots,\,\ell\bigr\}$; 
the entries $\ha{a}_{ij}$ are given by: 
%
%%%%%%%%%%%%%%%%%
%%% eq:affine %%%
%%%%%%%%%%%%%%%%%
%
\begin{equation} \label{eq:affine}
\ha{a}_{ij}=
 \begin{cases}
 2 & \text{if $i=j$}, \\
 -1 & \text{if $|i-j|=1$ or $\ell$}, \\
 0 & \text{otherwise},
 \end{cases}
\end{equation}
for $i,\,j \in \ha{I}$. 
Denote by $\ha{\Fh}$ the Cartan subalgebra of $\ha{\Fg}$, 
by $\ha{h}_{i} \in \ha{\Fh}$, $i \in \ha{I}$, 
the simple coroots of $\ha{\Fg}$, and 
by $\ha{\alpha}_{i} \in \ha{\Fh}^{\ast}:=
\Hom_{\BC}(\ha{\Fh},\,\BC)$, $i \in \ha{I}$, 
the simple roots of $\ha{\Fg}$; note that 
$\pair{\ha{h}_{i}}{\ha{\alpha}_{j}}=\ha{a}_{ij}$ 
for $i,\,j \in \ha{I}$, where $\pair{\cdot\,}{\cdot}$ 
is the canonical pairing between $\ha{\Fh}$ and $\ha{\Fh}^{\ast}$. 

Also, let $\ha{\Fg}^{\vee}$ denote 
the (Langlands) dual Lie algebra of $\ha{\Fg}$; 
that is, $\ha{\Fg}^{\vee}$ is the affine Lie algebra of 
type $A_{\ell}^{(1)}$ over $\BC$ associated to 
the transpose ${}^{t}\!\ha{A} \, (=\ha{A})$ of $\ha{A}$, 
with Cartan subalgebra $\ha{\Fh}^{\ast}$, 
simple coroots $\ha{\alpha}_{i} \in \ha{\Fh}^{\ast}$, $i \in \ha{I}$, 
and simple roots $\ha{h}_{i} \in \ha{\Fh}$, $i \in \ha{I}$.
Let $U_{q}(\ha{\Fg}^{\vee})$ be 
the quantized universal enveloping algebra over $\BC(q)$
associated to the Lie algebra $\ha{\Fg}^{\vee}$, 
$U_{q}^{-}(\ha{\Fg}^{\vee})$ 
the negative part of $U_{q}(\ha{\Fg}^{\vee})$, 
and $\ha{\CB}(\infty)$ the crystal basis of 
$U_{q}^{-}(\ha{\Fg}^{\vee})$. 
For a dominant integral weight $\ha{\lambda} \in \ha{\Fh}$ 
for $\ha{\Fg}^{\vee}$, $\ha{\CB}(\ha{\lambda})$ denotes 
the crystal basis of the irreducible highest weight 
$U_{q}(\ha{\Fg}^{\vee})$-module of highest weight $\ha{\lambda}$. 

%==============================%
%     START SUBSECTION 0402    %
%==============================%
%
\subsection{Dynkin diagram automorphism in type $A_{\infty}$ and 
its action on $\bz_{\BZ}$.}
\label{subsec:da}

For the fixed $\ell \in \BZ_{\ge 2}$, 
the (Dynkin) diagram automorphism in type $A_{\infty}$ 
is a bijection $\sigma:\BZ \rightarrow \BZ$ given by: 
$\sigma(i)=i+\ell+1$ for $i \in \BZ$. 
This induces a $\BC$-linear automorphism 
$\sigma:\Fh \stackrel{\sim}{\rightarrow} \Fh$ of 
$\Fh=\bigoplus_{i \in \BZ} \BC h_{i}$ by: 
$\sigma(h_{i})=h_{\sigma(i)}$ for $i \in \BZ$, 
and also a $\BC$-linear automorphism 
$\sigma:\Fh^{\ast}_{\res} \stackrel{\sim}{\rightarrow} \Fh^{\ast}_{\res}$ 
of the restricted dual space 
$\Fh^{\ast}_{\res}:=\bigoplus_{i \in \BZ} \BC \Lambda_{i}$ of 
$\Fh=\bigoplus_{i \in \BZ} \BC h_{i}$ by: 
$\sigma(\Lambda_{i})=\Lambda_{\sigma(i)}$ for $i \in \BZ$. 
Observe that $\pair{\sigma(h)}{\sigma(\Lambda)}=
\pair{h}{\Lambda}$ for all $h \in \Fh$ 
and $\Lambda \in \Fh^{\ast}_{\res}$, and 
$\sigma(\alpha_{i})=\alpha_{\sigma(i)}$ 
for $i \in \BZ$; note also that 
$\alpha_{i} \in \Fh^{\ast}_{\res}$ for all $i \in \BZ$,
since $\alpha_{i}=2\Lambda_{i}-\Lambda_{i-1}-\Lambda_{i+1}$. 
Moreover, this $\sigma:\BZ \rightarrow \BZ$ 
naturally induces a group automorphism 
$\sigma:W_{\BZ} \stackrel{\sim}{\rightarrow} W_{\BZ}$ of 
the Weyl group $W_{\BZ}$ by: 
$\sigma(s_{i})=s_{\sigma(i)}$ for $i \in \BZ$. 

It is easily seen that 
$-w\Lambda_{i} \in \Fh^{\ast}_{\res}$ 
for all $w \in W_{\BZ}$ and $i \in \BZ$, and hence 
the set $\Gamma_{\BZ}$ (of chamber weights) is a subset of 
$\Fh^{\ast}_{\res}$. In addition,
%
%%%%%%%%%%%%%%%%%%
%%% eq:sig-gam %%%
%%%%%%%%%%%%%%%%%%
%
\begin{equation} \label{eq:sig-gam}
\sigma(-w\Lambda_{i})=-\sigma(w)\Lambda_{\sigma(i)}
\quad \text{for $w \in W_{\BZ}$ and $i \in \BZ$}.
\end{equation}
Therefore, the restriction of 
$\sigma:\Fh^{\ast}_{\res} \stackrel{\sim}{\rightarrow} 
\Fh^{\ast}_{\res}$ to the subset $\Gamma_{\BZ}$ 
gives rise to a bijection $\sigma:\Gamma_{\BZ} 
\stackrel{\sim}{\rightarrow} \Gamma_{\BZ}$. 
%
%%%%%%%%%%%%%%%%%%%
%%% rem:sig-vpi %%%
%%%%%%%%%%%%%%%%%%%
%
\begin{rem} \label{rem:sig-vpi}
Let $I$ be an interval in $\BZ$, and $i \in I$; 
note that $\sigma(i)$ is contained in $\sigma(I)$. 
Because $\vpi_{i}^{I} \in \Gamma_{\BZ}$ 
can be written as: $\vpi_{i}^{I}=
\Lambda_{i}-\Lambda_{(\min I)-1}-\Lambda_{(\max I)+1}$ 
(see \eqref{eq:pair-vpi}), we deduce that 
$\sigma(\vpi_{i}^{I})=\vpi_{\sigma(i)}^{\sigma(I)}$.
\end{rem}

Let $\bM=(M_{\gamma})_{\gamma \in \Gamma_{\BZ}}$ be 
a collection of integers indexed by $\Gamma_{\BZ}$. 
We define collections $\sigma(\bM)$ and $\sigma^{-1}(\bM)$ of 
integers indexed by $\Gamma_{\BZ}$ by: 
$\sigma(\bM)_{\gamma}=M_{\sigma^{-1}(\gamma)}$ and 
$\sigma^{-1}(\bM)_{\gamma}=M_{\sigma(\gamma)}$ 
for each $\gamma \in \Gamma_{\BZ}$, respectively. 
%
%%%%%%%%%%%%%%%%%%%%%%%%%
%%% lem:sigma-stable %%%%
%%%%%%%%%%%%%%%%%%%%%%%%%
%
\begin{lem} \label{lem:sigma-stable}
If $\bM \in \bz_{\BZ}$, 
then $\sigma(\bM) \in \bz_{\BZ}$ and 
$\sigma^{-1}(\bM) \in \bz_{\BZ}$. 
\end{lem}

\begin{proof}
We prove that $\sigma(\bM) \in \bz_{\BZ}$; 
we can prove that $\sigma^{-1}(\bM) \in \bz_{\BZ}$ similarly. 
Write $\bM \in \bz_{\BZ}$ and $\sigma(\bM)$ as: 
$\bM=(M_{\gamma})_{\gamma \in \Gamma_{\BZ}}$ and 
$\sigma(\bM)=(M_{\gamma}')_{\gamma \in \Gamma_{\BZ}}$, 
respectively.
First we prove that $\sigma(\bM)=
(M_{\gamma}')_{\gamma \in \Gamma_{\BZ}}$ 
satisfies condition~(a) of Definition~\ref{dfn:BZdatum2}. 
Let $K$ be an interval in $\BZ$.
We need to show that $\sigma(\bM)_{K}=
(M_{\gamma}')_{\gamma \in \Gamma_{K}}$ satisfies 
condition~(1) of Definition~\ref{dfn:BZdatum} 
(with $I$ replaced by $K$). 
Fix $w \in W_{K}$, and $i \in K$. 
For simplicity of notation, 
we set $w_{1}:=\sigma^{-1}(w)$, 
$i_{1}:=\sigma^{-1}(i)$, and $K_{1}:=\sigma^{-1}(K)$; 
note that $w_{1} \in W_{K_{1}}$, and $i_{1} \in K_{1}$. 
Since $\bM=(M_{\gamma})_{\gamma \in \Gamma_{\BZ}} \in \bz_{\BZ}$, 
it follows from condition~(a) of Definition~\ref{dfn:BZdatum2} 
that $\bM_{K_{1}}=(M_{\gamma})_{\gamma \in \Gamma_{K_{1}}} 
\in \bz_{K_{1}}$. Hence we see from 
condition~(1) of Definition~\ref{dfn:BZdatum} that 
\begin{equation*}
M_{w_{1}\vpi_{i_{1}}^{K_{1}}}+
M_{w_{1}s_{i_{1}}\vpi_{i_{1}}^{K_{1}}}+
\sum_{j \in K_{1} \setminus \{i_{1}\}} 
a_{j,i_{1}} M_{w_{1}\vpi_{j}^{K_{1}}} \le 0.
\end{equation*}
Here, by the equality 
$a_{\sigma^{-1}(j),i_{1}}=a_{j,\sigma(i_{1})}$, 
\begin{equation*}
\sum_{j \in K_{1} \setminus \{i_{1}\}} 
a_{j, i_{1}} M_{w_{1}\vpi_{j}^{K_{1}}}=
\sum_{j \in K \setminus \{i\}} 
a_{\sigma^{-1}(j),i_{1}} M_{w_{1}\vpi_{\sigma^{-1}(j)}^{K_{1}}}=
\sum_{j \in K \setminus \{i\}} 
a_{ji} M_{w_{1}\vpi_{\sigma^{-1}(j)}^{K_{1}}}.
\end{equation*}
Also, we see from \eqref{eq:sig-gam} and 
Remark~\ref{rem:sig-vpi} that 
\begin{align*}
& M_{w\vpi_{i}^{K}}'=M_{\sigma^{-1}(w\vpi_{i}^{K})}=
  M_{w_{1}\vpi_{i_{1}}^{K_{1}}}, \\[1.5mm]
& M_{ws_{i}\vpi_{i}^{K}}'=
  M_{\sigma^{-1}(ws_{i}\vpi_{i}^{K})}=  
  M_{w_{1}s_{i_{1}}\vpi_{i_{1}}^{K_{1}}}, \\[1.5mm]
& M_{w\vpi_{j}^{K}}'=M_{\sigma^{-1}(w\vpi_{j}^{K})}=
  M_{w_{1}\vpi_{\sigma^{-1}(j)}^{K_{1}}} \quad 
\text{for $j \in K \setminus \{i\}$}.
\end{align*}
Combining these, we obtain 
\begin{equation*}
M_{w\vpi_{i}^{K}}'+
M_{ws_{i}\vpi_{i}^{K}}'+
\sum_{j \in K \setminus \{i\}} 
a_{ji} M_{w\vpi_{j}^{K}}' \le 0, 
\end{equation*}
as desired. Similarly, we can show that 
$\sigma(\bM)_{K}=
(M_{\gamma}')_{\gamma \in \Gamma_{K}}$ satisfies 
condition (2) of Definition~\ref{dfn:BZdatum} 
(with $I$ replaced by $K$); use the fact that 
if $i,\,j \in K$ and $w \in W_{K}$ are such that 
$a_{ij}=a_{ji}=-1$, and 
$ws_{i} > w$, $ws_{j} > w$, then 
$a_{i_{1},j_{1}}=a_{j_{1},i_{1}}=-1$, and 
$w_{1}s_{i_{1}} > w_{1}$, $w_{1}s_{j_{1}} > w_{1}$, 
where 
$i_{1}:=\sigma^{-1}(i),\,j_{1}:=\sigma^{-1}(j) \in 
K_{1}=\sigma^{-1}(K)$, and 
$w_{1}:=\sigma^{-1}(w) \in W_{K_{1}}$. 
It remains to show that 
$M_{w_{0}^{K}\vpi_{i}^{K}}'=0$ for all $i \in K$. 
Let $i \in K$, and set $i_{1}:=
\sigma^{-1}(i) \in K_{1}=\sigma^{-1}(K)$. 
Then, by \eqref{eq:sig-gam} and Remark~\ref{rem:sig-vpi}, 
we have
\begin{equation*}
M_{w_{0}^{K}\vpi_{i}^{K}}'=
M_{\sigma^{-1}(w_{0}^{K}\vpi_{i}^{K})}=
M_{w_{0}^{K_{1}}\vpi_{i_{1}}^{K_{1}}}, 
\end{equation*}
which is equal to zero 
since $\bM_{K_{1}} \in \bz_{K_{1}}$. 
This proves that $\sigma(\bM)_{K} \in \bz_{K}$, as desired.

Next we prove that $\sigma(\bM)=
(M_{\gamma}')_{\gamma \in \Gamma_{\BZ}}$ 
satisfies condition (b) of Definition~\ref{dfn:BZdatum2}. 
Fix $w \in W_{\BZ}$, and $i \in \BZ$. 
Take an interval $I$ in $\BZ$ such that 
$I_{1}:=\sigma^{-1}(I)$ is an element of 
$\Int(\bM; w_{1},i_{1})$, where 
$w_{1}:=\sigma^{-1}(w)$ and 
$i_{1}:=\sigma^{-1}(i)$. 
Let $J$ be an arbitrary interval in $\BZ$ 
containing $I$, and set $J_{1}:=\sigma^{-1}(J)$; 
note that $J_{1} \supset I_{1}$.
Then, we have 
\begin{align*}
M_{w\vpi_{i}^{J}}' & =
M_{\sigma^{-1}(w\vpi_{i}^{J})}=
M_{w_{1}\vpi_{i_{1}}^{J_{1}}} 
 \quad \text{by \eqref{eq:sig-gam} and Remark~\ref{rem:sig-vpi}} \\
& = M_{w_{1}\vpi_{i_{1}}^{I_{1}}} 
 \quad \text{since $I_{1} \in \Int(\bM; w_{1},i_{1})$ 
 and $J_{1} \supset I_{1}$} \\
& = M_{\sigma^{-1}(w\vpi_{i}^{I})}
 \quad \text{by \eqref{eq:sig-gam} and Remark~\ref{rem:sig-vpi}} \\
& = M_{w\vpi_{i}^{I}}'.
\end{align*}
This proves that $\sigma(\bM)=
(M_{\gamma}')_{\gamma \in \Gamma_{\BZ}}$ 
satisfies condition (b) of Definition~\ref{dfn:BZdatum2}, 
thereby completing the proof of the lemma.
\end{proof}
%
%%%%%%%%%%%%%%%%%%%%
%%% rem:sig-wlam %%%
%%%%%%%%%%%%%%%%%%%%
%
\begin{rem} \label{rem:sig-wlam}
Let $\bM=
(M_{\gamma})_{\gamma \in \Gamma_{\BZ}} \in \bz_{\BZ}$, and 
write $\sigma(\bM) \in \bz_{\BZ}$ as: $\sigma(\bM)=
(M_{\gamma}')_{\gamma \in \Gamma_{\BZ}}$. 
Fix $w \in W_{\BZ}$, and $i \in \BZ$. 
Set $w_{1}:=\sigma^{-1}(w)$, and $i_{1}:=\sigma^{-1}(i)$. 
We see from the proof of 
Lemma~\ref{lem:sigma-stable} that 
if we take an interval $I$ in $\BZ$ such that 
$I_{1}:=\sigma^{-1}(I)$ is an element of 
$\Int(\bM; w_{1},i_{1})$, then 
the interval $I$ is an element of 
$\Int(\sigma(\bM);w,i)$. 
Moreover, since 
$M_{w\vpi_{i}^{I}}'=
 M_{w_{1}\vpi_{i_{1}}^{I_{1}}}$,  
we have 
\begin{equation*}
M_{w\Lambda_{i}}'=
M_{w\vpi_{i}^{I}}'=
M_{w_{1}\vpi_{i_{1}}^{I_{1}}}=
M_{w_{1}\Lambda_{i_{1}}}=
M_{\sigma^{-1}(w\Lambda_{i})},
\end{equation*}
where $M_{w\Lambda_{i}}':=\Theta(\sigma(\bM))_{w\Lambda_{i}}$, and 
$M_{w_{1}\Lambda_{i_{1}}}:=\Theta(\bM)_{w_{1}\Lambda_{i_{1}}}$. 
\end{rem}

By Lemma~\ref{lem:sigma-stable}, 
we obtain maps $\sigma:\bz_{\BZ} \rightarrow \bz_{\BZ}$, 
$\bM \mapsto \sigma(\bM)$, and 
$\sigma^{-1}:\bz_{\BZ} \rightarrow \bz_{\BZ}$, 
$\bM \mapsto \sigma^{-1}(\bM)$; 
since both of the composite maps 
$\sigma \sigma^{-1}$ and 
$\sigma^{-1} \sigma$ 
are the identity map on $\bz_{\BZ}$, 
it follows that 
$\sigma:\bz_{\BZ} \rightarrow \bz_{\BZ}$ and 
$\sigma^{-1}:\bz_{\BZ} \rightarrow \bz_{\BZ}$ 
are bijective. 
%
%%%%%%%%%%%%%%%%%%
%%% lem:sig-ef %%%
%%%%%%%%%%%%%%%%%%
%
\begin{lem} \label{lem:sig-ef}
{\rm (1)} 
Let $\bM \in \bz_{\BZ}$, and $p \in \BZ$. 
Then, $\ve_{p}(\sigma(\bM))=\ve_{\sigma^{-1}(p)}(\bM)$. 

{\rm (2)}
There hold $\sigma \circ e_{p} = e_{\sigma(p)} \circ \sigma$ and 
$\sigma \circ f_{p} = f_{\sigma(p)} \circ \sigma$ on 
$\bz_{\BZ} \cup \{\bzero\}$ for all $p \in \BZ$. 
Here it is understood that $\sigma(\bzero):=\bzero$.
\end{lem}

\begin{proof}
Part (1) follows immediately from \eqref{eq:dfn-vep2} 
by using Remark~\ref{rem:sig-wlam}. We will prove part (2). 
Let $\bM \in \bz_{\BZ}$, and $p \in \BZ$. 
First we show that 
$\sigma (f_{p}\bM) = f_{\sigma(p)}(\sigma(\bM))$, i.e., 
$\bigl(\sigma (f_{p}\bM)\bigr)_{\gamma} = 
 \bigl(f_{\sigma(p)}(\sigma(\bM))\bigr)_{\gamma}$
for all $\gamma \in \Gamma_{\BZ}$. 
We write $\bM$ and $\sigma(\bM)$ as: 
$\bM=(M_{\gamma})_{\gamma \in \Gamma_{\BZ}}$ and 
$\sigma(\bM)=(M_{\gamma}')_{\gamma \in \Gamma_{\BZ}}$, 
respectively. 
It follows from \eqref{eq:AM2} that 
%
%%%%%%%%%%%%%%%%%%
%%% eq:sig-ef1 %%%
%%%%%%%%%%%%%%%%%%
%
\begin{align}
& \bigl(\sigma (f_{p}\bM)\bigr)_{\gamma} =
  (f_{p}\bM)_{\sigma^{-1}(\gamma)} \nonumber \\[3mm]
& \qquad =
\begin{cases}
 \min \bigl(
   M_{\sigma^{-1}(\gamma)},\ 
   M_{s_{p}\sigma^{-1}(\gamma)}+c_{p}(\bM)
   \bigr)
   & \text{if $\pair{h_{p}}{\sigma^{-1}(\gamma)} > 0$}, \\[1.5mm]
 M_{\sigma^{-1}(\gamma)} & \text{otherwise},
\end{cases} \label{eq:sig-ef1}
\end{align}
where $c_{p}(\bM)=
 M_{\Lambda_{p}}-M_{s_{p}\Lambda_{p}}-1$ 
with $M_{\Lambda_{p}}:=\Theta(\bM)_{\Lambda_{p}}$ and 
$M_{s_{p}\Lambda_{p}}:=\Theta(\bM)_{s_{p}\Lambda_{p}}$.
Also, it follows from \eqref{eq:AM2} that
%
%%%%%%%%%%%%%%%%%
%%% eq:sig-f1 %%%
%%%%%%%%%%%%%%%%%
%
\begin{equation} \label{eq:sig-f1}
\bigl(f_{\sigma(p)}(\sigma(\bM))\bigr)_{\gamma} = 
\begin{cases}
 \min \bigl(
   M_{\gamma}',\ 
   M_{s_{\sigma(p)}\gamma}'+c_{\sigma(p)}(\sigma(\bM))
   \bigr)
   & \text{if $\pair{h_{\sigma(p)}}{\gamma} > 0$}, \\[1.5mm]
 M_{\gamma}' & \text{otherwise},
\end{cases}
\end{equation}
where $c_{\sigma(p)}(\sigma(\bM))=
 M_{\Lambda_{\sigma(p)}}'-M_{s_{\sigma(p)}\Lambda_{\sigma(p)}}'-1$ 
with $M_{\Lambda_{\sigma(p)}}':=
\Theta(\sigma(\bM))_{\Lambda_{\sigma(p)}}$ and 
$M_{s_{\sigma(p)}\Lambda_{\sigma(p)}}':=
\Theta(\sigma(\bM))_{s_{\sigma(p)}\Lambda_{\sigma(p)}}$.
Here we see from Remark~\ref{rem:sig-wlam} that 
\begin{equation*}
M_{\Lambda_{\sigma(p)}}'=
M_{\sigma^{-1}(\Lambda_{\sigma(p)})}=
M_{\Lambda_{p}}
\quad \text{and} \quad
M_{s_{\sigma(p)}\Lambda_{\sigma(p)}}'=
M_{\sigma^{-1}(s_{\sigma(p)}\Lambda_{\sigma(p)})}=
M_{s_{p}\Lambda_{p}},
\end{equation*}
and hence that $c_{\sigma(p)}(\sigma(\bM))=c_{p}(\bM)$. 
In addition, 
\begin{equation*}
M_{\gamma}'=M_{\sigma^{-1}(\gamma)}
\quad \text{and} \quad 
M_{s_{\sigma(p)}\gamma}'=
 M_{\sigma^{-1}(s_{\sigma(p)}\gamma)}=
 M_{s_{p}\sigma^{-1}(\gamma)}
\end{equation*}
by the definitions. Observe that 
$\pair{h_{\sigma(p)}}{\gamma}=
\pair{\sigma(h_{p})}{\gamma}=
\pair{h_{p}}{\sigma^{-1}(\gamma)}$, 
and hence that 
$\pair{h_{\sigma(p)}}{\gamma} > 0$ 
if and only if 
$\pair{h_{p}}{\sigma^{-1}(\gamma)} > 0$.
Substituting these into \eqref{eq:sig-f1}, 
we obtain 
\begin{align*}
\bigl(f_{\sigma(p)}(\sigma(\bM))\bigr)_{\gamma} 
& = 
\begin{cases}
 \min \bigl(
   M_{\sigma^{-1}(\gamma)},\ 
   M_{s_{p}\sigma^{-1}(\gamma)}+c_{p}(\bM)
   \bigr)
   & \text{if $\pair{h_{p}}{\sigma^{-1}(\gamma)} > 0$}, \\[1.5mm]
 M_{\sigma^{-1}(\gamma)} & \text{otherwise},
\end{cases} \\[3mm]
& = \bigl(\sigma (f_{p}\bM)\bigr)_{\gamma},
\end{align*}
as desired. 

Next we show that 
$\sigma (e_{p}\bM) = e_{\sigma(p)}(\sigma(\bM))$. 
If $e_{p}\bM=\bzero$, or equivalently, 
$\ve_{p}(\bM)=0$, then it follows from part (1) that 
$\ve_{\sigma(p)}(\sigma(\bM))=\ve_{p}(\bM)=0$, 
and hence $e_{\sigma(p)}(\sigma(\bM))=\bzero$, 
which implies that $\sigma (e_{p}\bM) = 
e_{\sigma(p)}(\sigma(\bM))=\bzero$. 
Assume, therefore, that $e_{p}\bM \ne \bzero$, or equivalently, 
$\ve_{p}(\bM) > 0$. Then, it follows from part (1) that 
$\ve_{\sigma(p)}(\sigma(\bM))=\ve_{p}(\bM) > 0$, 
and hence $e_{\sigma(p)}(\sigma(\bM)) \ne \bzero$. 
Consequently, we see from Lemma~\ref{lem:ef}\,(1) that 
$f_{\sigma(p)}e_{\sigma(p)}(\sigma(\bM))=\sigma(\bM)$. 
Also, 
\begin{align*}
f_{\sigma(p)}(\sigma (e_{p}\bM))
 & = \sigma(f_{p}e_{p}\bM) \quad 
   \text{since $f_{\sigma(p)} \circ \sigma= \sigma \circ f_{p}$} \\
 & =\sigma(\bM) \quad \text{by Lemma~\ref{lem:ef}\,(1)}.
\end{align*}
Thus, we have 
$f_{\sigma(p)}e_{\sigma(p)}(\sigma(\bM))=\sigma(\bM)=
 f_{\sigma(p)}(\sigma (e_{p}\bM))$. 
Applying $e_{\sigma(p)}$ to both sides of this equation, 
we obtain $e_{\sigma(p)}(\sigma(\bM))=\sigma (e_{p}\bM)$ 
by Lemma~\ref{lem:ef}\,(1), 
as desired. This completes the proof of the lemma. 
\end{proof}

%==============================%
%     START SUBSECTION 0403    %
%==============================%
%
\subsection{
 BZ data of type $A_{\ell}^{(1)}$ 
 and a crystal structure on them.}
\label{subsec:cry-aff}
%
%%%%%%%%%%%%%%%%
%%% dfn:sigM %%%
%%%%%%%%%%%%%%%%
%
\begin{dfn} \label{dfn:sigM}
A BZ datum of type $A_{\ell}^{(1)}$ is 
a BZ datum $\bM=(M_{\gamma})_{\gamma \in \Gamma_{\BZ}} 
\in \bz_{\BZ}$ of type $A_{\infty}$ such that 
$\sigma(\bM)=\bM$, or equivalently, 
$M_{\sigma^{-1}(\gamma)}=M_{\gamma}$ 
for all $\gamma \in \Gamma_{\BZ}$. 
\end{dfn}
%
%%%%%%%%%%%%%%%%%%%%%
%%% rem:sig-wlam2 %%%
%%%%%%%%%%%%%%%%%%%%%
%
\begin{rem} \label{rem:sig-wlam2}
Keep the notation of Remark~\ref{rem:sig-wlam}. 
In addition, we assume that $\sigma(\bM)=\bM$. 
Because $I \in \Int (\sigma(\bM);w,i)=\Int (\bM;w,i)$ and 
$M_{w\vpi_{i}^{I}}'=M_{w\vpi_{i}^{I}}$ 
by the assumption that $\sigma(\bM)=\bM$, 
it follows that $M_{w\Lambda_{i}}'=
M_{w\vpi_{i}^{I}}'=M_{w\vpi_{i}^{I}}=M_{w\Lambda_{i}}$. 
Since $M_{w\Lambda_{i}}'=M_{\sigma^{-1}(w\Lambda_{i})}$ 
as shown in Remark~\ref{rem:sig-wlam}, 
we obtain $M_{\sigma^{-1}(w\Lambda_{i})}=M_{w\Lambda_{i}}$.
\end{rem}

Denote by $\bz_{\BZ}^{\sigma}$ 
the set of all BZ data of type $A_{\ell}^{(1)}$; that is,
\begin{equation}
\bz_{\BZ}^{\sigma}:=
 \bigl\{\bM \in \bz_{\BZ} \mid \sigma(\bM)=\bM \bigr\}.
\end{equation}
Let us define a crystal structure for $U_{q}(\ha{\Fg}^{\vee})$ 
on the set $\bz_{\BZ}^{\sigma}$ 
(see Proposition~\ref{prop:crystal} below). 

For $\bM \in \bz_{\BZ}^{\sigma}$, we set 
%
%%%%%%%%%%%%%%%%%
%%% eq:def-wt %%%
%%%%%%%%%%%%%%%%%
%
\begin{equation} \label{eq:def-wt}
\wt(\bM):=\sum_{i \in \ha{I}} M_{\Lambda_{i}} \ha{h}_{i},
\end{equation}
where $M_{\Lambda_{i}}:=\Theta(\bM)_{\Lambda_{i}}$ 
for $i \in \BZ$. 

In what follows, we need the following notation. 
Let $L$ be a finite subset of $\BZ$ such that 
$|q-q'| \ge 2$ for all $q,\,q' \in L$ 
with $q \ne q'$. Then, it follows from 
Lemma~\ref{lem:ef}\,(3) that 
$f_{q}f_{q'}=f_{q'}f_{q}$ and 
$e_{q}e_{q'}=e_{q'}e_{q}$ for all $q,\,q' \in L$. 
Hence we can define the following operator on 
$\bz_{\BZ} \cup \{\bzero\}$: 
\begin{equation*}
f_{L}:=\prod_{q \in L}f_{q} \quad \text{and} \quad
e_{L}:=\prod_{q \in L}e_{q}. 
\end{equation*}
For $\bM \in \bz_{\BZ}^{\sigma}$ and $p \in \BZ$, 
we define $\ha{f}_{p}\bM=
(M_{\gamma}')_{\gamma \in \Gamma_{\BZ}}$ by 
%
%%%%%%%%%%%%%%%%%%
%%% eq:def-haf %%%
%%%%%%%%%%%%%%%%%%
%
\begin{equation} \label{eq:def-haf}
(\ha{f}_{p}\bM)_{\gamma}=M_{\gamma}':= 
(f_{L(\gamma,p)}\bM)_{\gamma}
\quad \text{for $\gamma \in \Gamma_{\BZ}$}, 
\end{equation}
where we set
\begin{equation*}
L(\gamma,p):=\bigl\{ q \in p+(\ell+1)\BZ \mid 
 \pair{h_{q}}{\gamma} > 0 \bigr\}
\end{equation*}
for $\gamma \in \Gamma_{\BZ}$ and $p \in \ha{I}$; 
note that $L(\gamma,p)$ is a finite subset of $p+(\ell+1)\BZ$. 
It is obvious that 
if $p \in \BZ$ and $q \in \BZ$ are congruent 
modulo $\ell+1$, then 
%
%%%%%%%%%%%%%%%%
%%% eq:fp-fq %%%
%%%%%%%%%%%%%%%%
%
\begin{equation} \label{eq:fp-fq}
\ha{f}_{p}\bM=\ha{f}_{q}\bM 
\quad \text{for all $\bM \in \bz_{\BZ}^{\sigma}$}.
\end{equation} 
%
%%%%%%%%%%%%%%%
%%% rem:haf %%%
%%%%%%%%%%%%%%%
%
\begin{rem} \label{rem:haf}
Let $\bM \in \bz_{\BZ}^{\sigma}$, and $p \in \BZ$. 
For each $\gamma \in \Gamma_{\BZ}$, take an arbitrary 
finite subset $L$ of $p + (\ell+1)\BZ$ containing 
$L(\gamma,\,p)$. Then we have 
%
%%%%%%%%%%%%%%
%%% eq:haf %%%
%%%%%%%%%%%%%%
%
\begin{equation} \label{eq:haf}
(f_{L}\bM)_{\gamma}=
(f_{L(\gamma,p)}\bM)_{\gamma}=
(\ha{f}_{p}\bM)_{\gamma}.
\end{equation}
Indeed, we have 
$(f_{L}\bM)_{\gamma}=
(f_{L(\gamma,p)}f_{L \setminus L(\gamma,p)} \bM)_{\gamma}$. 
Since $\pair{h_{q}}{\gamma} \le 0$ for all 
$q \in L \setminus L(\gamma,p)$ by the definition of $L(\gamma,p)$, 
we deduce, using \eqref{eq:AM2} repeatedly, that 
$(f_{L(\gamma,p)}f_{L \setminus L(\gamma,p)} \bM)_{\gamma}=
 (f_{L(\gamma,p)}\bM)_{\gamma}$. 
\end{rem}
%
%%%%%%%%%%%%%%%%%%%%%%%
%%% prop:haf-stable %%%
%%%%%%%%%%%%%%%%%%%%%%%
%
\begin{prop} \label{prop:haf-stable}
Let $\bM \in \bz_{\BZ}^{\sigma}$, and $p \in \BZ$. 
Then, $\ha{f}_{p}\bM$ is an element of $\bz_{\BZ}^{\sigma}$. 
\end{prop}

By this proposition, for each $p \in \BZ$, 
we obtain a map $\ha{f}_{p}$ from $\bz_{\BZ}^{\sigma}$ 
to itself sending $\bM \in \bz_{\BZ}$ to 
$\ha{f}_{p}\bM \in \bz_{\BZ}$, 
which we call the lowering Kashiwara operator 
on $\bz_{\BZ}^{\sigma}$. 
By convention, we set $\ha{f}_{p}\bzero:=\bzero$ for all $p \in \BZ$. 

\begin{proof}[Proof of Proposition~\ref{prop:haf-stable}]
First we show that $\ha{f}_{p}\bM$ satisfies 
condition (a) of Definition~\ref{dfn:BZdatum2}. 
Let $K$ be an interval in $\BZ$. 
Take a finite subset $L$ of $p+(\ell+1)\BZ$ such that 
$L \supset L(\gamma,\,p)$ for all $\gamma \in \Gamma_{K}$. 
Then, we see from Remark~\ref{rem:haf} that 
$(\ha{f}_{p}\bM)_{\gamma} = (f_{L}\bM)_{\gamma}$ 
for all $\gamma \in \Gamma_{K}$, and hence that 
$(\ha{f}_{p}\bM)_{K} = (f_{L}\bM)_{K}$. 
Since $f_{L}\bM \in \bz_{\BZ}$ by 
Proposition~\ref{prop:fj}, 
it follows from condition (a) of 
Definition~\ref{dfn:BZdatum2} that 
$(f_{L}\bM)_{K} \in \bz_{K}$, and hence 
$(\ha{f}_{p}\bM)_{K} \in \bz_{K}$. 

Next we show that $\ha{f}_{p}\bM$ satisfies 
condition (b) of Definition~\ref{dfn:BZdatum2}.
Fix $w \in W_{\BZ}$ and $i \in \BZ$. We set 
%
%%%%%%%%%%%%
%%% eq:L %%%
%%%%%%%%%%%%
%
\begin{equation} \label{eq:L}
L:=
\begin{cases}
\bigl\{q \in p+(\ell+1)\BZ \mid w^{-1} h_{q} \ne h_{q}\bigr\}
 & \text{if $i \notin p+(\ell+1)\BZ$}, \\[3mm]
\bigl\{q \in p+(\ell+1)\BZ \mid w^{-1} h_{q} \ne h_{q}\bigr\} \cup \{i\}
 & \text{otherwise.}
\end{cases}
\end{equation}
It is easily checked that 
$L$ is a finite subset of $p+(\ell+1)\BZ$. 
Furthermore, we can verify that 
$L \supset L(w\vpi_{i}^{I},p)$ for all intervals 
$I$ in $\BZ$ such that $w \in W_{I}$ and $i \in I$.
Indeed, suppose that $q \in p+(\ell+1)\BZ$ is not contained in $L$;
note that $q \ne i$ and $w^{-1} h_{q} = h_{q}$. We see that 
\begin{equation*}
\pair{h_{q}}{w\vpi_{i}^{I}} = 
\pair{w^{-1}h_{q}}{\vpi_{i}^{I}}=
\pair{h_{q}}{\vpi_{i}^{I}}, 
\end{equation*}
and that $\pair{h_{q}}{\vpi_{i}^{I}} \le 0$ 
by \eqref{eq:pair-vpi} since $q \ne i$. 
This implies that $q$ is not contained in 
$L(w\vpi_{i}^{I},p)$.

Now, let us take $I \in \Int (f_{L}\bM; w,i)$, 
and let $J$ be an arbitrary interval in $\BZ$ containing $I$. 
We claim that 
$(\ha{f}_{p}\bM)_{w\vpi_{i}^{J}}=
 (\ha{f}_{p}\bM)_{w\vpi_{i}^{I}}$. 
Since $I \in \Int (f_{L}\bM; w,i)$, 
it follows that 
$(f_{L}\bM)_{w\vpi_{i}^{J}}=
(f_{L}\bM)_{w\vpi_{i}^{I}}$. 
Also, because $L \supset L(w\vpi_{i}^{J},p)$ and 
$L \supset L(w\vpi_{i}^{I},p)$ as seen above, 
we see from Remark~\ref{rem:haf} that 
$(\ha{f}_{p}\bM)_{w\vpi_{i}^{J}}=
(f_{L}\bM)_{w\vpi_{i}^{J}}$ and 
$(\ha{f}_{p}\bM)_{w\vpi_{i}^{I}}=
(f_{L}\bM)_{w\vpi_{i}^{I}}$.
Combining these, we obtain 
$(\ha{f}_{p}\bM)_{w\vpi_{i}^{J}}=
(f_{L}\bM)_{w\vpi_{i}^{J}} = 
(f_{L}\bM)_{w\vpi_{i}^{I}} = 
(\ha{f}_{p}\bM)_{w\vpi_{i}^{I}}$, 
as desired. 
Thus, we have shown that $\ha{f}_{p}\bM$ satisfies 
condition (b) of Definition~\ref{dfn:BZdatum2}, 
and hence $\ha{f}_{p}\bM \in \bz_{\BZ}$. 

Finally, we show that 
$\sigma(\ha{f}_{p}\bM)=\ha{f}_{p}\bM$, or equivalently, 
$(\ha{f}_{p}\bM)_{\sigma^{-1}(\gamma)}=
 (\ha{f}_{p}\bM)_{\gamma}$ for all $\gamma \in \Gamma_{\BZ}$.
Fix $\gamma \in \Gamma_{\BZ}$. 
Observe that 
$\sigma(L(\sigma^{-1}(\gamma),p))=L(\gamma,p)$
since $\pair{h_{\sigma(q)}}{\gamma}=
\pair{\sigma(h_{q})}{\gamma}=
\pair{h_{q}}{\sigma^{-1}(\gamma)}$. 
Therefore, we have 
\begin{align*}
(\ha{f}_{p}\bM)_{\sigma^{-1}(\gamma)} 
& = (f_{L(\sigma^{-1}(\gamma),p)}\bM)_{\sigma^{-1}(\gamma)}
  = \bigl(\sigma(f_{L(\sigma^{-1}(\gamma),p)}\bM)\bigr)_{\gamma} \\
& = \bigl(f_{\sigma(L(\sigma^{-1}(\gamma),p))}\sigma(\bM)\bigr)_{\gamma} 
    \quad \text{by Lemma~\ref{lem:sig-ef}\,(2)} \\
& = \bigl(f_{\sigma(L(\sigma^{-1}(\gamma),p))}\bM\bigr)_{\gamma} 
    \quad \text{by the assumption that $\sigma(\bM)=\bM$} \\
& = \bigl(f_{L(\gamma,p)}\bM\bigr)_{\gamma}
  \quad \text{since $\sigma(L(\sigma^{-1}(\gamma),p))=L(\gamma,p)$} \\
& = (\ha{f}_{p}\bM)_{\gamma},
\end{align*}
as desired. This completes the proof of the proposition. 
\end{proof}

Now, for $\bM \in \bz_{\BZ}^{\sigma}$
and $p \in \BZ$, we set
%
%%%%%%%%%%%%%%%%%%%
%%% eq:def-have %%%
%%%%%%%%%%%%%%%%%%%
%
\begin{equation} \label{eq:def-have}
\ha{\ve}_{p}(\bM):= - \left(
     M_{\Lambda_{p}}+M_{s_{p}\Lambda_{p}}
     +\sum_{q \in \BZ \setminus \{p\}} a_{qp} M_{\Lambda_{q}}
    \right)=\ve_{p}(\bM), 
\end{equation}
where $M_{\Lambda_{i}}:=\Theta(\bM)_{\Lambda_{i}}$ 
for $i \in \BZ$, and $M_{s_{p}\Lambda_{p}}:=
\Theta(\bM)_{s_{p}\Lambda_{p}}$. 
It follows from \eqref{eq:vej} that 
$\ha{\ve}_{p}(\bM)=\ve_{p}(\bM)$ is 
a nonnegative integer. 
Also, using Lemma~\ref{lem:sig-ef}\,(1) repeatedly, 
we can easily verify that 
if $p \in \BZ$ and $q \in \BZ$ are congruent modulo $\ell+1$, 
then 
%
%%%%%%%%%%%%%%%
%%% eq:have %%%
%%%%%%%%%%%%%%%
%
\begin{equation} \label{eq:have}
\ha{\ve}_{p}(\bM)=\ve_{p}(\bM)=\ve_{q}(\bM)=\ha{\ve}_{q}(\bM)
\quad \text{for all $\bM \in \bz_{\BZ}^{\sigma}$}. 
\end{equation}
%
%%%%%%%%%%%%%%%%%%%%
%%% lem:def-hae1 %%%
%%%%%%%%%%%%%%%%%%%%
%
\begin{lem} \label{lem:def-hae1}
Let $\bM \in \bz_{\BZ}^{\sigma}$, and $p \in \BZ$. 
Suppose that $\ha{\ve}_{p}(\bM) > 0$. Then, 
$e_{L}\bM \ne \bzero$ for every finite subset $L$ of $p+(\ell+1)\BZ$. 
\end{lem}

\begin{proof}
We show by induction on the cardinality $|L|$ of $L$ that 
$e_{L}\bM \ne \bzero$, and 
$\ve_{q}(e_{L}\bM)=\ha{\ve}_{p}(\bM) > 0$ 
for all $q \in p+(\ell+1)\BZ$ with $q \notin L$. 
Assume first that $|L|=1$. Then, 
$L=\{q'\}$ for some $q' \in p+(\ell+1)\BZ$, 
and $e_{L}=e_{q'}$. 
It follows from \eqref{eq:have}
that $\ve_{q'}(\bM)=\ha{\ve}_{p}(\bM) > 0$, 
which implies that $e_{q'}\bM \ne \bzero$. 
Also, for $q \in p+(\ell+1)\BZ$ with $q \ne q'$, 
it follows from Lemma~\ref{lem:ef}\,(2) and 
\eqref{eq:have} that 
$\ve_{q}(e_{q'}\bM)=\ve_{q}(\bM)=\ha{\ve}_{p}(\bM)$. 

Assume next that $|L| > 1$. 
Take an arbitrary $q' \in L$, and set $L':=L \setminus \{q'\}$. 
Then, by the induction hypothesis, 
we have $e_{L'}\bM \ne \bzero$, and 
$\ve_{q'}(e_{L'}\bM)=\ha{\ve}_{p}(\bM) > 0$; 
note that $q' \notin L'$. 
This implies that $e_{L}\bM=e_{q'}(e_{L'}\bM) \ne \bzero$. 
Also, for $q \in p+(\ell+1)\BZ$ with $q \notin L$, 
we see from Lemma~\ref{lem:ef}\,(2) and 
the induction hypothesis that 
$\ve_{q}(e_{L}\bM)=\ve_{q}(e_{q'}e_{L'}\bM)=
\ve_{q}(e_{L'}\bM)=\ha{\ve}_{p}(\bM)$. 
This proves the lemma. 
\end{proof}

For $\bM \in \bz_{\BZ}^{\sigma}$ and $p \in \BZ$, 
we define $\ha{e}_{p}\bM$ as follows.
If $\ha{\ve}_{p}(\bM)=0$, then 
we set $\ha{e}_{p}\bM:=\bzero$. 
If $\ha{\ve}_{p}(\bM) > 0$, then 
we define $\ha{e}_{p}\bM=
(M_{\gamma}')_{\gamma \in \Gamma_{\BZ}}$ by 
%
%%%%%%%%%%%%%%%%%%
%%% eq:def-hae %%%
%%%%%%%%%%%%%%%%%%
%
\begin{equation} \label{eq:def-hae}
(\ha{e}_{p}\bM)_{\gamma}=M_{\gamma}':= 
(e_{L(\gamma,p)}\bM)_{\gamma}
\quad \text{for each $\gamma \in \Gamma_{\BZ}$};
\end{equation}
note that $e_{L(\gamma,p)}\bM \ne \bzero$ 
by Lemma~\ref{lem:def-hae1}. 
It is easily seen by \eqref{eq:have} that 
if $p \in \BZ$ and $q \in \BZ$ are congruent 
modulo $\ell+1$, then 
%
%%%%%%%%%%%%%%%%
%%% eq:ep-eq %%%
%%%%%%%%%%%%%%%%
%
\begin{equation} \label{eq:ep-eq}
\ha{e}_{p}\bM=\ha{e}_{q}\bM
\quad \text{for all $\bM \in \bz_{\BZ}^{\sigma}$}.
\end{equation}
%
%%%%%%%%%%%%%%%
%%% rem:hae %%%
%%%%%%%%%%%%%%%
%
\begin{rem} \label{rem:hae}
Let $\bM \in \bz_{\BZ}^{\sigma}$, and $p \in \BZ$. 
Assume that $\ha{\ve}_{p}(\bM) > 0$, or equivalently, 
$\ha{e}_{p}\bM \ne \bzero$. 
For each $\gamma \in \Gamma_{\BZ}$, take an arbitrary 
finite subset $L$ of $p + (\ell+1)\BZ$ containing $L(\gamma,\,p)$. 
Then we see by Lemma~\ref{lem:def-hae1} that 
$e_{L}\bM \ne \bzero$. Moreover, 
by the same argument as for \eqref{eq:haf} 
(using \eqref{eq:ep-gamma} instead of \eqref{eq:AM2}), 
we derive
%
%%%%%%%%%%%%%%
%%% eq:hae %%%
%%%%%%%%%%%%%%
%
\begin{equation} \label{eq:hae}
(e_{L}\bM)_{\gamma}=
(e_{L(\gamma,p)}\bM)_{\gamma}=
(\ha{e}_{p}\bM)_{\gamma}.
\end{equation}
\end{rem}
%
%%%%%%%%%%%%%%%%%%%%%%%
%%% prop:hae-stable %%%
%%%%%%%%%%%%%%%%%%%%%%%
%
\begin{prop} \label{prop:hae-stable}
Let $\bM \in \bz_{\BZ}^{\sigma}$, and $p \in \BZ$. 
Then, $\ha{e}_{p}\bM$ is contained in 
$\bz_{\BZ}^{\sigma} \cup \{\bzero\}$. 
\end{prop}

Because the proof of this proposition is similar to that of 
Proposition~\ref{prop:haf-stable}, we omit it. By this proposition, 
for each $p \in \BZ$, we obtain a map $\ha{e}_{p}$ from $\bz_{\BZ}^{\sigma}$ 
to $\bz_{\BZ}^{\sigma} \cup \{\bzero\}$ sending $\bM \in \bz_{\BZ}$ to 
$\ha{e}_{p}\bM \in \bz_{\BZ} \cup \{\bzero\}$, 
which we call the raising Kashiwara operator on $\bz_{\BZ}^{\sigma}$. 
By convention, we set $\ha{e}_{p}\bzero:=\bzero$ for all $p \in \BZ$. 

Finally, we set 
%
%%%%%%%%%%%%%%%%%%%
%%% eq:def-havp %%%
%%%%%%%%%%%%%%%%%%%
%
\begin{equation} \label{eq:def-havp}
\ha{\vp}_{p}(\bM):=
 \pair{\wt(\bM)}{\ha{\alpha}_{\ol{p}}}+\ha{\ve}_{p}(\bM)
\quad
\text{for $\bM \in \bz_{\BZ}^{\sigma}$ and $p \in \BZ$},
\end{equation} 
where $\ol{p}$ denotes a unique element in 
$\ha{I}=\bigl\{0,\,1,\,\dots,\,\ell\bigr\}$ 
to which $p \in \BZ$ is congruent modulo $\ell+1$. 
%
%%%%%%%%%%%%%%%%%%%%
%%% prop:crystal %%%
%%%%%%%%%%%%%%%%%%%%
%
\begin{prop} \label{prop:crystal}
The set $\bz_{\BZ}^{\sigma}$, 
equipped with the maps 
$\wt$, $\ha{e}_{p},\,\ha{f}_{p} \ (p \in \ha{I})$, and 
$\ha{\ve}_{p},\,\ha{\vp}_{p} \ (p \in \ha{I})$ above, is 
a crystal for $U_{q}(\ha{\Fg}^{\vee})$. 
\end{prop}

\begin{proof}
It is obvious from \eqref{eq:def-havp} that 
$\ha{\vp}_{p}(\bM)=
 \pair{\wt(\bM)}{\ha{\alpha}_{p}}+\ha{\ve}_{p}(\bM)$ 
for $\bM \in \bz_{\BZ}^{\sigma}$ and $p \in \ha{I}$ 
(see condition (1) of \cite[Definition~4.5.1]{HK}). 

We show that $\wt(\ha{f}_{p}\bM)=\wt(\bM)-\ha{h}_{p}$ 
for $\bM \in \bz_{\BZ}^{\sigma}$ and $p \in \ha{I}$  
(see condition (3) of \cite[Definition~4.5.1]{HK}). 
Write $\bM$, $f_{p}\bM$, and $\ha{f}_{p}\bM$ as: 
$\bM=(M_{\gamma})_{\gamma \in \Gamma_{\BZ}}$, 
$f_{p}\bM=(M_{\gamma}')_{\gamma \in \Gamma_{\BZ}}$, and 
$\ha{f}_{p}\bM=(M_{\gamma}'')_{\gamma \in \Gamma_{\BZ}}$, respectively; 
write $\Theta(\bM)$, $\Theta(f_{p}\bM)$, and 
$\Theta(\ha{f}_{p}\bM)$ as: 
$\Theta(\bM)=(M_{\xi})_{\xi \in \Xi_{\BZ}}$, 
$\Theta(f_{p}\bM)=(M_{\xi}')_{\xi \in \Xi_{\BZ}}$, and 
$\Theta(\ha{f}_{p}\bM)=(M_{\xi}'')_{\xi \in \Xi_{\BZ}}$, respectively.
We claim that $M_{\Lambda_{i}}''=M_{\Lambda_{i}}'$ 
for all $i \in \BZ$. Fix $i \in \BZ$, and take 
an interval $I$ in $\BZ$ such that 
$I \in \Int(\ha{f}_{p}\bM;e,i) \cap \Int(f_{p}\bM;e,i)$. 
Then, we have 
$M_{\Lambda_{i}}''=M_{\vpi_{i}^{I}}''=
 (\ha{f}_{p}\bM)_{\vpi_{i}^{I}}$, 
and $M_{\Lambda_{i}}'=M_{\vpi_{i}^{I}}'$ by the definitions. 
Also, since $L(\vpi_{i}^{I},p) \subset \{p\}$ 
by \eqref{eq:pair-vpi}, it follows from 
Remark~\ref{rem:haf} that 
$(\ha{f}_{p}\bM)_{\vpi_{i}^{I}}=
 (f_{p}\bM)_{\vpi_{i}^{I}}=M_{\vpi_{i}^{I}}'$. 
Combining these, we infer that 
$M_{\Lambda_{i}}''=M_{\Lambda_{i}}'$, as desired.  
Therefore, we see from \eqref{eq:lam-i} that 
%
%%%%%%%%%%%%%%%
%%% eq:cry0 %%%
%%%%%%%%%%%%%%%
%
\begin{equation} \label{eq:cry0}
M_{\Lambda_{i}}''=
M_{\Lambda_{i}}'=
\begin{cases}
M_{\Lambda_{p}}-1 & \text{if $i=p$}, \\[1.5mm]
M_{\Lambda_{i}} & \text{otherwise}.
\end{cases}
\end{equation}
The equation $\wt(\ha{f}_{p}\bM)=\wt(\bM)-\ha{h}_{p}$ 
follows immediately from \eqref{eq:cry0} and 
the definition \eqref{eq:def-wt} of the map $\wt$. 

Similarly, we can show that 
$\wt(\ha{e}_{p}\bM)=\wt(\bM)+\ha{h}_{p}$ 
for $\bM \in \bz_{\BZ}^{\sigma}$ and $p \in \ha{I}$ 
if $\ha{e}_{p}\bM \ne \bzero$  
(see condition (2) of \cite[Definition~4.5.1]{HK}). 

Let us show that $\ha{\ve}_{p}(\ha{f}_{p}\bM)=\ha{\ve}_{p}(\bM)+1$ 
and $\ha{\vp}_{p}(\ha{f}_{p}\bM)=\ha{\vp}_{p}(\bM)-1$ for 
$\bM \in \bz_{\BZ}^{\sigma}$ and $p \in \ha{I}$  
(see condition (5) of \cite[Definition~4.5.1]{HK}). 
The second equation follows immediately from the first one and 
the definition \eqref{eq:def-havp} of the map $\ha{\vp}$, 
since $\wt(\ha{f}_{p}\bM)=\wt(\bM)-\ha{h}_{p}$ as shown above. 
It, therefore, suffices to show the first equation; 
to do this, we use the notation above. 
We claim that $M_{s_{p}\Lambda_{p}}''=
M_{s_{p}\Lambda_{p}}'=M_{s_{p}\Lambda_{p}}$. 
Indeed, let $I$ be an interval in $\BZ$ such that 
$I \in \Int (\ha{f}_{p}\bM; s_{p},p) \cap \Int (f_{p}\bM; s_{p},p)$. 
Then, in exactly the same way as above, 
we see that 
\begin{align*}
M_{s_{p}\Lambda_{p}}'' & 
 =M_{s_{p}\vpi_{p}^{I}}''
 =(\ha{f}_{p}\bM)_{s_{p}\vpi_{p}^{I}} \\
& =(f_{p}\bM)_{s_{p}\vpi_{p}^{I}}
\quad \text{by Remark~\ref{rem:haf} (note that 
  $L(s_{p}\vpi_{p}^{I},\,p)=\emptyset$ by \eqref{eq:pair-vpi})} \\
& =M_{s_{p}\vpi_{p}^{I}}'
  =M_{s_{p}\Lambda_{p}}'.
\end{align*}
In addition, the equality 
$M_{s_{p}\Lambda_{p}}'=M_{s_{p}\Lambda_{p}}$ 
follows from \eqref{eq:silami}. 
Hence we get $M_{s_{p}\Lambda_{p}}''=
M_{s_{p}\Lambda_{p}}$, as desired.  
Using this and \eqref{eq:cry0}, we deduce 
from the definition \eqref{eq:def-have} of the 
map $\ha{\ve}_{p}$ that 
$\ha{\ve}_{p}(\ha{f}_{p}\bM)=\ha{\ve}_{p}(\bM)+1$. 

Similarly, we can show that 
$\ha{\ve}_{p}(\ha{e}_{p}\bM)=\ha{\ve}_{p}(\bM)-1$ 
and $\ha{\vp}_{p}(\ha{e}_{p}\bM)=\ha{\vp}_{p}(\bM)+1$ for 
$\bM \in \bz_{\BZ}^{\sigma}$ and $p \in \ha{I}$  
if $\ha{e}_{p}\bM \ne \bzero$ 
(see condition (4) of \cite[Definition~4.5.1]{HK}). 

Finally, we show that $\ha{e}_{p}\ha{f}_{p}\bM=\bM$ 
for $\bM \in \bz_{\BZ}^{\sigma}$ and $p \in \ha{I}$, and 
that $\ha{f}_{p}\ha{e}_{p}\bM=\bM$ 
for $\bM \in \bz_{\BZ}^{\sigma}$ and $p \in \ha{I}$ 
if $\ha{e}_{p}\bM \ne \bzero$ 
(see condition (6) of \cite[Definition~4.5.1]{HK}). 
We give a proof only for the first equation, 
since the proof of the second one is similar. 
Write $\bM \in \bz_{\BZ}^{\sigma}$ as: 
$\bM=(M_{\gamma})_{\gamma \in \Gamma_{\BZ}}$. 
Note that $\ha{e}_{p}\ha{f}_{p}\bM \ne \bzero$, 
since $\ha{\ve}_{p}(\ha{f}_{p}\bM)=
\ha{\ve}_{p}(\bM)+1 > 0$. 
We need to show that 
$(\ha{e}_{p}\ha{f}_{p}\bM)_{\gamma}=M_{\gamma}$ 
for all $\gamma \in \Gamma_{\BZ}$. 
Fix $\gamma \in \Gamma_{\BZ}$. 
We deduce from Lemma~\ref{lem:tech2} below that
\begin{equation*}
(\ha{e}_{p}\ha{f}_{p}\bM)_{\gamma}=
(e_{L(\gamma,p)}f_{L(\gamma,p)}\bM)_{\gamma}. 
\end{equation*}
Therefore, it follows from Lemma~\ref{lem:ef}\,(1) and (3) that 
$e_{L(\gamma,p)}f_{L(\gamma,p)}\bM=\bM$. 
Hence we obtain 
$(\ha{e}_{p}\ha{f}_{p}\bM)_{\gamma}=M_{\gamma}$. 
Thus, we have shown that $\ha{e}_{p}\ha{f}_{p}\bM=\bM$, 
thereby completing the proof of the proposition. 
\end{proof}
%
%%%%%%%%%%%%%%%%%%%%
%%% rem:have-max %%%
%%%%%%%%%%%%%%%%%%%%
%
\begin{rem} \label{rem:have-max}
Let $\bM \in \bz_{\BZ}^{\sigma}$, and $p \in \ha{I}$. 
From the definition, it follows that 
$\ha{\ve}_{p}(\bM)=0$ if and only if 
$\ha{e}_{p}\bM=\bzero$, and that
$\ha{\ve}_{p}(\bM) \in \BZ_{\ge 0}$. 
In addition, $\ha{\ve}_{p}(\ha{e}_{p}\bM)=
\ha{\ve}_{p}(\bM)-1$. 
Consequently, we deduce that 
$\ha{\ve}_{p}(\bM)=
 \max \bigl\{N \ge 0 \mid 
 \ha{e}_{p}^{N}\bM \ne \bzero\bigr\}$. 
Moreover, by \eqref{eq:have} and \eqref{eq:ep-eq}, 
the same is true for all $p \in \BZ$. 
\end{rem}

The following lemma will be needed 
in the proof of Lemma~\ref{lem:tech2} below. 
%
%%%%%%%%%%%%%%%%
%%% lem:tech %%%
%%%%%%%%%%%%%%%%
%
\begin{lem} \label{lem:tech}
Let $K$ be an interval in $\BZ$, and 
let $X$ be a product of Kashiwara operators of the form{\rm:}
$X=x_{1}x_{2} \cdots x_{a}$, where $x_{b} \in 
\bigl\{f_{q},\, e_{q} \mid \min K < q < \max K\bigr\}$ 
for each $1 \le b \le a$.
If $\bM \in \bz_{\BZ}^{\sigma}$ and 
$X\ha{y}_{p} \bM \ne \bzero$ for some $p \in \BZ$, 
where $\ha{y}_{p}=\ha{e}_{p}$ or $\ha{f}_{p}$, then 
there exists a finite subset $L_{0}$ of $p+(\ell+1)\BZ$ 
such that $Xy_{L}\bM \ne \bzero$ and 
$(X\ha{y}_{p}\bM)_{K}=(Xy_{L}\bM)_{K}$ 
for every finite subset $L$ of $p+(\ell+1)\BZ$ 
containing $L_{0}$, 
where $y_{L}=e_{L}$ if $\ha{y}_{p}=\ha{e}_{p}$, and 
$y_{L}=f_{L}$ if $\ha{y}_{p}=\ha{f}_{p}$. 
\end{lem}

\begin{proof}
Note that $\ha{y}_{p}\bM \ne \bzero$ since 
$X \ha{y}_{p} \bM \ne \bzero$ by our assumption.
Let $I$ be an interval in $\BZ$ containing $K$ such that 
$I \in \Int (\ha{y}_{p}\bM;v,k)$ 
for all $v \in W_{K}$ and $k \in K$, and such that 
$\min I < \min K \le \max K < \max I$. 
Then, we have $\ha{y}_{p}\bM \in \bz_{\BZ}(I,K)$ 
(for the definition of $\bz_{\BZ}(I,K)$, 
 see the paragraph following Remark~\ref{rem:fj}). 
Because $X$ is a product of those Kashiwara operators 
which are taken from the set 
$\bigl\{f_{q},\, e_{q} \mid \min K < q < \max K\bigr\}$, 
it follows from Lemmas~\ref{lem:IK-fj}\,(2) and \ref{lem:IK-ej}\,(2) that 
%
%%%%%%%%%%%%%%%%
%%% eq:haef1 %%%
%%%%%%%%%%%%%%%%
%
\begin{equation} \label{eq:haef1}
X (\ha{y}_{p} \bM)_{I} \ne \bzero \quad \text{and} \quad
(X \ha{y}_{p} \bM)_{I}=X (\ha{y}_{p} \bM)_{I}. 
\end{equation}

Now, we set $L_{0}:=\bigcup_{\zeta \in \Gamma_{I}} L(\zeta,p)$, and 
take an arbitrary finite subset $L$ of $p+(\ell+1)\BZ$ containing $L_{0}$. 
Then, we see from Remark~\ref{rem:haf} (if $\ha{y}_{p}=\ha{f}_{p}$) 
or Remark~\ref{rem:hae} (if $\ha{y}_{p}=\ha{e}_{p}$) that 
%
%%%%%%%%%%%%%%%%%
%%% eq:L-zeta %%%
%%%%%%%%%%%%%%%%%
%
\begin{equation} \label{eq:L-zeta}
(\ha{y}_{p}\bM)_{\zeta}=(y_{L}\bM)_{\zeta}
\quad \text{for all $\zeta \in \Gamma_{I}$},
\end{equation}
which implies that 
$(\ha{y}_{p}\bM)_{I}=(y_{L}\bM)_{I}$. 
Combining this and \eqref{eq:haef1}, we obtain 
%
%%%%%%%%%%%%%%%%
%%% eq:haef2 %%%
%%%%%%%%%%%%%%%%
%
\begin{equation} \label{eq:haef2}
X (y_{L} \bM)_{I} \ne \bzero \quad \text{and} \quad
(X \ha{y}_{p} \bM)_{I}= X (y_{L}\bM)_{I}. 
\end{equation}
We show that $I \in \Int (y_{L}\bM;v,k)$ 
for all $v \in W_{K}$ and $k \in K$. 
To do this, we need the following claim. 

\begin{claim*}
Keep the notation above. 
If $J$ is an interval in $\BZ$ containing $I$, then 
$L(v\vpi_{k}^{J},p)=L(v\vpi_{k}^{I},p)$ 
for all $v \in W_{K}$ and $k \in K$. 
\end{claim*}

\noindent
{\it Proof of Claim.}
Fix $v \in W_{K}$ and $k \in K$. 
First, let us show that if $q \in p+(\ell+1)\BZ$ is 
not contained in $I$, then $q$ is contained neither 
in $L(v\vpi_{k}^{J},p)$ nor in $L(v\vpi_{k}^{I},p)$. 
Because $\min I < \min K$ and $\max I > \max K$, 
we have $q < (\min K)-1$ or $q > (\max K)+1$. 
Hence it follows that $v^{-1}h_{q}=h_{q}$ since $v \in W_{K}$. 
Also, note that $q \ne k$ since $k \in K \subset I$.
Therefore, we see that 
$\pair{h_{q}}{v\vpi_{k}^{J}}=\pair{h_{q}}{\vpi_{k}^{J}} \le 0$ and 
$\pair{h_{q}}{v\vpi_{k}^{I}}=\pair{h_{q}}{\vpi_{k}^{I}} \le 0$ 
by \eqref{eq:pair-vpi}, which implies that 
$q \notin L(v\vpi_{k}^{J},p)$ and $q \notin L(v\vpi_{k}^{I},p)$. 

Next, let us consider the case that 
$q \in p+(\ell+1)\BZ$ is contained in $I$. 
In this case, we have $v^{-1}h_{q} \in 
\bigoplus_{i \in I} \BZ h_{i} \subset 
\bigoplus_{i \in J} \BZ h_{i}$, and hence 
$\pair{h_{q}}{v\vpi_{k}^{J}}=\pair{v^{-1}h_{q}}{\vpi_{k}^{J}}=
 \pair{v^{-1}h_{q}}{\vpi_{k}^{I}}=\pair{h_{q}}{v\vpi_{k}^{I}}$ 
by \eqref{eq:pair-vpi}. 
In particular, 
$\pair{h_{q}}{v\vpi_{k}^{J}} > 0$
if and only if 
$\pair{h_{q}}{v\vpi_{k}^{I}} > 0$. 
Therefore, $q \in L(v\vpi_{k}^{J},p)$ if and only if 
$q \in L(v\vpi_{k}^{I},p)$. This proves the claim. \bqed

%%%%
\vsp
%%%%

Fix $v \in W_{K}$ and $k \in K$, 
and let $J$ be an arbitrary interval in $\BZ$ containing $I$. 
We verify that 
$(y_{L}\bM)_{v\vpi_{k}^{J}}=
 (y_{L}\bM)_{v\vpi_{k}^{I}}$. 
Since $I \in \Int (\ha{y}_{p}\bM;v,k)$ 
by assumption, it follows that 
$(\ha{y}_{p}\bM)_{v\vpi_{k}^{J}}=
 (\ha{y}_{p}\bM)_{v\vpi_{k}^{I}}$. 
Note that 
$(\ha{y}_{p}\bM)_{v\vpi_{k}^{I}}=
 (y_{L}\bM)_{v\vpi_{k}^{I}}$ by \eqref{eq:L-zeta} 
since $v\vpi_{k}^{I} \in \Gamma_{I}$. 
Also, it follows from the claim above that 
$L(v\vpi_{k}^{J},p)=L(v\vpi_{k}^{I},p) \subset L_{0} \subset L$. 
Hence we see again from 
Remark~\ref{rem:haf} (if $\ha{y}_{p}=\ha{f}_{p}$) 
or Remark~\ref{rem:hae} (if $\ha{y}_{p}=\ha{e}_{p}$) that 
$(\ha{y}_{p}\bM)_{v\vpi_{k}^{J}}=
 (y_{L}\bM)_{v\vpi_{k}^{J}}$. 
Combining these, we obtain 
$(y_{L}\bM)_{v\vpi_{k}^{J}}=
(\ha{y}_{p}\bM)_{v\vpi_{k}^{J}}=
(\ha{y}_{p}\bM)_{v\vpi_{k}^{I}}=
(y_{L}\bM)_{v\vpi_{k}^{I}}$, as desired. 
Thus we have shown that $I \in \Int (y_{L}\bM;v,k)$ 
for all $v \in W_{K}$ and $k \in K$, which implies that 
$y_{L}\bM \in \bz_{\BZ}(I,K)$. 

Here we recall that $X$ is 
a product of those Kashiwara operators which are taken from the set 
$\bigl\{f_{q},\, e_{q} \mid \min K < q < \max K\bigr\}$ by assumption, 
and that $X(y_{L}\bM)_{I} \ne \bzero$ by \eqref{eq:haef2}. 
Therefore, we deduce again from 
Lemmas~\ref{lem:IK-fj}\,(2) and \ref{lem:IK-ej}\,(2) that 
$Xy_{L}\bM \ne \bzero$, and 
$X(y_{L}\bM)_{I}=(Xy_{L}\bM)_{I}$. 
Combining this and \eqref{eq:haef2}, 
we obtain 
$(X\ha{y}_{p} \bM)_{I}=(Xy_{L}\bM)_{I}$. 
Since $K \subset I$ (recall the correspondences 
\eqref{eq:bij-index1} and \eqref{eq:bij-index3}), it follows that 
\begin{equation*}
(X\ha{y}_{p} \bM)_{K}=
\bigl((X\ha{y}_{p}\bM)_{I}\bigr)_{K}=
\bigl((Xy_{L}\bM)_{I}\bigr)_{K}=
(Xy_{L}\bM)_{K}. 
\end{equation*}
This completes the proof of the lemma. 
\end{proof}

We used the following lemma in the proof of 
Proposition~\ref{prop:crystal} above; 
we will also use this lemma in the proof of 
Theorem~\ref{thm:main2} below. 
%
%%%%%%%%%%%%%%%%%
%%% lem:tech2 %%%
%%%%%%%%%%%%%%%%%
%
\begin{lem} \label{lem:tech2}
Let $p,\,q \in \BZ$ be such that $0 < |p-q| < \ell$, and let 
$\ha{X}$ be a product of Kashiwara operators of the form\,{\rm:}
$\ha{X}=\ha{x}_{1}\ha{x}_{2} \cdots \ha{x}_{a}$, where 
$\ha{x}_{b} \in 
 \bigl\{\ha{e}_{p},\,\ha{f}_{p},\,\ha{e}_{q},\,\ha{f}_{q}\bigr\}$
for each $1 \le b \le a$. 
If $\bM \in \bz_{\BZ}^{\sigma}$ and $\ha{X}\bM \ne \bzero$, 
then $X\bM \ne \bzero$, and 
$(\ha{X}\bM)_{\gamma}=(X\bM)_{\gamma}$ for each $\gamma \in \Gamma_{\BZ}$,
where $X$ is a product of Kashiwara operators of the form 
$X:=x_{1}x_{2} \cdots x_{a}$, with 
%
%%%%%%%%%%%%
%%% eq:xb %%
%%%%%%%%%%%%
%
\begin{equation} \label{eq:xb}
x_{b}=
 \begin{cases}
 e_{L_{p}} & \text{\rm if $\ha{x}_{b}=\ha{e}_{p}$},  \\[1.5mm]
 f_{L_{p}} & \text{\rm if $\ha{x}_{b}=\ha{f}_{p}$},  \\[1.5mm]
 e_{L_{q}} & \text{\rm if $\ha{x}_{b}=\ha{e}_{q}$},  \\[1.5mm]
 f_{L_{q}} & \text{\rm if $\ha{x}_{b}=\ha{f}_{q}$},
 \end{cases}
\end{equation}
for each $1 \le b \le a$. Here, 
$L_{p}$ is an arbitrary finite subset of $p+(\ell+1)\BZ$ 
such that $L_{p} \supset L(\gamma,p)$ and such that 
$L_{q}:=\bigl\{t+(q-p) \mid t \in L_{p} \bigr\} 
 \supset L(\gamma,q)$.
\end{lem}
%
%%%%%%%%%%%%%%%%
%%% rem:dist %%%
%%%%%%%%%%%%%%%%
%
\begin{rem} \label{rem:dist}
Keep the notation and assumptions of Lemma~\ref{lem:tech2}. 
If $r \in p+(\ell+1)\BZ$ is not contained in $L_{p}$, 
then $|r-t| \ge 2$ for all $t \in L_{p} \cup L_{q}$. Indeed, 
if $t \in L_{p}$, then it is obvious that 
$|r-t| \ge \ell+1 > 2$. If $t \in L_{q}$, then 
\begin{equation*}
|r-t| 
  = \bigl|r-\{t+(p-q)\}+(p-q)\bigr|
   \ge \bigl|r-\{t+(p-q)\}\bigr| - |p-q|.
\end{equation*}
Here note that $\bigl|r-\{t+(p-q)\}\bigr| \ge \ell+1$ 
since $t+(p-q) \in L_{p}$, and that $|p-q| < \ell$ by assumption. 
Therefore, we get $|r-t| \ge 2$. 
Similarly, we can show that 
if $r \in q+(\ell+1)\BZ$ is not contained in $L_{q}$, 
then $|r-t| \ge 2$ for all $t \in L_{p} \cup L_{q}$.
\end{rem}

\begin{proof}[Proof of Lemma~\ref{lem:tech2}]
For each $1 \le b \le a$, we set 
$\ha{X}_{b}:=\ha{x}_{b+1}\ha{x}_{b+2} \cdots \ha{x}_{a}$ and 
$X_{b}:=x_{1}x_{2} \cdots x_{b}$. 
We prove by induction on $b$ the claim 
that $X_{b}\ha{X}_{b}\bM \ne \bzero$ and 
$(\ha{X}\bM)_{\gamma} = (X_{b}\ha{X}_{b}\bM)_{\gamma}$ 
for all $1 \le b \le a$; the assertion of the lemma 
follows from the case $b=a$. 
We see easily from Remark~\ref{rem:haf} 
(if $\ha{x}_{1}=\ha{f}_{p}$ or $\ha{f}_{q}$) or 
Remark~\ref{rem:hae}
(if $\ha{x}_{1}=\ha{e}_{p}$ or $\ha{e}_{q}$) that 
the claim above holds if $b=1$. 
Assume, therefore, that $b > 1$. 
By the induction hypothesis, we have 
%
%%%%%%%%%%%%%%%%%%
%%% eq:tech2-1 %%%
%%%%%%%%%%%%%%%%%%
%
\begin{equation} \label{eq:tech2-1}
X_{b-1}\ha{X}_{b-1}\bM=
X_{b-1}\ha{x}_{b}\ha{X}_{b}\bM \ne \bzero
\quad \text{and} \quad
(\ha{X}\bM)_{\gamma}=
 (X_{b-1}\ha{x}_{b}\ha{X}_{b}\bM)_{\gamma}.
\end{equation}
Take an interval $K$ in $\BZ$ 
such that $\gamma \in \Gamma_{K}$, and such that 
$\min K < t < \max K$ for all $t \in L_{p} \cup L_{q}$. 
Define $r \in \{p,\,q\}$ by: 
$r=p$ if $\ha{x}_{b}=\ha{e}_{p}$ or $\ha{f}_{p}$, and 
$r=q$ if $\ha{x}_{b}=\ha{e}_{q}$ or $\ha{f}_{q}$. 
Then we deduce from Lemma~\ref{lem:tech} that 
there exists a finite subset $L$ of $r+(\ell+1)\BZ$ 
such that 
\begin{equation*}
X_{b-1}x_{b}'\ha{X}_{b}\bM \ne \bzero
\quad \text{and} \quad
(X_{b-1}\ha{x}_{b}\ha{X}_{b}\bM)_{K}=
(X_{b-1}x_{b}'\ha{X}_{b}\bM)_{K},
\end{equation*}
where $x_{b}'$ is defined by the formula \eqref{eq:xb}, 
with $L_{p}$ and $L_{q}$ replaced by 
$L \cup L_{p}$ and $L \cup L_{q}$, respectively.
Also, it follows from 
Remark~\ref{rem:dist} and Lemma~\ref{lem:ef}\,(3) that 
\begin{equation*}
(\bzero \ne \,) \quad 
X_{b-1}x_{b}'\ha{X}_{b}\bM = 
X_{b-1}x_{b}''x_{b}\ha{X}_{b}\bM = 
x_{b}''X_{b-1}x_{b}\ha{X}_{b}\bM=
x_{b}''X_{b}\ha{X}_{b}\bM,
\end{equation*}
where $x_{b}''$ is defined by the formula \eqref{eq:xb}, 
with $L_{p}$ and $L_{q}$ replaced by 
$L \setminus L_{p}$ and $L \setminus L_{q}$, respectively.
In particular, we obtain 
$X_{b}\ha{X}_{b}\bM \ne \bzero$. 
Moreover, since $\gamma \in \Gamma_{K}$, we have 
\begin{equation*}
(X_{b-1}\ha{x}_{b}\ha{X}_{b}\bM)_{\gamma}=
(X_{b-1}x_{b}'\ha{X}_{b}\bM)_{\gamma} = 
(x_{b}''X_{b}\ha{X}_{b}\bM)_{\gamma}.
\end{equation*}
Since $L_{r} \supset L(\gamma,r)$, the intersection of 
$L \setminus L_{r}$ and $L(\gamma,r)$ is empty, and hence 
$\pair{h_{t}}{\gamma} \le 0$ for all $t \in L \setminus L_{r}$. 
Therefore, we see from \eqref{eq:AM2} 
(if $\ha{x}_{1}=\ha{f}_{p}$ or $\ha{f}_{q}$) or 
\eqref{eq:ep-gamma} 
(if $\ha{x}_{1}=\ha{e}_{p}$ or $\ha{e}_{q}$) that 
$(x_{b}''X_{b}\ha{X}_{b}\bM)_{\gamma}=
 (X_{b}\ha{X}_{b}\bM)_{\gamma}$.
Combining these with \eqref{eq:tech2-1}, we conclude that 
$(\ha{X}\bM)_{\gamma} = (X_{b}\ha{X}_{b}\bM)_{\gamma}$, 
as desired. This proves the lemma.
\end{proof}

%==============================%
%     START SUBSECTION 0404    %
%==============================%
%
\subsection{Main results.}
\label{subsec:main}

Recall the BZ datum $\bO$ of type $A_{\infty}$ 
whose $\gamma$-component is equal to $0$ for each 
$\gamma \in \Gamma_{\BZ}$ (see Example~\ref{ex:O}).
It is obvious that $\sigma(\bO)=\bO$, and hence 
$\bO \in \bz_{\BZ}^{\sigma}$.
Also, $\ha{\ve}_{p}(\bO)=0$ for all $p \in \ha{I}$, 
which implies that $\ha{e}_{p}\bO=\bzero$ for all $p \in \ha{I}$. 
Let $\bz_{\BZ}^{\sigma}(\bO)$ denote the connected component 
of (the crystal graph of) the crystal 
$\bz_{\BZ}^{\sigma}$ containing $\bO$. 
The following theorem is the first main result of this paper; 
the proof will be given in the next section. 
%
%%%%%%%%%%%%%%%%
%%% thm:main %%%
%%%%%%%%%%%%%%%%
%
\begin{thm} \label{thm:main}
The crystal $\bz_{\BZ}^{\sigma}(\bO)$ 
is isomorphic, as a crystal for $U_{q}(\ha{\Fg}^{\vee})$, 
to the crystal basis $\ha{\CB}(\infty)$ of 
the negative part $U_{q}^{-}(\ha{\Fg}^{\vee})$ 
of $U_{q}(\ha{\Fg}^{\vee})$. 
\end{thm}

For each dominant integral weight $\ha{\lambda} \in \ha{\Fh}$ 
for $\ha{\Fg}^{\vee}$, let $\bz_{\BZ}^{\sigma}(\bO; \ha{\lambda})$ denote 
the subset of $\bz_{\BZ}^{\sigma}(\bO)$ consisting of all elements 
$\bM=(M_{\gamma})_{\gamma \in \Gamma_{\BZ}} \in \bz_{\BZ}^{\sigma}(\bO)$ 
satisfying the condition (cf. \eqref{eq:blambda}) that 
%
%%%%%%%%%%%%%%%%%%%
%%% eq:blam-aff %%%
%%%%%%%%%%%%%%%%%%%
%
\begin{equation} \label{eq:blam-aff}
M_{-s_{i}\Lambda_{i}} \ge -\pair{\ha{\lambda}}{\ha{\alpha}_{\ol{i}}}
\quad \text{for all $i \in \BZ$};
\end{equation}
recall that $\ol{i}$ denotes a unique element in 
$\ha{I}=\bigl\{0,\,1,\,\dots,\,\ell\bigr\}$ to which 
$i \in \BZ$ is congruent modulo $\ell+1$. 
Let us define a crystal structure for $U_{q}(\ha{\Fg}^{\vee})$ 
on the set $\bz_{\BZ}^{\sigma}(\bO; \ha{\lambda})$ 
(see Proposition~\ref{prop:crystal2} below). 
%
%%%%%%%%%%%%%%%%%%%%%
%%% lem:lam-ej-st %%%
%%%%%%%%%%%%%%%%%%%%%
%
\begin{lem} \label{lem:lam-ej-st}
The set $\bz_{\BZ}^{\sigma}(\bO; \ha{\lambda}) \cup \{\bzero\}$ 
is stable under the raising Kashiwara operators
$\ha{e}_{p}$ on $\bz_{\BZ}^{\sigma}$ for $p \in \BZ$. 
\end{lem}

\begin{proof}
Let $\bM=(M_{\gamma})_{\gamma \in \Gamma_{\BZ}} \in 
\bz_{\BZ}^{\sigma}(\bO; \ha{\lambda})$, and $p \in \BZ$. 
Suppose that $\bM':=\ha{e}_{p}\bM \ne \bzero$, and write it as:
$\bM'=\ha{e}_{p}\bM=(M_{\gamma}')_{\gamma \in \Gamma_{\BZ}}$. 
In order to prove that $\ha{e}_{p}\bM \in 
\bz_{\BZ}^{\sigma}(\bO; \ha{\lambda})$, 
it suffices to show that $M_{\gamma} \le M_{\gamma}'$ 
for all $\gamma \in \Gamma_{\BZ}$. 
Fix $\gamma \in \Gamma_{\BZ}$. 
We know from Proposition~\ref{prop:crystal} that 
$\ha{f}_{p}\bM'=\ha{f}_{p}\ha{e}_{p}\bM=\bM$. 
Also, it follows from the definition of $\ha{f}_{p}$ that
$M_{\gamma}=(\ha{f}_{p}\bM')_{\gamma}=
 (f_{L(\gamma,p)}\bM')_{\gamma}$. 
Therefore, we deduce from Remark~\ref{rem:AM2}\,(1) that 
$(f_{L(\gamma,p)}\bM')_{\gamma} \le M_{\gamma}'$, and 
hence $M_{\gamma} \le M_{\gamma}'$. 
This proves the lemma. 
\end{proof}

\begin{rem}
In contrast to the situation in Lemma~\ref{lem:lam-ej-st}, 
the set $\bz_{\BZ}^{\sigma}(\bO; \ha{\lambda})$ 
is not stable under the lowering Kashiwara operators 
$\ha{f}_{p}$ on $\bz_{\BZ}^{\sigma}$ for $p \in \BZ$. 
\end{rem}

For each $p \in \BZ$, 
we define a map $\ha{F}_{p}:
\bz_{\BZ}^{\sigma}(\bO; \ha{\lambda}) \rightarrow 
\bz_{\BZ}^{\sigma}(\bO; \ha{\lambda}) \cup \{\bzero\}$ by: 
\begin{equation}
\ha{F}_{p}\bM=
\begin{cases}
\ha{f}_{p}\bM & 
 \text{if $\ha{f}_{p}\bM$ is contained in 
 $\bz_{\BZ}^{\sigma}(\bO;\ha{\lambda})$}, \\[1.5mm]
\bzero & \text{otherwise},
\end{cases}
\end{equation}
for $\bM \in \bz_{\BZ}^{\sigma}(\bO; \ha{\lambda})$; 
by convention, we set $\ha{F}_{p}\bzero:=\bzero$ for all $p \in \BZ$. 
We define the weight $\Wt(\bM)$ of 
$\bM \in \bz_{\BZ}^{\sigma}(\bO; \ha{\lambda})$ by: 
%
%%%%%%%%%%%%%%%
%%% eq:haWt %%%
%%%%%%%%%%%%%%%
%
\begin{equation} \label{eq:haWt}
\Wt(\bM)
=\ha{\lambda}+\wt(\bM)
=\ha{\lambda}+
\sum_{i \in \ha{I}} M_{\Lambda_{i}} \, \ha{h}_{i},
\end{equation}
where $M_{\Lambda_{i}}:=\Theta(\bM)_{\Lambda_{i}}$ 
for $i \in \ha{I}$. Also, we set 
\begin{equation}
\ha{\Phi}_{p}(\bM):=
 \pair{\Wt(\bM)}{\ha{\alpha}_{\ol{p}}}+\ha{\ve}_{p}(\bM)
\quad \text{for 
$\bM \in \bz_{\BZ}^{\sigma}(\bO;\ha{\lambda})$ 
and $p \in \BZ$}.
\end{equation}
Then, it is easily seen 
from the definition \eqref{eq:def-have} of 
the map $\ha{\ve}_{p}$ and 
Remark~\ref{rem:sig-wlam2} that 
%
%%%%%%%%%%%%%%%%%%%
%%% eq:lam-havp %%%
%%%%%%%%%%%%%%%%%%%
%
\begin{equation} \label{eq:lam-havp}
\ha{\Phi}_{p}(\bM)=
 M_{\Lambda_{p}}-M_{s_{p}\Lambda_{p}}+
 \pair{\ha{\lambda}}{\ha{\alpha}_{\ol{p}}},
\end{equation}
where $M_{\Lambda_{p}}:=\Theta(\bM)_{\Lambda_{p}}$ and 
$M_{s_{p}\Lambda_{p}}:=\Theta(\bM)_{s_{p}\Lambda_{p}}$ 
(cf. \eqref{eq:vp-bM}). 
%
%%%%%%%%%%%%%%%%%%%%%
%%% prop:crystal2 %%%
%%%%%%%%%%%%%%%%%%%%%
%
\begin{prop} \label{prop:crystal2}
{\rm (1)} The set $\bz_{\BZ}^{\sigma}(\bO;\ha{\lambda})$, 
equipped with the maps 
$\Wt$, $\ha{e}_{p},\,\ha{F}_{p} \ (p \in \ha{I})$, and 
$\ha{\ve}_{p},\,\ha{\Phi}_{p} \ (p \in \ha{I})$ above, is 
a crystal for $U_{q}(\ha{\Fg}^{\vee})$. 

{\rm (2)} For $\bM \in \bz_{\BZ}^{\sigma}(\bO;\ha{\lambda})$ 
and $p \in \ha{I}$, there hold
\begin{equation*}
\ha{\ve}_{p}(\bM)=
  \max \bigl\{N \ge 0 \mid \ha{e}_{p}^{N}\bM \ne \bzero\bigr\}, \qquad
\ha{\Phi}_{p}(\bM)=
  \max \bigl\{N \ge 0 \mid \ha{F}_{p}^{N}\bM \ne \bzero\bigr\}.
\end{equation*}
\end{prop}

\begin{proof}
(1) This follows easily from Proposition~\ref{prop:crystal}. 
As examples, we show that 
%
%%%%%%%%%%%%%%%%%
%%% eq:cry2-1 %%%
%%%%%%%%%%%%%%%%%
%
\begin{equation} \label{eq:cry2-1}
\Wt(\ha{F}_{p}\bM)=\Wt(\bM)-\ha{h}_{p},
\end{equation}
%
%%%%%%%%%%%%%%%%%
%%% eq:cry2-2 %%%
%%%%%%%%%%%%%%%%%
%
\begin{equation} \label{eq:cry2-2}
\ha{\ve}_{p}(\ha{F}_{p}\bM)=\ha{\ve}_{p}(\bM)+1
\quad \text{and} \quad 
\ha{\Phi}_{p}(\ha{F}_{p}\bM)=\ha{\Phi}_{p}(\bM)-1, 
\end{equation}
for $\bM \in \bz_{\BZ}^{\sigma}(\bO;\ha{\lambda})$ 
and $p \in \ha{I}$ if $\ha{F}_{p}\bM \ne \bzero$. 
Note that in this case, $\ha{F}_{p}\bM=\ha{f}_{p}\bM$ 
by the definition of $\ha{F}_{p}$. 
First we show \eqref{eq:cry2-1}. 
It follows from the definition of $\Wt$ that 
\begin{equation*}
\Wt(\ha{F}_{p}\bM)=
\Wt(\ha{f}_{p}\bM)=
\ha{\lambda}+\wt(\ha{f}_{p}\bM). 
\end{equation*}
Since $\wt(\ha{f}_{p}\bM)=\wt(\bM)-\ha{h}_{p}$ 
by Proposition~\ref{prop:crystal}, we have
\begin{equation*}
\Wt(\ha{F}_{p}\bM)=
\ha{\lambda}+\wt(\ha{f}_{p}\bM)=
\ha{\lambda}+\wt(\bM)-\ha{h}_{p}=\Wt(\bM)-\ha{h}_{p},
\end{equation*}
as desired. Next we show \eqref{eq:cry2-2}. 
It follows from (the proof of) 
Proposition~\ref{prop:crystal} that 
$\ha{\ve}_{p}(\ha{F}_{p}\bM)=
\ha{\ve}_{p}(\ha{f}_{p}\bM)=\ha{\ve}_{p}(\bM)+1$.
Also, we compute: 
\begin{align*}
\ha{\Phi}_{p}(\ha{F}_{p}\bM) & =
\ha{\Phi}_{p}(\ha{f}_{p}\bM) =
\pair{\Wt(\ha{f}_{p}\bM)}{\ha{\alpha}_{p}}+
\ha{\ve}_{p}(\ha{f}_{p}\bM) 
\quad \text{by the definition of $\ha{\Phi}_{p}$} \\
& = \pair{\Wt(\bM)-\ha{h}_{p}}{\ha{\alpha}_{p}}+
    \ha{\ve}_{p}(\bM)+1 
\quad \text{by \eqref{eq:cry2-1} and 
Proposition~\ref{prop:crystal}} \\
& = \pair{\Wt(\bM)}{\ha{\alpha}_{p}}+ \ha{\ve}_{p}(\bM)-1 
  = \ha{\Phi}_{p}(\bM)-1
\quad \text{by the definition of $\ha{\Phi}_{p}$},
\end{align*}
as desired. 

(2) The first equation 
follows immediately from Remark~\ref{rem:have-max} 
together with Lemma~\ref{lem:lam-ej-st}.
We will prove the second equation. 
Fix $p \in \ha{I}$. We first show that 
%
%%%%%%%%%%%%%%%%
%%% eq:Phi-0 %%%
%%%%%%%%%%%%%%%%
%
\begin{equation} \label{eq:Phi-0}
\ha{\Phi}_{p}(\bM) \ge 0
\quad \text{for all $\bM \in 
\bz_{\BZ}^{\sigma}(\bO;\ha{\lambda})$}.
\end{equation}
Fix $\bM \in \bz_{\BZ}^{\sigma}(\bO;\ha{\lambda})$, 
and take an interval $I$ in $\BZ$ such that 
$I \in \Int(\bM;e,p) \cap \Int(\bM;s_{p},p)$. 
Then we see from \eqref{eq:lam-havp} that 
%
%%%%%%%%%%%%%%%%
%%% eq:Phi-1 %%%
%%%%%%%%%%%%%%%%
%
\begin{equation} \label{eq:Phi-1}
\ha{\Phi}_{p}(\bM)=
 M_{\Lambda_{p}}-M_{s_{p}\Lambda_{p}}+
 \pair{\ha{\lambda}}{\ha{\alpha}_{p}}=
 M_{\vpi_{p}^{I}}-M_{s_{p}\vpi_{p}^{I}}+
 \pair{\ha{\lambda}}{\ha{\alpha}_{p}}.
\end{equation}
Now we define a dominant integral weight 
$\lambda \in \Fh_{I}$ for $\Fg_{I}^{\vee}$ by: 
$\pair{\lambda}{\alpha_{i}}=
 \pair{\ha{\lambda}}{\ha{\alpha}_{\ol{i}}}$ 
for $i \in I$. Then, we deduce 
from \eqref{eq:blambda}, \eqref{eq:blam-aff}, and 
\eqref{eq:bij-index3} that $\bM_{I} \in \bz_{I}$ 
is contained in $\bz_{I}(\lambda) \subset \bz_{I}$. 
Because $\bz_{I}(\lambda)$ is isomorphic, 
as a crystal for $U_{q}(\Fg_{I}^{\vee})$, to 
the crystal basis $\CB_{I}(\lambda)$ 
(see Theorem~\ref{thm:blambda}), it follows that 
$\Phi_{p}(\bM_{I}) \ge 0$. 
Also, we see from \eqref{eq:vp-bM} that 
%
%%%%%%%%%%%%%%%%
%%% eq:Phi-2 %%%
%%%%%%%%%%%%%%%%
%
\begin{equation} \label{eq:Phi-2}
\Phi_{p}(\bM_{I})=
 M_{\vpi_{p}^{I}}-M_{s_{p}\vpi_{p}^{I}}+
 \pair{\lambda}{\alpha_{p}}.
\end{equation}
Since $\pair{\lambda}{\alpha_{p}}=
 \pair{\ha{\lambda}}{\ha{\alpha}_{p}}$ 
by the definition of $\lambda \in \Fh_{I}$, 
we conclude from \eqref{eq:Phi-1} and \eqref{eq:Phi-2}
that $\ha{\Phi}_{p}(\bM)=\Phi_{p}(\bM_{I}) \ge 0$, 
as desired.

Next we show that 
for $\bM \in \bz_{\BZ}^{\sigma}(\bO;\ha{\lambda})$,
%
%%%%%%%%%%%%%%%%
%%% eq:Phi-3 %%%
%%%%%%%%%%%%%%%%
%
\begin{equation} \label{eq:Phi-3}
\ha{F}_{p}\bM=\bzero
\quad \text{if and only if} \quad
\ha{\Phi}_{p}(\bM)=0.
\end{equation}
Fix $\bM \in \bz_{\BZ}^{\sigma}(\bO;\ha{\lambda})$. 
Suppose that $\ha{\Phi}_{p}(\bM)=0$, and $\ha{F}_{p}\bM \ne \bzero$. 
Then, since $\ha{\Phi}_{p}(\ha{F}_{p}\bM)=
\ha{\Phi}_{p}(\bM)-1$ by \eqref{eq:cry2-2}, we have 
$\ha{\Phi}_{p}(\ha{F}_{p}\bM)=-1$, 
which contradicts \eqref{eq:Phi-0}. 
Hence, if $\ha{\Phi}_{p}(\bM)=0$, then 
$\ha{F}_{p}\bM=\bzero$. To show the converse, assume 
that $\ha{F}_{p}\bM=\bzero$, or equivalently, 
$\ha{f}_{p}\bM \notin \bz_{\BZ}^{\sigma}(\bO;\ha{\lambda})$. 
Let us write $\bM \in \bz_{\BZ}^{\sigma}(\bO;\ha{\lambda})$ and 
$\ha{f}_{p}\bM \in \bz_{\BZ}^{\sigma}(\bO)$ as: 
$\bM=(M_{\gamma})_{\gamma \in \Gamma_{\BZ}}$ and 
$\ha{f}_{p}\bM=(M_{\gamma}')_{\gamma \in \Gamma_{\BZ}}$, 
respectively.
From the assumption that 
$\ha{f}_{p}\bM \notin \bz_{\BZ}^{\sigma}(\bO;\ha{\lambda})$, 
it follows that 
$M_{-s_{q}\Lambda_{q}}'
 < -\pair{\ha{\lambda}}{\ha{\alpha}_{\ol{q}}}$ for some $q \in \BZ$. 
Note that since $M_{\gamma}'=M_{\sigma^{-1}(\gamma)}'$ 
for all $\gamma \in \Gamma_{\BZ}$, 
we may assume $q \in \ha{I}$. Then, we infer that 
this $q$ is equal to $p$. Indeed, 
for each $i \in \ha{I} \setminus \{p\}$, 
we have $L(-s_{i}\Lambda_{i},p)=\emptyset$, since 
$\pair{h_{i}}{s_{i}\Lambda_{i}}=-1$ and 
$\pair{h_{j}}{s_{i}\Lambda_{i}} \ge 0$ 
for all $j \in \BZ$ with $j \ne i$. 
Therefore, by the definition of $\ha{f}_{p}$, 
\begin{equation*}
M_{-s_{i}\Lambda_{i}}'
 =(\ha{f}_{p}\bM)_{-s_{i}\Lambda_{i}}
 =(f_{\emptyset}\bM)_{-s_{i}\Lambda_{i}}
 =M_{-s_{i}\Lambda_{i}}.
\end{equation*}
Hence it follows that $M_{-s_{i}\Lambda_{i}}'=
M_{-s_{i}\Lambda_{i}} \ge -\pair{\ha{\lambda}}{\ha{\alpha}_{\ol{i}}}$ 
since $\bM \in \bz_{\BZ}^{\sigma}(\bO;\ha{\lambda})$. 
Consequently, $q \in \ha{I}$ is not equal to 
any $i \in \ha{I} \setminus \{p\}$, that is, $q=p$. 

Now, as in the proof of \eqref{eq:Phi-0} above, 
take an interval $I$ in $\BZ$ such that 
$I \in \Int(\bM;e,p) \cap \Int(\bM;s_{p},p)$, and 
then define a dominant integral weight 
$\lambda \in \Fh_{I}$ for $\Fg_{I}^{\vee}$ by: 
$\pair{\lambda}{\alpha_{i}}=
 \pair{\ha{\lambda}}{\ha{\alpha}_{\ol{i}}}$ 
for $i \in I$; we know from the argument above that 
$\bM_{I} \in \bz_{I}(\lambda)$, and 
$\ha{\Phi}_{p}(\bM)=\Phi_{p}(\bM_{I})$. 
Therefore, in order to show that 
$\ha{\Phi}_{p}(\bM)=0$, it suffices to show that 
$\Phi_{p}(\bM_{I})=0$, which is equivalent to 
$F_{p}\bM_{I}=\bzero$ by Theorem~\ref{thm:blambda}. 
Recall from the above that $M_{-s_{p}\Lambda_{p}}' < 
-\pair{\ha{\lambda}}{\ha{\alpha}_{p}}=
-\pair{\lambda}{\alpha_{p}}$.
Also, it follows from the definition of 
$\ha{f}_{p}$ on $\bz_{\BZ}^{\sigma}$ and 
the definition of $f_{p}$ on $\bz_{\BZ}$ that 
\begin{align*}
M_{-s_{p}\Lambda_{p}}' & =
 (\ha{f}_{p}\bM)_{-s_{p}\Lambda_{p}} = 
 (f_{p}\bM)_{-s_{p}\Lambda_{p}}
 \quad \text{since $L(-s_{p}\Lambda_{p},p)=\bigl\{p\bigr\}$} \\
& =(f_{p}\bM_{I})_{-s_{p}\Lambda_{p}}. 
\end{align*}
Combining these, we obtain 
$(f_{p}\bM_{I})_{-s_{p}\Lambda_{p}} < 
-\pair{\lambda}{\alpha_{p}}$, which implies that 
$f_{p}\bM_{I} \notin \bz_{I}(\lambda)$, 
and hence $F_{p}\bM_{I}=\bzero$ by the definition. 
Thus we have shown \eqref{eq:Phi-3}. 

From \eqref{eq:Phi-0}, \eqref{eq:Phi-3}, and the second equation of 
\eqref{eq:cry2-2}, we deduce that 
$\ha{\Phi}_{p}(\bM)=
  \max \bigl\{N \ge 0 \mid \ha{F}_{p}^{N}\bM \ne \bzero\bigr\}$
for $\bM \in \bz_{\BZ}^{\sigma}(\bO;\ha{\lambda})$ 
and $p \in \ha{I}$, as desired. 
This completes the proof of the proposition. 
\end{proof}

The following theorem is the second main result of this paper; 
the proof will be given in the next section. 
%
%%%%%%%%%%%%%%%%%
%%% thm:main2 %%%
%%%%%%%%%%%%%%%%%
%
\begin{thm} \label{thm:main2}
Let $\ha{\lambda} \in \Fh$ be 
a dominant integral weight for $\ha{\Fg}^{\vee}$. 
The crystal $\bz_{\BZ}^{\sigma}(\bO; \ha{\lambda})$ is 
isomorphic, as a crystal for $U_{q}(\ha{\Fg}^{\vee})$, 
to the crystal basis $\ha{\CB}(\ha{\lambda})$ of 
the irreducible highest weight $U_{q}(\ha{\Fg}^{\vee})$-module 
of highest weight $\ha{\lambda}$. 
\end{thm}

%==============================%
%     START SUBSECTION 0405    %
%==============================%
%
\subsection{Proofs of Theorems~\ref{thm:main} and \ref{thm:main2}.}
\label{subsec:proofs}

We first prove Theorem~\ref{thm:main2}; 
Theorem~\ref{thm:main} is obtained as a corollary 
of Theorem~\ref{thm:main2}. 
\begin{proof}[Proof of Theorem~\ref{thm:main2}]
By Proposition~\ref{prop:crystal2} and 
Theorem~\ref{thm:stem} in the Appendix, 
it suffices to prove that 
the crystal $\bz_{\BZ}^{\sigma}(\bO; \ha{\lambda})$
satisfies conditions (C1)--(C6) of Theorem~\ref{thm:stem}. 
First we prove that the crystal 
$\bz_{\BZ}^{\sigma}(\bO; \ha{\lambda})$ satisfies condition (C6). 
Note that $\bO \in \bz_{\BZ}^{\sigma}(\bO; \ha{\lambda})$. 
It follows from the definition of the raising Kashiwara operators 
$\ha{e}_{p}$, $p \in \ha{I}$, on $\bz_{\BZ}^{\sigma}(\bO; \ha{\lambda})$ 
(see also the beginning of \S\ref{subsec:main}) that 
$\ha{e}_{p}\bO=\bzero$ for all $p \in \ha{I}$. 
Also, $\Theta(\bO)_{\Lambda_{p}}$ and $\Theta(\bO)_{s_{p}\Lambda_{p}}$ 
are equal to $0$ by the definitions. Therefore, 
it follows from \eqref{eq:haWt} and \eqref{eq:lam-havp} that 
$\Wt(\bO)=\ha{\lambda}$ and 
$\ha{\Phi}_{p}(\bO)=
\pair{\ha{\lambda}}{\ha{\alpha}_{p}}$ for all $p \in \ha{I}$, as desired.

We also need to prove that 
the crystal $\bz_{\BZ}^{\sigma}(\bO; \ha{\lambda})$
satisfies conditions (C1)--(C5) of Theorem~\ref{thm:stem}. 
We will prove that 
$\bz_{\BZ}^{\sigma}(\bO; \ha{\lambda})$
satisfies condition~(C5); the proofs for 
the other conditions are similar. 
Namely, we will prove the following assertion:  
Let $\bM \in \bz_{\BZ}^{\sigma}(\bO; \ha{\lambda})$, 
and $p,\,q \in \ha{I}$. Assume that 
$\ha{F}_{p}\bM \ne \bzero$ and $\ha{F}_{q}\bM \ne \bzero$, and 
that $\ha{\Phi}_{q}(\ha{F}_{p}\bM)=\ha{\Phi}_{q}(\bM)+1$ and 
$\ha{\Phi}_{p}(\ha{F}_{q}\bM)=\ha{\Phi}_{p}(\bM)+1$. Then, 
\begin{equation} \label{eq:step1}
\ha{F}_{p}\ha{F}_{q}^{2}\ha{F}_{p}\bM \ne \bzero
\quad \text{\rm and} \quad 
\ha{F}_{q}\ha{F}_{p}^{2}\ha{F}_{q}\bM \ne \bzero,
\end{equation}
\begin{equation} \label{eq:step2}
\ha{F}_{p}\ha{F}_{q}^{2}\ha{F}_{p}\bM=
\ha{F}_{q}\ha{F}_{p}^{2}\ha{F}_{q}\bM,
\end{equation}
\begin{equation} \label{eq:step3}
\ha{\ve}_{q}(\ha{F}_{p}\bM)=
\ha{\ve}_{q}(\ha{F}_{p}^{2}\ha{F}_{q}\bM)
\quad \text{\rm and} \quad
\ha{\ve}_{p}(\ha{F}_{q}\bM)=
\ha{\ve}_{p}(\ha{F}_{q}^{2}\ha{F}_{p}\bM).
\end{equation}
Here we note that $p \ne q$. 
Indeed, if $p=q$, then it follows from 
the second equation of \eqref{eq:cry2-2} that 
$\ha{\Phi}_{p}(\ha{F}_{p}\bM)=\ha{\Phi}_{q}(\bM)-1$, 
which contradicts the assumption that 
$\ha{\Phi}_{p}(\ha{F}_{p}\bM)=\ha{\Phi}_{p}(\bM)+1$. 

A key to the proof of \eqref{eq:step1}--\eqref{eq:step3} 
is Claim~\ref{c:main2-2} below. 
For an interval $I$ in $\BZ$, 
we define a dominant integral weight 
$\lambda_{I} \in \Fh_{I}$ for $\Fg_{I}^{\vee}$ by: 
%
%%%%%%%%%%%%%%%
%%% eq:lamI %%%
%%%%%%%%%%%%%%%
%
\begin{equation} \label{eq:lamI}
\pair{\lambda_{I}}{\alpha_{i}}=
\pair{\ha{\lambda}}{\ha{\alpha}_{\ol{i}}}
\quad \text{for $i \in I$}.
\end{equation}
As mentioned in the proof of 
Proposition~\ref{prop:crystal2}\,(2), 
$\bM_{I} \in \bz_{I}$ is contained in 
$\bz_{I}(\lambda_{I}) \subset \bz_{I}$; 
recall from Theorem~\ref{thm:blambda} that 
$\bz_{I}(\lambda_{I})$ is isomorphic, 
as a crystal for $U_{q}(\Fg_{I}^{\vee})$, to 
the crystal basis $\CB_{I}(\lambda_{I})$. 
%
%%%%%%%%%%%%%%%%%
%%% c:main2-2 %%%
%%%%%%%%%%%%%%%%%
%
\begin{claim} \label{c:main2-2}
Let $r,\,t \in \BZ$ be such that 
$\ol{r}=p$, $\ol{t}=q$, and $0 < |r-t| < \ell$. 
Assume that an interval $I$ in $\BZ$ satisfies 
the following conditions\,{\rm:}

{\rm (a1)} 
$I \in \Int(\bM;e,r) \cap \Int(\bM;s_{r},r)${\rm;}

{\rm (a2)}
$I \in \Int(\bM;e,t) \cap \Int(\bM;s_{t},t)${\rm;} 

{\rm (a3)} 
$I \in \Int(\ha{F}_{p}\bM;e,t) \cap 
 \Int(\ha{F}_{p}\bM;s_{t},t)${\rm;}

{\rm (a4)} 
$I \in \Int(\ha{F}_{q}\bM;e,r) \cap 
 \Int(\ha{F}_{q}\bM;s_{r},r)${\rm.}

{\rm (i)} We have 
$\Phi_{r}(\bM_{I})=\ha{\Phi}_{p}(\bM) > 0$ and 
$\Phi_{t}(\bM_{I})=\ha{\Phi}_{q}(\bM) > 0$, and hence 
$F_{r}\bM_{I} \ne \bzero$ and $F_{t}\bM_{I} \ne \bzero$. 
Also, we have 
$\Phi_{t}(F_{r}\bM_{I})=\Phi_{t}(\bM_{I})+1$ and 
$\Phi_{r}(F_{t}\bM_{I})=\Phi_{r}(\bM_{I})+1$. 

{\rm (ii)} We have
\begin{equation*}
F_{r}F_{t}^{2}F_{r}\bM_{I} \ne \bzero
\quad \text{\rm and} \quad 
F_{t}F_{r}^{2}F_{t}\bM_{I} \ne \bzero,
\end{equation*}
\begin{equation*}
F_{r}F_{t}^{2}F_{r}\bM_{I}=
F_{t}F_{r}^{2}F_{t}\bM_{I},
\end{equation*}
\begin{equation*}
\ve_{t}(F_{r}\bM_{I})=
\ve_{t}(F_{r}^{2}F_{t}\bM_{I})
\quad \text{\rm and} \quad
\ve_{r}(F_{t}\bM_{I})=
\ve_{r}(F_{t}^{2}F_{r}\bM_{I}).
\end{equation*}
\end{claim}

\noindent
{\it Proof of Claim~\ref{c:main2-2}.} 
(i) We write $\bM \in \bz_{\BZ}^{\sigma}(\bO;\ha{\lambda})$ and 
$\Theta(\bM)$ as: 
$\bM=(M_{\gamma})_{\gamma \in \Gamma_{\BZ}}$ and 
$\Theta(\bM)=(M_{\xi})_{\xi \in \Xi_{\BZ}}$, respectively. 
Then, we compute:
\begin{align*}
\Phi_{r}(\bM_{I}) & =
 M_{\vpi_{r}^{I}}-M_{s_{r}\vpi_{r}^{I}}+
 \pair{\lambda_{I}}{\alpha_{r}} 
 \quad \text{by \eqref{eq:vp-bM}} \\
& = 
 M_{\Lambda_{r}}-M_{s_{r}\Lambda_{r}}+
 \pair{\lambda_{I}}{\alpha_{r}} 
 \quad \text{by condition (a1)}. 
\end{align*}
Since $r$ is congruent to $p$ modulo $\ell+1$ by assumption, 
we have $r=\sigma^{n}(p)$ for some $n \in \BZ$. 
Hence, by Remark~\ref{rem:sig-wlam2},
\begin{align*}
& M_{\Lambda_{r}}=M_{\Lambda_{\sigma^{n}(p)}}=
  M_{\sigma^{n}(\Lambda_{p})}=M_{\Lambda_{p}}, \\
& M_{s_{r}\Lambda_{r}}=
  M_{s_{\sigma^{n}(p)}\Lambda_{\sigma^{n}(p)}}=
  M_{\sigma^{n}(s_{p}\Lambda_{p})}=
  M_{s_{p}\Lambda_{p}}.
\end{align*}
Also, by the definition of $\lambda_{I}$, we have 
$\pair{\lambda_{I}}{\alpha_{r}}=\pair{\ha{\lambda}}{\ha{\alpha}_{p}}$. 
Substituting these into the above, we obtain 
\begin{equation*}
\Phi_{r}(\bM_{I})=
 M_{\Lambda_{p}}-M_{s_{p}\Lambda_{p}}+
 \pair{\ha{\lambda}}{\ha{\alpha}_{p}} 
 = \ha{\Phi}_{p}(\bM)
 \quad \text{by \eqref{eq:lam-havp}}.
\end{equation*}
Since $\ha{\Phi}_{p}(\bM) > 0$ 
by the assumption that 
$\ha{F}_{p}\bM \ne \bzero$, 
we get $\Phi_{r}(\bM_{I})=\ha{\Phi}_{p}(\bM_{I}) > 0$, 
as desired. Similarly, we can show that $\Phi_{t}(\bM_{I})=
\ha{\Phi}_{q}(\bM) > 0$. 

Now, we write 
$\ha{F}_{p}\bM 
 \in \bz_{\BZ}^{\sigma}(\bO;\ha{\lambda})$ and 
$\Theta(\ha{F}_{p}\bM)$ as: 
$\ha{F}_{p}\bM=(M_{\gamma}')_{\gamma \in \Gamma_{\BZ}}$ and 
$\Theta(\ha{F}_{p}\bM)=(M_{\xi}')_{\xi \in \Xi_{\BZ}}$, 
respectively. 
Since $L(\vpi_{t}^{I},p)=\emptyset \subset \{r\}$ 
(recall that $0 < |r-t| < \ell$), we have 
\begin{align*}
M_{\Lambda_{t}}' & = M_{\vpi_{t}^{I}}' 
 \quad \text{by condition (a3)} \\
& = (\ha{F}_{p}\bM)_{\vpi_{t}^{I}}
  = (F_{r}\bM)_{\vpi_{t}^{I}} 
 \quad \text{by Remark~\ref{rem:haf}} \\
& = (F_{r}\bM_{I})_{\vpi_{t}^{I}} 
 \quad \text{by conditions (a1), (a2), and the definition of $F_{r}M$}. 
\end{align*}
Also, it follows from \eqref{eq:pair-vpi} that
$\bigl\{i \in \BZ \mid \pair{h_{i}}{s_{t}\vpi_{t}^{I}} > 0 \bigr\}
 \subset \bigl\{t-1,\,t+1\bigr\}$. 
Since $0 < |r-t| < \ell$, it is easily seen that $r+(\ell+1)n > t+1$ 
and $r-(\ell+1)n < t-1$ for every $n \in \BZ_{> 0}$. Hence we deduce that 
$L(s_{t}\vpi_{t}^{I},p) \subset \{r\}$. 
Using this fact, we can show in exactly the same way as above that 
$M_{s_{t}\Lambda_{t}}'=(F_{r}\bM_{I})_{s_{t}\vpi_{t}^{I}}$. 
Therefore, 
\begin{align*}
\Phi_{t}(F_{r}\bM_{I}) & =
 (F_{r}\bM_{I})_{\vpi_{t}^{I}}-
 (F_{r}\bM_{I})_{s_{t}\vpi_{t}^{I}}+
 \pair{\lambda_{I}}{\alpha_{t}} 
 \quad \text{by \eqref{eq:vp-bM}} \\
& = 
 M_{\Lambda_{t}}'-M_{s_{t}\Lambda_{t}}'+
 \pair{\lambda_{I}}{\alpha_{t}} \\
& = 
 M_{\Lambda_{q}}'-M_{s_{q}\Lambda_{q}}'+
 \pair{\ha{\lambda}}{\ha{\alpha}_{q}} 
 \quad \text{by Remark~\ref{rem:sig-wlam2} and 
 the definition of $\lambda_{I}$} \\
& =\ha{\Phi}_{q}(\ha{F}_{p}\bM) 
 \quad \text{by \eqref{eq:lam-havp}}.
\end{align*}
Because $\ha{\Phi}_{q}(\ha{F}_{p}\bM)=
\ha{\Phi}_{q}(\bM)+1$ by our assumption, and 
$\ha{\Phi}_{q}(\bM)=\Phi_{t}(\bM_{I})$ as shown above, 
we obtain 
$\Phi_{t}(F_{r}\bM_{I}) = 
\ha{\Phi}_{q}(\ha{F}_{p}\bM) = 
\ha{\Phi}_{q}(\bM)+1=
\Phi_{t}(\bM_{I})+1$, 
as desired. The equation 
$\Phi_{r}(F_{t}\bM_{I})=\Phi_{r}(\bM_{I})+1$ 
can be shown similarly. 

(ii) Because $\bz_{I}(\lambda_{I})$ is isomorphic, 
as a crystal for $U_{q}(\Fg_{I}^{\vee})$, to 
the crystal basis $\CB_{I}(\lambda_{I})$ by Theorem~\ref{thm:blambda},
this crystal satisfies condition (C5) of Theorem~\ref{thm:stem}. 
Hence the equations in part~(ii) 
follow immediately from part~(i). 
This proves Claim~\ref{c:main2-2}. \bqed

%%%%
\vsp
%%%%

First we show \eqref{eq:step1}; 
we only prove that 
$\ha{F}_{p}\ha{F}_{q}^{2}\ha{F}_{p}\bM \ne \bzero$, 
since we can prove that 
$\ha{F}_{q}\ha{F}_{p}^{2}\ha{F}_{q}\bM \ne \bzero$ similarly. 
Recall that $\ha{F}_{p}\bM \ne \bzero$ by our assumption. 
Also, since $\ha{F}_{q}\bM \ne \bzero$ by our assumption, 
it follows from Proposition~\ref{prop:crystal2}\,(2) 
that $\ha{\Phi}_{q}(\bM) > 0$. 
Therefore, we have $\ha{\Phi}_{q}(\ha{F}_{p}\bM)=
\ha{\Phi}_{q}(\bM)+1 \ge 2$ by our assumption, 
which implies that $\ha{F}_{q}^{2}\ha{F}_{p}\bM \ne \bzero$ 
by Proposition~\ref{prop:crystal2}\,(2). 
We set $\bM'':=\ha{F}_{q}^{2}\ha{F}_{p}\bM \in 
\bz_{\BZ}^{\sigma}(\bO;\ha{\lambda})$, and 
write $\bM''$ and $\Theta(\bM'')$ as: 
$\bM''=(M_{\gamma}'')_{\gamma \in \Gamma_{\BZ}}$ and 
$\Theta(\bM'')=(M_{\xi}'')_{\xi \in \Xi_{\BZ}}$, respectively.
In order to show that 
$\ha{F}_{p}\ha{F}_{q}^{2}\ha{F}_{p}\bM=
 \ha{F}_{p}\bM'' \ne \bzero$, 
it suffices to show that 
\begin{equation*}
\ha{\Phi}_{p}(\bM'')=
M_{\Lambda_{p}}''-M_{s_{p}\Lambda_{p}}''+
\pair{\ha{\lambda}}{\ha{\alpha}_{p}} > 0
\end{equation*}
by Proposition~\ref{prop:crystal2}\,(2) 
along with equation \eqref{eq:lam-havp}.
We define $r,\,t \in \BZ$ by:
%
%%%%%%%%%%%%%%%%%
%%% eq:def-rt %%%
%%%%%%%%%%%%%%%%%
%
\begin{equation} \label{eq:def-rt}
(r,t)=
\begin{cases}
(p,q) & \text{if $|p-q| < \ell$}, \\[1.5mm]
(\ell,\,\ell+1) & \text{if $p=\ell$ and $q=0$}, \\[1.5mm]
(\ell+1,\,\ell) & \text{if $p=0$ and $q=\ell$}.
\end{cases}
\end{equation}
Let $K$ be an interval in $\BZ$ 
such that $r,\,t \in K$, and % $\min K < r,\,t < \max K$, and 
take an interval $I$ in $\BZ$ satisfying 
conditions (a1)--(a4) in Claim~\ref{c:main2-2} and 
the following: 

(b1) $I \in \Int(\bM'';e,r) \cap \Int(\bM'';s_{r},r)$; 

(b2) $I \in \Int(\bM;v,k)$ for all $v \in W_{K}$ and $k \in K$. 

\noindent
It follows from Remark~\ref{rem:sig-wlam2} and 
condition (b1) that $M_{\Lambda_{p}}''=M_{\Lambda_{r}}''=
M_{\vpi_{r}^{I}}''$. Also, 
\begin{align*}
M_{\vpi_{r}^{I}}'' & = 
 (\ha{F}_{q}^{2}\ha{F}_{p}\bM)_{\vpi_{r}^{I}}=
 (\ha{f}_{q}^{2}\ha{f}_{p}\bM)_{\vpi_{r}^{I}}
 \quad \text{by the definitions of $\ha{F}_{q}$ and $\ha{F}_{p}$} \\
& = 
 (\ha{f}_{t}^{2}\ha{f}_{r}\bM)_{\vpi_{r}^{I}} 
 \quad \text{by \eqref{eq:fp-fq}}.
\end{align*}
Here we note that $L(\vpi_{r}^{I},r)=\{r\}$ and 
$L(\vpi_{r}^{I},t)=\emptyset$ since $0 < |r-t| < \ell$. 
Therefore, we deduce from Lemma~\ref{lem:tech2}
(with $p=r$, $q=t$, 
$\ha{X}=\ha{f}_{t}^{2}\ha{f}_{r}$, 
$\gamma=\vpi_{r}^{I}$, and $L_{r}=\{r\}$) that 
$f_{t}^{2}f_{r}\bM \ne \bzero$ and 
$(\ha{f}_{t}^{2}\ha{f}_{r}\bM)_{\vpi_{r}^{I}}=
 (f_{t}^{2}f_{r}\bM)_{\vpi_{r}^{I}}$. 
Since $\bM \in \bz_{\BZ}(I,K)$ by condition (b2), 
we see from Lemma~\ref{lem:IK-fj}\,(2) that 
$(f_{t}^{2}f_{r}\bM)_{I}=f_{t}^{2}f_{r}\bM_{I}$, 
and hence that 
$(f_{t}^{2}f_{r}\bM)_{\vpi_{r}^{I}}=
 (f_{t}^{2}f_{r}\bM_{I})_{\vpi_{r}^{I}}$.  
Also, because $r,\,t \in \BZ$ satisfies the conditions that 
$\ol{r}=p$, $\ol{t}=q$, and $0 < |r-t| < \ell$, and 
because the interval $I$ satisfies conditions (a1)--(a4) 
of Claim~\ref{c:main2-2}, it follows from 
Claim~\ref{c:main2-2}\,(ii) that 
$F_{t}^{2}F_{r}\bM_{I} \ne \bzero$, and hence 
$f_{t}^{2}f_{r}\bM_{I}=F_{t}^{2}F_{r}\bM_{I}$. 
Putting the above together, we obtain 
$M_{\Lambda_{p}}''=(F_{t}^{2}F_{r}\bM_{I})_{\vpi_{r}^{I}}$. 
Similarly, we can show that 
$M_{s_{p}\Lambda_{p}}''=
(F_{t}^{2}F_{r}\bM_{I})_{s_{r}\vpi_{r}^{I}}$. 
Consequently, we see that 
\begin{align*}
\ha{\Phi}_{p}(\bM'') & =
M_{\Lambda_{p}}''-M_{s_{p}\Lambda_{p}}''+
\pair{\ha{\lambda}}{\ha{\alpha}_{p}} \\
& = 
(F_{t}^{2}F_{r}\bM_{I})_{\vpi_{r}^{I}}-
(F_{t}^{2}F_{r}\bM_{I})_{s_{r}\vpi_{r}^{I}}+
\pair{\lambda_{I}}{\alpha_{r}} \\
& = 
\Phi_{r}(F_{t}^{2}F_{r}\bM_{I}) 
\quad \text{by \eqref{eq:vp-bM} together with Theorem~\ref{thm:blambda}} \\
& > 0 \quad \text{by Claim~\ref{c:main2-2}\,(ii)}.
\end{align*}
Thus we have shown \eqref{eq:step1}. 

%%%%%

Next we show equation \eqref{eq:step2}. 
Define $r,\,t \in \BZ$ as in \eqref{eq:def-rt}. 
Since $\ha{F}_{p}\ha{F}_{q}^{2}\ha{F}_{p}\bM \ne \bzero$ and 
$\ha{F}_{q}\ha{F}_{p}^{2}\ha{F}_{q}\bM \ne \bzero$
by \eqref{eq:step1}, it follows from 
the definitions of $\ha{F}_{p}$ and $\ha{F}_{q}$ along with 
\eqref{eq:fp-fq} that 
\begin{align*}
& \ha{F}_{p}\ha{F}_{q}^{2}\ha{F}_{p}\bM=
 \ha{f}_{p}\ha{f}_{q}^{2}\ha{f}_{p}\bM=
 \ha{f}_{r}\ha{f}_{t}^{2}\ha{f}_{r}\bM, \\
& \ha{F}_{q}\ha{F}_{p}^{2}\ha{F}_{q}\bM=
 \ha{f}_{q}\ha{f}_{p}^{2}\ha{f}_{q}\bM=
 \ha{f}_{t}\ha{f}_{r}^{2}\ha{f}_{t}\bM.
\end{align*}
Therefore, it suffices to show that 
\begin{equation*}
(\ha{f}_{r}\ha{f}_{t}^{2}\ha{f}_{r}\bM)_{\gamma}=
(\ha{f}_{t}\ha{f}_{r}^{2}\ha{f}_{t}\bM)_{\gamma}
\quad \text{for all $\gamma \in \Gamma_{\BZ}$}.
\end{equation*}
Fix $\gamma \in \Gamma_{\BZ}$, and take 
a finite subset $L_{r}$ of $r+(\ell+1)\BZ$ such that 
$L_{r} \supset L(\gamma,r)$ and such that $L_{t}:=
\bigl\{u+(t-r) \mid u \in L_{r} \bigr\} \supset L(\gamma,t)$. 
We infer from Lemma~\ref{lem:tech2} that 
\begin{equation*}
(\ha{f}_{r}\ha{f}_{t}^{2}\ha{f}_{r}\bM)_{\gamma}=
 (f_{L_{r}}f_{L_{t}}^{2}f_{L_{r}}\bM)_{\gamma}
\quad \text{and} \quad
(\ha{f}_{t}\ha{f}_{r}^{2}\ha{f}_{t}\bM)_{\gamma}=
 (f_{L_{t}}f_{L_{r}}^{2}f_{L_{t}}\bM)_{\gamma}.
\end{equation*}
Let us write $L_{r}$ and $L_{t}$ as: 
$L_{r}=\bigl\{r_{1},\,r_{2},\,\dots,\,r_{a}\bigr\}$ and 
$L_{t}=\bigl\{t_{1},\,t_{2},\,\dots,\,t_{a}\bigr\}$, respectively, 
where $t_{b}=r_{b}+(t-r)$ for each $1 \le b \le a$; 
note that $0 < |r_{b}-t_{b}| < \ell$ 
for all $1 \le b \le a$. 
Let $K$ be an interval in $\BZ$ 
containing $L_{r} \cup L_{t}$, 
% $\min K < u < \max K$ for all $u \in L_{r} \cup L_{t}$, 
and take an interval $I$ in $\BZ$ satisfying 
the following:

(a1)'
$I \in \Int(\bM;e,r_{b}) \cap \Int(\bM;s_{r_{b}},r_{b})$
for all $1 \le b \le a$; 

(a2)'
$I \in \Int(\bM;e,t_{b}) \cap \Int(\bM;s_{t_{b}},t_{b})$ 
for all $1 \le b \le a$; 

(a3)'
$I \in \Int(\ha{F}_{p}\bM;e,t_{b}) \cap 
 \Int(\ha{F}_{p}\bM;s_{t_{b}},t_{b})$
for all $1 \le b \le a$; 

(a4)'
$I \in \Int(\ha{F}_{q}\bM;e,r_{b}) \cap 
 \Int(\ha{F}_{q}\bM;s_{r_{b}},r_{b})$
for all $1 \le b \le a$; 

(c1) $\gamma \in \Gamma_{I}$; 

(c2) $I \in \Int(\bM;v,k)$ 
for all $v \in W_{K}$ and $k \in K$. 

\noindent
Then, since $\bM \in \bz_{\BZ}(I,K)$ by condition (c2), 
we see from Lemma~\ref{lem:IK-fj}\,(3) that 
\begin{equation*}
(f_{L_{r}}f_{L_{t}}^{2}f_{L_{r}}\bM)_{I}=
f_{L_{r}}f_{L_{t}}^{2}f_{L_{r}}\bM_{I}
\quad \text{and} \quad
(f_{L_{t}}f_{L_{r}}^{2}f_{L_{t}}\bM)_{I}=
f_{L_{t}}f_{L_{r}}^{2}f_{L_{t}}\bM_{I}, 
\end{equation*}
and hence, by condition (c1), that 
\begin{equation*}
(f_{L_{r}}f_{L_{t}}^{2}f_{L_{r}}\bM)_{\gamma}=
(f_{L_{r}}f_{L_{t}}^{2}f_{L_{r}}\bM_{I})_{\gamma}
\quad \text{and} \quad
(f_{L_{t}}f_{L_{r}}^{2}f_{L_{t}}\bM)_{\gamma}=
(f_{L_{t}}f_{L_{r}}^{2}f_{L_{t}}\bM_{I})_{\gamma}.
\end{equation*}
Thus, in order to show that 
$(\ha{f}_{r}\ha{f}_{t}^{2}\ha{f}_{r}\bM)_{\gamma}=
 (\ha{f}_{t}\ha{f}_{r}^{2}\ha{f}_{t}\bM)_{\gamma}$, 
it suffices to show that 
%
%%%%%%%%%%%%%%%
%%% eq:s2-1 %%%
%%%%%%%%%%%%%%%
%
\begin{equation} \label{eq:s2-1}
f_{L_{r}}f_{L_{t}}^{2}f_{L_{r}}\bM_{I} = 
f_{L_{t}}f_{L_{r}}^{2}f_{L_{t}}\bM_{I}.
\end{equation}

We now define 
\begin{align*}
& X_{b}:=
(F_{r_{b}}F_{t_{b}}^{2}F_{r_{b}}) \cdots 
(F_{r_{2}}F_{t_{2}}^{2}F_{r_{2}})
(F_{r_{1}}F_{t_{1}}^{2}F_{r_{1}}), \\
& Y_{b}:=
(F_{t_{b}}F_{r_{b}}^{2}F_{t_{b}}) \cdots 
(F_{t_{2}}F_{r_{2}}^{2}F_{t_{2}})
(F_{t_{1}}F_{r_{1}}^{2}F_{t_{1}}), 
\end{align*}
for $0 \le b \le a$; 
$X_{0}$ and $Y_{0}$ are understood to be 
the identity map on $\bz_{I}(\lambda_{I})$. 
We will show by induction on $b$ that 
$X_{b}\bM_{I} \ne \bzero$, 
$Y_{b}\bM_{I} \ne \bzero$, and 
$X_{b}\bM_{I}=Y_{b}\bM_{I}$ for all $0 \le b \le a$. 
If $b=0$, then there is nothing to prove. 
Assume, therefore, that $b > 0$. 
Note that $\bM_{I} \in \bz_{I}(\lambda_{I})$
(see the comment preceding Claim~\ref{c:main2-2}). 
Hence, $X_{b-1}\bM_{I} \in \bz_{I}(\lambda_{I})$ 
since $X_{b-1}\bM_{I} \ne \bzero$ by the induction hypothesis. 
Because $\bz_{I}(\lambda_{I}) \cong \CB_{I}(\lambda_{I})$ as 
crystals for $U_{q}(\Fg^{\vee}_{I})$ by Theorem~\ref{thm:blambda}, 
we have
\begin{equation*}
\Phi_{r_{b}}(X_{b-1}\bM_{I})=
 \max \bigl\{N \ge 0 \mid 
 F_{r_{b}}^{N}X_{b-1}\bM_{I} \ne \bzero
 \bigr\}.
\end{equation*}
Here, observe that $F_{r_{b}}X_{b-1}=X_{b-1}F_{r_{b}}$ 
by the definition of $X_{b-1}$ since for $1 \le c \le b-1$, 
%
%%%%%%%%%%%%%%%%%
%%% eq:dif-rt %%%
%%%%%%%%%%%%%%%%%
%
\begin{equation} \label{eq:dif-rt}
|r_{b}-r_{c}| \ge \ell+1, \quad \text{and} \quad
|r_{b}-t_{c}| \ge 
\underbrace{|r_{b}-r_{c}|}_{\ge \ell+1}-
\underbrace{|r_{c}-t_{c}|}_{< \ell}
 > (\ell+1)-\ell=1.
\end{equation}
As a result, we have
\begin{equation*}
\max \bigl\{
  N \ge 0 \mid 
  F_{r_{b}}^{N}X_{b-1}\bM_{I} \ne \bzero
 \bigr\}
=
\max \bigl\{
  N \ge 0 \mid 
  F_{r_{b}}^{N}\bM_{I} \ne \bzero
 \bigr\}
=\Phi_{r_{b}}(\bM_{I}),
\end{equation*}
and hence $\Phi_{r_{b}}(X_{b-1}\bM_{I}) = 
\Phi_{r_{b}}(\bM_{I})$.
Recall that for each $1 \le b \le a$, 
the integers $r_{b}$ and $t_{b}$ are such that 
$\ol{r_{b}}=p$, $\ol{t_{b}}=q$, and 
$0 < |r_{b}-t_{b}| < \ell$, and that 
the interval $I$ satisfies conditions (a1)'--(a4)', 
which are just conditions (a1)--(a4) of 
Claim~\ref{c:main2-2}, with $r$ and $t$ 
replaced by $r_{b}$ and $t_{b}$, respectively. 
Consequently, it follows from Claim~\ref{c:main2-2}\,(i) that 
$\Phi_{r_{b}}(\bM_{I}) = 
 \ha{\Phi}_{p}(\bM) > 0$, and hence 
$\Phi_{r_{b}}(X_{b-1}\bM_{I}) = \Phi_{r_{b}}(\bM_{I}) = 
 \ha{\Phi}_{p}(\bM) > 0$.
Similarly, we can show that 
$\Phi_{t_{b}}(X_{b-1}\bM_{I})=
 \Phi_{t_{b}}(\bM_{I})=
 \ha{\Phi}_{q}(\bM) > 0$. 
Moreover, since $F_{t_{b}}X_{b-1}=X_{b-1}F_{t_{b}}$ and 
$F_{r_{b}}X_{b-1}=X_{b-1}F_{r_{b}}$, we have 
\begin{align*}
\Phi_{r_{b}}(F_{t_{b}}X_{b-1}\bM_{I}) 
& = \max \bigl\{
  N \ge 0 \mid 
  F_{r_{b}}^{N}F_{t_{b}}X_{b-1}\bM_{I} \ne \bzero
 \bigr\} \\
& = \max \bigl\{
  N \ge 0 \mid 
  F_{r_{b}}^{N}F_{t_{b}}\bM_{I} \ne \bzero
 \bigr\}
\quad \text{} \\
& = \Phi_{r_{b}}(F_{t_{b}}\bM_{I}).
\end{align*}
Also, it follows from Claim~\ref{c:main2-2}\,(i) that 
$\Phi_{r_{b}}(F_{t_{b}}\bM_{I})=\Phi_{r_{b}}(\bM_{I})+1$; 
note that $\Phi_{r_{b}}(\bM_{I})=\Phi_{r_{b}}(X_{b-1}\bM_{I})$ 
as shown above. Combining these, we get 
$\Phi_{r_{b}}(F_{t_{b}}X_{b-1}\bM_{I})=
 \Phi_{r_{b}}(X_{b-1}\bM_{I})+1$.
Similarly, we have $\Phi_{t_{b}}(F_{r_{b}}X_{b-1}\bM_{I})=
 \Phi_{t_{b}}(X_{b-1}\bM_{I})+1$. 
Here we remark that the crystal 
$\bz_{I}(\lambda_{I}) \cong \CB_{I}(\lambda_{I})$ 
satisfies condition (C5) of Theorem~\ref{thm:stem}. 
Therefore, we obtain 
\begin{equation*}
X_{b}\bM_{I}=F_{r_{b}}F_{t_{b}}^{2}F_{r_{b}}X_{b-1}\bM_{I} \ne \bzero
\quad \text{and} \quad
F_{t_{b}}F_{r_{b}}^{2}F_{t_{b}}X_{b-1}\bM_{I} \ne \bzero, 
\end{equation*}
and
\begin{equation*}
\bzero \ne X_{b}\bM_{I}=
F_{r_{b}}F_{t_{b}}^{2}F_{r_{b}}X_{b-1}\bM_{I}=
F_{t_{b}}F_{r_{b}}^{2}F_{t_{b}}X_{b-1}\bM_{I}.
\end{equation*}
Also, since $X_{b-1}\bM_{I}=Y_{b-1}\bM_{I}$ 
by the induction hypothesis, we obtain 
\begin{equation*}
Y_{b}\bM_{I}=
F_{t_{b}}F_{r_{b}}^{2}F_{t_{b}}Y_{b-1}\bM_{I}=
F_{t_{b}}F_{r_{b}}^{2}F_{t_{b}}X_{b-1}\bM_{I} \ne \bzero, 
\end{equation*}
and 
\begin{equation*}
X_{b}\bM_{I}=
F_{t_{b}}F_{r_{b}}^{2}F_{t_{b}}X_{b-1}\bM_{I}=
F_{t_{b}}F_{r_{b}}^{2}F_{t_{b}}Y_{b-1}\bM_{I}=
Y_{b}\bM_{I}. 
\end{equation*}
Thus, we have shown that 
$X_{b}\bM_{I} \ne \bzero$, 
$Y_{b}\bM_{I} \ne \bzero$, and 
$X_{b}\bM_{I}=Y_{b}\bM_{I}$ for all $0 \le b \le a$, 
as desired.

Since $X_{a}\bM_{I} \ne \bzero$, we have 
\begin{align*}
X_{a}\bM_{I} & = 
(F_{r_{a}}F_{t_{a}}^{2}F_{r_{a}}) \cdots 
(F_{r_{2}}F_{t_{2}}^{2}F_{r_{2}})
(F_{r_{1}}F_{t_{1}}^{2}F_{r_{1}})\bM_{I} \\
& = 
(f_{r_{a}}f_{t_{a}}^{2}f_{r_{a}}) \cdots 
(f_{r_{2}}f_{t_{2}}^{2}f_{r_{2}})
(f_{r_{1}}f_{t_{1}}^{2}f_{r_{1}})\bM_{I} \\
& = f_{L_{r}}f_{L_{t}}^{2}f_{L_{r}}\bM_{I} 
\quad \text{by Theorem~\ref{thm:binf}};
\end{align*}
on the crystal $\bz_{I} \cong \CB_{I}(\infty)$, we have 
$f_{r_{b}}f_{r_{c}}=f_{r_{c}}f_{r_{b}}$ and 
$f_{t_{b}}f_{t_{c}}=f_{t_{c}}f_{t_{b}}$ for all 
$1 \le b,\,c \le a$, and 
$f_{r_{b}}f_{t_{c}}=f_{t_{c}}f_{r_{b}}$ for all 
$1 \le b,\,c \le a$ with $b \ne c$ (see \eqref{eq:dif-rt}). 
Similarly, we can show that 
$Y_{a}\bM_{I} = f_{L_{t}}f_{L_{r}}^{2}f_{L_{t}}\bM_{I}$. 
Since $X_{a}\bM_{I}=Y_{a}\bM_{I}$ as shown above, 
we obtain \eqref{eq:s2-1}, and hence \eqref{eq:step2}. 

Finally, we show \eqref{eq:step3}; 
we give a proof only for the first equation, since 
the proof of the second one is similar. 
Define $r,\,t \in \BZ$ as in \eqref{eq:def-rt}; 
note that $\ha{a}_{pq}=a_{rt}$ and 
$\ha{a}_{qp}=a_{tr}$ by the definitions. 
Let $K$ be an interval in $\BZ$ 
such that $r,\,t \in K$, and 
take an interval $I$ in $\BZ$ satisfying 
conditions (a1)--(a4) in Claim~\ref{c:main2-2}, 
conditions (b1), (b2) in the proof of \eqref{eq:step1} 
with $\bM''=\ha{F}_{q}^{2}\ha{F}_{p}\bM$ and $r$ replaced by 
$\ha{F}_{p}^{2}\ha{F}_{q}\bM$ and $t$, respectively, and the following: 

(d) $I \in \Int(\bM;e,t-1) \cap \Int(\bM;e,t) \cap \Int(\bM;e,t+1)$. 

\noindent
Then, we see from the proof of Claim~\ref{c:main2-2}\,(i) that 
$\ha{\Phi}_{q}(\ha{F}_{p}\bM)=\Phi_{t}(F_{r}\bM_{I})$. Therefore, 
%
%%%%%%%%%%%%%%%
%%% eq:s3-1 %%%
%%%%%%%%%%%%%%%
%
\begin{align} 
\ha{\ve}_{q}(\ha{F}_{p}\bM) & = 
\ha{\Phi}_{q}(\ha{F}_{p}\bM) - 
 \pair{\Wt(\ha{F}_{p}\bM)}{\ha{\alpha}_{q}} \nonumber \\
& =
\Phi_{t}(F_{r}\bM_{I}) - 
\pair{\Wt(\bM)-\ha{h}_{p}}{\ha{\alpha}_{q}} \nonumber \\
& = 
\Phi_{t}(F_{r}\bM_{I}) - 
\pair{\ha{\lambda}+\wt(\bM)-\ha{h}_{p}}{\ha{\alpha}_{q}}. 
\label{eq:s3-1}
\end{align}
Let us compute the value 
$\pair{\wt(\bM)}{\ha{\alpha}_{q}}$. 
We deduce from the definition \eqref{eq:def-wt} of $\wt(\bM)$ 
along with Remark~\ref{rem:sig-wlam2} that 
$\pair{\wt(\bM)}{\ha{\alpha}_{q}}=
-M_{\Lambda_{q-1}}+2M_{\Lambda_{q}}-M_{\Lambda_{q+1}}$. 
Also, 
\begin{align*}
& -M_{\Lambda_{q-1}}+2M_{\Lambda_{q}}-M_{\Lambda_{q+1}}
  = -M_{\Lambda_{t-1}}+2M_{\Lambda_{t}}-M_{\Lambda_{t+1}}
  \quad \text{by Remark~\ref{rem:sig-wlam2}} \\ 
& \qquad 
  = -M_{\vpi_{t-1}^{I}}+2M_{\vpi_{t}^{I}}-M_{\vpi_{t+1}^{I}} 
  = \pair{\wt(\bM_{I})}{\alpha_{t}} \quad \text{by condition (d)}.
\end{align*}
Hence we obtain 
$\pair{\wt(\bM)}{\ha{\alpha}_{q}}=
\pair{\wt(\bM_{I})}{\alpha_{t}}$. 
In addition, note that 
$\pair{\ha{\lambda}}{\ha{\alpha}_{q}}=
 \pair{\lambda_{I}}{\alpha_{t}}$ 
by the definition \eqref{eq:lamI} of $\lambda_{I} \in \Fh_{I}$, 
and that $\pair{\ha{h}_{p}}{\ha{\alpha}_{q}}=
\ha{a}_{pq}=a_{rt}=\pair{h_{r}}{\alpha_{t}}$. 
Substituting these equations into \eqref{eq:s3-1}, 
we see that  
\begin{align*} 
\ha{\ve}_{q}(\ha{F}_{p}\bM) & =
\Phi_{t}(F_{r}\bM_{I}) - 
\pair{\lambda_{I}+\wt(\bM_{I})-h_{r}}{\alpha_{t}} \\
& =
\Phi_{t}(F_{r}\bM_{I}) - 
\pair{\Wt(\bM_{I})-h_{r}}{\alpha_{t}} \\
& = 
\Phi_{t}(F_{r}\bM_{I}) - 
\pair{\Wt(F_{r}\bM_{I})}{\alpha_{t}} 
= \ve_{t}(F_{r}\bM_{I}).
\end{align*}
Now, the same argument as in the proof of \eqref{eq:step1}
yields $\ha{\Phi}_{q}(\ha{F}_{p}^{2}\ha{F}_{q}\bM)=
\Phi_{t}(F_{r}^{2}F_{t}\bM_{I})$. Using this, we can show 
in exactly the same way as above that 
$\ha{\ve}_{q}(\ha{F}_{p}^{2}\ha{F}_{q}\bM)=
 \ve_{t}(F_{r}^{2}F_{t}\bM_{I})$. 
Since we know from Claim~\ref{c:main2-2}\,(ii) that 
$\ve_{t}(F_{r}\bM_{I})=\ve_{t}(F_{r}^{2}F_{t}\bM_{I})$, 
we conclude that 
$\ha{\ve}_{q}(\ha{F}_{p}\bM)=
 \ha{\ve}_{q}(\ha{F}_{p}^{2}\ha{F}_{q}\bM)$, as desired. 
Thus we have shown \eqref{eq:step3}. 
This completes the proof of the theorem. 
\end{proof}

\begin{proof}[Proof of Theorem~\ref{thm:main}]
Recall from \cite[\S8.1]{Kas} that the crystal basis $\ha{\CB}(\infty)$ 
can be regarded as the ``direct limit'' of 
$\ha{\CB}(\ha{\lambda})$'s as $\ha{\lambda} \in \ha{\Fh}$ tends to infinity, 
i.e., as $\pair{\ha{\lambda}}{\ha{\alpha}_{i}} \to +\infty$ 
for all $i \in \ha{I}$. 
Also, by using \eqref{eq:blam-aff}, we can verify that 
the direct limit of $\bz_{\BZ}^{\sigma}(\bO; \ha{\lambda})$'s 
(as $\ha{\lambda} \in \ha{\Fh}$ tends to infinity) is nothing but 
$\bz_{\BZ}^{\sigma}(\bO)$. 
Consequently, the crystal $\bz_{\BZ}^{\sigma}(\bO)$ 
is isomorphic to the crystal basis $\ha{\CB}(\infty)$. 
This proves Theorem~\ref{thm:main}.
\end{proof}

%========================%
%     START SECTION A    %
%========================%
%
\appendix
\section{Appendix.}
\label{appendix}

%============================%
%     START SUBSECTION A1    %
%============================%
%
\subsection{Characterization of some crystal bases in the simply-laced case.}
\label{subsec:stemb}

In this appendix, 
let $A=(a_{ij})_{i,\,j \in I}$ be 
a generalized Cartan matrix indexed by a finite set $I$ 
such that $a_{ij} \in \bigl\{0,\,-1\bigr\}$ 
for all $i,\,j \in I$ with $i \ne j$. 
Let $\Fg$ be the (simply-laced) 
Kac-Moody algebra over $\BC$ 
associated to this generalized Cartan matrix $A$, 
with Cartan subalgebra $\Fh$, and 
simple coroots $h_{i}$, $i \in I$. 
Let $U_{q}(\Fg)$ be the quantized universal enveloping algebra 
over $\BC(q)$ associated to $\Fg$. 
For a dominant integral weight 
$\lambda \in \Fh^{\ast}:=\Hom_{\BC}(\Fh,\BC)$ for $\Fg$, 
let $\CB(\lambda)$ denote the crystal basis of the 
irreducible highest weight $U_{q}(\Fg)$-module of highest 
weight $\lambda$. 

Let $\CB$ be a crystal for $U_{q}(\Fg)$, equipped with 
the maps $\wt$, $e_{p},\,f_{p} \ (p \in I)$, and 
$\ve_{p},\,\vp_{p} \ (p \in I)$. 
We assume that $\CB$ is semiregular in the sense of 
\cite[p.86]{HK}; namely, for $x \in \CB$ and $p \in I$, 
\begin{align*}
& \ve_{p}(x)=
  \max \bigl\{N \ge 0 \mid e_{p}^{N}x \ne \bzero\bigr\} 
  \in \BZ_{\ge 0}, \\
& \vp_{p}(x)=
  \max \bigl\{N \ge 0 \mid f_{p}^{N}x \ne \bzero\bigr\} 
  \in \BZ_{\ge 0},
\end{align*}
where $\bzero$ is an additional element, 
which is not contained in $\CB$. 
Let $X$ denote the crystal graph of the crystal $\CB$. 
We further assume that the crystal graph $X$ is connected. 
The following theorem is a restatement of results 
in \cite{Stem}. 
%
%%%%%%%%%%%%%%%%
%%% thm:stem %%%
%%%%%%%%%%%%%%%%
%
\begin{thm} \label{thm:stem}
Keep the setting above. 
Let $\lambda \in \Fh^{\ast}$ be a dominant integral weight for $\Fg$. 
The crystal $\CB$ is isomorphic, as a crystal for $U_{q}(\Fg)$, 
to the crystal basis $\CB(\lambda)$ if and only if 
$\CB$ satisfies the following conditions {\rm (C1)--(C6)}{\rm:}

{\rm (C1)} If $x \in \CB$ and $p,\,q \in I$ are such that 
$p \ne q$ and $e_{p}x \ne \bzero$, then 
$\ve_{q}(x) \le \ve_{q}(e_{p}x)$ and 
$\vp_{q}(e_{p}x) \le \vp_{q}(x)$. 

{\rm (C2)} Let $x \in \CB$, and $p,\,q \in I$. 
Assume that $e_{p}x \ne \bzero$ and $e_{q}x \ne \bzero$, 
and that $\ve_{q}(e_{p}x)=\ve_{q}(x)$. 
Then, $e_{p}e_{q}x \ne \bzero$, $e_{q}e_{p}x \ne \bzero$, 
and $e_{p}e_{q}x=e_{q}e_{p}x$. 

{\rm (C3)} Let $x \in \CB$, and $p,\,q \in I$. 
Assume that $e_{p}x \ne \bzero$ and $e_{q}x \ne \bzero$, 
and that $\ve_{q}(e_{p}x)=\ve_{q}(x)+1$ and 
$\ve_{p}(e_{q}x)=\ve_{p}(x)+1$. 
Then, $e_{p}e_{q}^{2}e_{p}x \ne \bzero$, 
$e_{q}e_{p}^{2}e_{q}x \ne \bzero$, and 
$e_{p}e_{q}^{2}e_{p}x=e_{q}e_{p}^{2}e_{q}x$. 
Moreover, 
$\vp_{q}(e_{p}x)=\vp_{q}(e_{p}^{2}e_{q}x)$ and 
$\vp_{p}(e_{q}x)=\vp_{p}(e_{q}^{2}e_{p}x)$. 

{\rm (C4)} Let $x \in \CB$, and $p,\,q \in I$. 
Assume that $f_{p}x \ne \bzero$ and $f_{q}x \ne \bzero$, 
and that $\ve_{q}(f_{p}x)=\ve_{q}(x)$. 
Then, $f_{p}f_{q}x \ne \bzero$, $f_{q}f_{p}x \ne \bzero$, 
and $f_{p}f_{q}x=f_{q}f_{p}x$. 

{\rm (C5)} Let $x \in \CB$, and $p,\,q \in I$. 
Assume that $f_{p}x \ne \bzero$ and $f_{q}x \ne \bzero$, 
and that $\vp_{q}(f_{p}x)=\vp_{q}(x)+1$ and 
$\vp_{p}(f_{q}x)=\vp_{p}(x)+1$. 
Then, $f_{p}f_{q}^{2}f_{p}x \ne \bzero$, 
$f_{q}f_{p}^{2}f_{q}x \ne \bzero$, and 
$f_{p}f_{q}^{2}f_{p}x=f_{q}f_{p}^{2}f_{q}x$. 
Moreover, 
$\ve_{q}(f_{p}x)=\ve_{q}(f_{p}^{2}f_{q}x)$ and 
$\ve_{p}(f_{q}x)=\ve_{p}(f_{q}^{2}f_{p}x)$. 

{\rm (C6)} 
There exists an element $x_{0} \in \CB$ of weight $\lambda$ 
such that $e_{p}x_{0}=\bzero$ and 
$\vp_{p}(x_{0})=\pair{h_{p}}{\lambda}$ for all $p \in I$. 
\end{thm}

\begin{proof}[(Sketch of) Proof]
First we prove the ``if'' part. 
Recall that the crystal graph $X$ of the crystal $\CB$ is 
an $I$-colored directed graph. 
We will show that $X$ is $A$-regular 
in the sense of \cite[Definition~1.1]{Stem}. 
It is obvious that $X$ satisfies 
condition (P1) on page 4809 of \cite{Stem}
since $\CB$ is assumed to be semiregular. 
Also, it follows immediately from the axioms of a crystal 
that $X$ satisfies condition (P2) on page 4809 of \cite{Stem}. 
Now we note that for $x \in \CB$ and $p \in I$, 
$\ve(x,p)$ (resp., $\delta(x,p)$) 
in the notation of \cite{Stem} agrees with 
$\vp_{p}(x)$ (resp., $-\ve_{p}(x)$) in our notation. Hence, 
for $x \in \CB$ and $p,\,q \in I$ with $e_{p}x \ne \bzero$, 
$\Delta_{p}\delta(x,q)$ (resp., $\Delta_{p}\ve(x,q)$) in 
the notation of \cite{Stem} agrees with 
$-\ve_{q}(e_{p}x)+\ve_{q}(x)$ 
(resp., $\vp_{q}(e_{p}x)-\vp_{q}(x)$) in our notation. 
Hence, in our notation, we can rewrite 
condition (P3) on page 4809 of \cite{Stem} as: 
$\bigl\{-\ve_{q}(e_{p}x)+\ve_{q}(x)\bigr\}+
\bigl\{\vp_{q}(e_{p}x)-\vp_{q}(x)\bigr\}=a_{pq}$ 
for $x \in \CB$ and $p,\,q \in I$ such that 
$p \ne q$ and $e_{p}x \ne \bzero$. 
From the axioms of a crystal, we have
\begin{align*}
\vp_{q}(e_{p}x)-\ve_{q}(e_{p}x) & =
\pair{h_{q}}{\wt(e_{p}x)}=
\pair{h_{q}}{\alpha_{p}}+\pair{h_{q}}{\wt x} \\
& = a_{qp}+\pair{h_{q}}{\wt x}, \\
\vp_{q}(x)-\ve_{q}(x) & =\pair{h_{q}}{\wt x}.
\end{align*}
Thus, condition (P3) on page 4809 of \cite{Stem} 
holds for the crystal graph $X$. 
Similarly, in our notation, we can rewrite 
condition (P4) on page 4809 of \cite{Stem} as: 
$-\ve_{q}(e_{p}x)+\ve_{q}(x) \le 0$ and 
$\vp_{q}(e_{p}x)-\vp_{q}(x) \le 0$ 
for $x \in \CB$ and $p,\,q \in I$ such that 
$p \ne q$ and $e_{p}x \ne \bzero$, 
which is equivalent to condition (C1). 
In addition, note that for $x \in \CB$ and $p,\,q \in I$ 
with $f_{p}x \ne \bzero$, $\nabla_{p}\delta(x,q)$ 
(resp., $\nabla_{p}\ve(x,q)$) in the notation of \cite{Stem} 
agrees with $-\ve_{q}(x)+\ve_{q}(f_{p}x)$ 
(resp., $\vp_{q}(x)-\vp_{q}(f_{p}x)$) 
in our notation. In is easy to check that 
conditions (P5) and (P6) on page 4809 of \cite{Stem} 
are equivalent to conditions (C2) and (C3), respectively. 
Similarly, it is easily seen that 
conditions (P5') and (P6') on page 4809 of \cite{Stem} 
are equivalent to conditions (C4) and (C5), 
respectively. Thus, we have shown that 
the crystal graph $X$ is $A$-regular. 

We know from \cite[\S3]{Stem} that 
the crystal graph of the crystal basis $\CB(\lambda)$ 
is $A$-regular. Also, it is obvious that 
the highest weight element $u_{\lambda}$ of $\CB(\lambda)$ 
satisfies the condition that $e_{p}u_{\lambda}=\bzero$ and 
$\vp_{p}(u_{\lambda})=\pair{h_{p}}{\lambda}$ for all $p \in I$ 
(cf. condition (C6)). 
Therefore, we conclude from \cite[Proposition~1.4]{Stem} 
that the crystal graph $X$ of the crystal $\CB$ is 
isomorphic, as an $I$-colored directed graph, to 
the crystal graph of the crystal basis $\CB(\lambda)$; 
note that $x_{0} \in \CB$ corresponds to 
$u_{\lambda} \in \CB(\lambda)$ under this isomorphism. 
Since the crystal graphs of $\CB$ and $\CB(\lambda)$ are both connected, and 
since $x_{0} \in \CB$ and $u_{\lambda} \in \CB(\lambda)$ are both of weight $\lambda$, 
it follows that the crystal $\CB$ is isomorphic to 
the crystal basis $\CB(\lambda)$. 
This proves the ``if'' part. 

The ``only if'' part is now clear from the argument above. 
Thus we have proved the theorem. 
\end{proof}

%======================%
%     BIBLIOGRAPHY     %
%======================%

{\small
\setlength{\baselineskip}{13pt}
\renewcommand{\refname}{References}

}


\begin{thebibliography}{XXXX}

\bibitem[A]{A}
J. E. Anderson, 
A polytope calculus for semisimple groups, 
{\it Duke Math. J.} {\bf 116} (2003), 567--588.

\bibitem[BjB]{BjB}
A. Bj\"{o}rner and F. Brenti, 
``Combinatorics of Coxeter Groups'', 
Graduate Texts in Mathematics Vol.~231, 
Springer, New York, 2005.

\bibitem[BF1]{BF1}
A. Braverman and M. Finkelberg, 
Pursuing the double affine Grassmannian I: 
Transversal slices via instantons on $A_{k}$-singularities, 
{\it Duke Math. J.} {\bf 152} (2010), 175--206. 

\bibitem[BF2]{BF2}
A. Braverman and M. Finkelberg, 
Pursuing the double affine Grassmannian I\!I: 
Convolution, preprint, arXiv:0908.3390.

\bibitem[HK]{HK}
J. Hong and S.-J. Kang, 
``Introduction to quantum groups and crystal bases'', 
Graduate Studies in Mathematics Vol.~42, 
Amer. Math. Soc., Providence, RI, 2002.

\bibitem[Kam1]{Kam1}
J. Kamnitzer, 
Mirkovi\'{c}-Vilonen cycles and polytopes, 
{\it Ann. of Math. (2)} {\bf 171} (2010), 731--777. 

\bibitem[Kam2]{Kam2}
J. Kamnitzer, 
The crystal structure on the set of 
Mirkovi\'{c}-Vilonen polytopes, 
{\it Adv. Math.} {\bf 215} (2007), 66--93. 

\bibitem[Kas]{Kas}
M. Kashiwara, 
On crystal bases, 
{\it in} ``Representations of Groups'' 
(B.N. Allison and G.H. Cliff, Eds.), 
CMS Conf. Proc. Vol.~16, pp.~155--197, Amer. Math. Soc., 
Providence, RI, 1995.

\bibitem[N]{Nak}
H. Nakajima, Quiver varieties and branching, 
{\it SIGMA Symmetry Integrability Geom. Methods Appl.} {\bf 5} (2009), 
Paper 003, 37 pages. 

\bibitem[NSS]{NSS}
S. Naito, D. Sagaki, and Y. Saito, 
Toward Berenstein-Zelevinsky data in affine type $A$, 
I\!I: Explicit description, in preparation. 

\bibitem[S]{Stem}
J. Stembridge, A local characterization of simply-laced crystals, 
{\it Trans. Amer. Math. Soc.} {\bf 355} (2003), 4807--4823. 

\end{thebibliography}
\end{document}